\documentclass[twoside,11pt]{article}

\usepackage[preprint]{jmlr2e}

\usepackage{bm}
\usepackage{tikz,mathpazo}
\usepackage{pgf}
\usepackage{indentfirst}
\usepackage{graphicx}
\usepackage{graphics}
\usepackage{bm}
\usepackage{bbm}
\usepackage{listings}  
\usepackage{amsmath}
\usepackage{setspace} 
\usepackage{indentfirst}
\usepackage{caption}
\usepackage{multirow} 
\usepackage{lipsum,multicol}
\usepackage{pdfpages}
\usepackage{float}
\usepackage{pdfpages}
\usepackage{url}
\usepackage{colortbl}
\usepackage{subfigure}
\usepackage{epsfig}
\usepackage{epstopdf}
\usepackage{fancyhdr}
\usepackage{abstract}
\usepackage{tikz,mathpazo}
\usepackage{amssymb}
\usepackage{latexsym}
\usepackage{verbatim}
\usepackage{bbm}
\allowdisplaybreaks[4]

\newcommand{\inner}[1]{\left\langle #1 \right\rangle}
\newcommand{\norm}[1]{\left\Vert #1\right\Vert}
\newcommand{\bb}[1]{\mathbb{#1}}

\newcommand{\conv}[0]{\mathrm{conv}\,}

\newcommand{\ca}[1]{\mathcal{#1}}

\newcommand{\A}{\ca{A}}

\newcommand{\xk}{{x_{k} }}

\newcommand{\xkp}{{x_{k+1} }}

\newcommand{\etak}{{\eta_{k} }}

\newcommand{\zk}{{z_{k} }}
\newcommand{\zkp}{{z_{k+1} }}

\newcommand{\mk}{{m_{k} }}
\newcommand{\vk}{{v_{k} }}
\newcommand{\mkp}{{m_{k+1} }}
\newcommand{\vkp}{{v_{k+1} }}

\newcommand{\D}{\ca{D}}

\newcommand{\Rn}{\mathbb{R}^n}

\usepackage{pifont}
\newcommand{\cmark}{\ding{51}}%
\newcommand{\xmark}{\ding{55}}%

\newtheorem{theo}{Theorem}
\newtheorem{lem}{Lemma}
\newtheorem{prop}{Proposition}

\newtheorem{coro}{Corollary}
\newtheorem{defin}{Definition}
\newtheorem{rmk}{Remark}
\newtheorem{assumpt}{Assumption}

\usepackage{caption}
\usepackage{algorithm}
\usepackage{algpseudocode}

\usepackage{lastpage}
\jmlrheading{25}{2024}{1-\pageref{LastPage}}{5/23; Revised 2/24}{2/24}{23-0576}{Nachuan Xiao, Xiaoyin Hu, Xin Liu, and Kim-Chuan Toh}

\ShortHeadings{Adam-family Methods for Nonsmooth Optimization}{XIAO, HU, LIU, and TOH}
\firstpageno{1}

\begin{document}
\title{Adam-family Methods for Nonsmooth Optimization with Convergence Guarantees}

\author{\name Nachuan Xiao \email xnc@lsec.cc.ac.cn \\
	\addr Institute of Operations Research and Analytics\\
	National University of Singapore\\
	3 Research Link, Singapore, 117602
	\AND
	\name Xiaoyin Hu\thanks{Corresponding author} \email hxy@amss.ac.cn \\
	\addr School of Computer and Computing Science\\
	Hangzhou City University\\
	Hangzhou, China, 310015
	\AND 
	\name Xin Liu \email liuxin@lsec.cc.ac.cn\\
	\addr State Key Laboratory of Scientific and Engineering Computing\\
	Academy of Mathematics and Systems Science, Chinese Academy of Sciences\\
	Beijing, China, 100190
	\AND
	\name Kim-Chuan Toh \email mattohkc@nus.edu.sg\\
	\addr Department of Mathematics and Institute of Operations Research and Analytics\\
	National University of Singapore\\
	10 Lower Kent Ridge Road, Singapore, 119076
}

\editor{Krishnakumar Balasubramanian}

\maketitle

\begin{abstract}
In this paper, we present a comprehensive study on the convergence properties of Adam-family methods for nonsmooth optimization, especially in the training of nonsmooth neural networks. We introduce a novel two-timescale framework that adopts a two-timescale updating scheme, and prove its convergence properties under mild assumptions. Our proposed framework encompasses various popular Adam-family methods, providing convergence guarantees for these methods in training nonsmooth neural networks. Furthermore, we develop stochastic subgradient methods that incorporate gradient clipping techniques for training nonsmooth neural networks with heavy-tailed noise. Through our framework, we show that our proposed methods converge even when the evaluation noises are only assumed to be integrable. Extensive numerical experiments demonstrate the high efficiency and robustness of our proposed methods. 
\end{abstract}


\section{Introduction}
In this paper, we consider the following unconstrained nonlinear optimization problem
\begin{equation}
	\label{Prob_Ori}
	\tag{UNP}
	\begin{aligned}
		\min_{x \in \Rn}\quad f(x), 
	\end{aligned}
\end{equation}
where $f$ is nonconvex, locally Lipschitz continuous and possibly nonsmooth over $\Rn$. 

The optimization problem in the form of \ref{Prob_Ori} has numerous important applications in machine learning and data science, especially in training deep neural networks. In these applications of \ref{Prob_Ori},  we usually only have access to the stochastic evaluations of the exact gradients of $f$.  The stochastic gradient descent (SGD) is one of the most popular methods for solving \ref{Prob_Ori}, and incorporating the momentum terms to SGD for acceleration is also very popular in practice. In SGD, the updating rule depends on the stepsizes (i.e., learning rates), where all of the coordinates of the variable $x$ are equipped with the same stepsize. Recently,  a variety of accelerated versions for SGD are proposed. In particular, the widely used Adam algorithm \citep{kingma2014adam} is developed based on the adaptive adjustment of the coordinate-wise stepsizes and the incorporation of momentum terms in each iteration.
These enhancements have led to its high efficiency in practice. Motivated by Adam, a number of efficient Adam-family methods are developed, such as  AdaBelief \citep{zhuang2020adabelief}, AMSGrad \citep{reddi2019convergence}, NAdam \citep{dozat2016incorporating}, Yogi \citep{zaheer2018adaptive}, etc.

Towards the convergence properties of these Adam-family methods, \citet{kingma2014adam} shows the convergence properties for Adam with constant stepsize in minimizing a Lipschitz continuously differentiable objective function $f$. Then a great number of existing works are conducted to establish the convergence properties of Adam-family methods, see \citep{de2018convergence,zaheer2018adaptive,zou2019sufficient,barakat2021convergence,guo2021novel,shi2021rmsprop,zhang2022adam,wang2022provable} for more information. Some of these existing works \citep{zou2019sufficient,guo2021novel,shi2021rmsprop,barakat2021convergence,wang2022provable,zhang2022adam} adopt diminishing stepsizes to ensure the almost surely convergence to stationary points of $f$ for Adam, while some other existing works  \citep{de2018convergence,zaheer2018adaptive} fix the stepsize as constant and show that the sequence converges to a neighborhood of the stationary points of $f$. A comparison of the results of these existing works is presented in Table \ref{Table_intro_1}.

Despite extensive studies on Adam-family methods, most existing works focus on the cases where $f$ is differentiable over $\Rn$, as depicted in Table \ref{Table_intro_1}. However, nonsmooth activation functions, including ReLU and leaky ReLU, are very popular choices in building neural networks in practice \citep{ming2018food,fu2020drts,fu2021end,wang2023chromosome}, and most of the existing works (e.g., the works listed in Table \ref{Table_intro_1})
 test their analyzed Adam-family methods on the neural networks built by nonsmooth activation functions. As highlighted in \citep{bolte2021conservative,bianchi2022convergence}, when we build a neural network by nonsmooth blocks, the corresponding loss function is typically nonsmooth and not Clarke regular. Consequently, although numerous existing works establish the convergence properties for Adam and its variants, their results are not applicable to the analysis of these Adam-family methods in training nonsmooth neural networks. This naturally leads us to the following
question:
\begin{quote}
	Do Adam-family methods have any convergence guarantees in minimizing nonsmooth functions under practical settings, especially in training nonsmooth neural networks?
\end{quote}

	\begin{table}
	\tiny
	\centering
	\begin{tabular}{c|ccccc}
		\hline
		& Beyond differentiability & Beyond global Lipschitz & Nesterov mo. & A.s. convergence & Stepsize  \\
		\hline
		Our work                      & \cmark             & \cmark                & \cmark                 & \cmark           & $o(1/\log(k))$             \\
		\citet{dozat2016incorporating} & \xmark             & \xmark                & \cmark                 & \xmark           & No convergence result      \\
		\citet{de2018convergence}      & \xmark             & \xmark                & \xmark                 & \xmark           & Constant                   \\
		\citet{zou2019sufficient}      & \xmark             & \xmark                & \xmark                 & \cmark           & $O(k^{-s})$, $s \in (0,1]$ \\
		\citet{zaheer2018adaptive}     & \xmark             & \xmark                & \xmark                 & \xmark           & Constant                   \\
		\citet{guo2021novel}           & \xmark             & \xmark                & \xmark                 & \cmark           & Square-summable               \\
		\citet{shi2021rmsprop}         & \xmark             & \xmark                & \xmark                 & \cmark           & Square-summable               \\
		\citet{barakat2021convergence} & \xmark             & \cmark               & \xmark                 & \cmark           & Square-summable               \\
		\citet{wang2022provable}       & \xmark             & \cmark                & \xmark                 & \cmark           & Square-summable               \\
		\citet{zhang2022adam}          & \xmark             & \xmark                & \xmark                 & \cmark           & Square-summable           \\
		\citet{chen2022towards}        & \xmark             & \xmark                & \xmark                 & \cmark           & $O(k^{-s})$, $s \in (0,1]$   \\\hline  
	\end{tabular}
	\caption{\small Comparison of existing results on the convergence of Adam.  Here ``mo.'' is the abbreviation for ``momentum''. The term ``A.s. convergence'' refers to whether the sequence converges to stationary points of $f$ rather than to a neighborhood of these stationary points almost surely.   
	}
	\label{Table_intro_1}
\end{table}

\subsection{Challenges from Training Nonsmooth Neural Networks}
In training nonsmooth neural networks, one of the major challenges lies in how to differentiate their loss functions. These functions are typically formulated as compositions of elementary blocks that may not be smooth.
To address this issue, automatic differentiation (AD) algorithms have been widely adopted in various well-known machine learning packages, such as PyTorch, TensorFlow, JAX, MindSpore, and PaddlePaddle. Based on the chain rule, the AD algorithms can efficiently compute the gradients for those functions expressed through the composition of elementary differentiable blocks. However, as the chain rule fails for Clarke subdifferential, when we differentiate a neural network built from nonsmooth blocks by those AD algorithms, the results may not be contained in the Clarke subdifferential of its loss function.  As pointed out in \citet{bolte2021conservative}, most of the existing works ignore this issue. They use AD algorithms in training nonsmooth neural networks, but assume differentiability or weak convexity for the objective functions in their theoretical analysis to bypass these theoretical issues arising from the application of AD algorithms. Based on the chain rule for directional derivatives, some existing works \citep{barton2018computationally} propose specifically designed forward mode AD algorithms for evaluating the elements in lexicographic subdifferential \citep{nesterov2005lexicographic},  which is contained in the Clarke subdifferential. However, as described in \citet{bolte2021nonsmooth}, these approaches have expensive computational costs and require significant modifications to the algorithms in existing machine learning packages, and hence are less applicable to practical scenarios.

To understand how AD algorithms differentiate the loss functions of nonsmooth neural networks, \citet{bolte2021conservative} introduces the concept of the {\it conservative field} as a generalization of Clarke subdifferential for its corresponding {\it potential functions}. The class of potential functions includes semi-algebraic functions, semi-analytic functions, and functions whose graphs are definable in some $o$-minimal structures, hence covering the objective functions in a wide range of real-world applications \citep{davis2020stochastic,bolte2021conservative}. The conservative field preserves the validity of the chain rule for nonsmooth functions, explaining the results generated by AD algorithms from various popular numerical libraries such as PyTorch, TensorFlow, JAX, etc. Based on the concept of conservative field, we can characterize the stationarity and design algorithms for the unconstrained nonsmooth optimization, especially when the objective function is differentiated by AD algorithms.  Interested readers can refer \citet{bolte2021conservative} for more detailed properties of the conservative field.

The theoretical properties of the conservative field enable us to investigate the convergence properties of stochastic subgradient algorithms, especially when applied to train nonsmooth neural networks with AD algorithms. Some existing frameworks \citep{benaim2005stochastic,davis2020stochastic,bolte2021conservative} establish the convergence properties for stochastic subgradient methods by 
analyzing the limiting behaviour of their corresponding differential inclusions, and prove that these methods converge to stationary points of $f$ in the sense of its corresponding conservative field $\D_f$. Based on these frameworks, some recent works  \citep{davis2020stochastic,bolte2021conservative,hu2023improved} prove the convergence properties for SGD and proximal SGD. Moreover, \citet{castera2021inertial} proposes the inertial 
Newton algorithm (INNA), which can be regarded as a variant of SGD with heavy-ball momentum. Additionally, \citet{ruszczynski2020convergence,le2023nonsmooth} show the convergence property of SGD with heavy-ball momentum for nonsmooth nonconvex functions from the Norkin class. For Adam and its variants, some existing works \citep{da2020general,barakat2021convergence,gadat2022asymptotic} established the convergence properties of Adam for Lipschitz smooth $f$ by analyzing the limiting behaviour of its corresponding differential equation. However, their approaches rely on some time-dependent differential equations, 
which are challenging to be extended to nonsmooth cases based on the frameworks in \citep{benaim2005stochastic,davis2020stochastic,bolte2021conservative}. To the best of our knowledge, no existing work addresses the convergence properties of Adam-family methods for nonsmooth optimization.

Furthermore, \citet{bolte2021conservative} demonstrate that the Clarke subdifferential is a subset of the conservative field for any potential functions. For nonsmooth neural networks, the conservative fields associated with AD algorithms may introduce infinitely many spurious stationary points \citep{bolte2021nonsmooth,bianchi2022convergence}. Therefore, when we design stochastic subgradient methods based on the conservative field,  the results in some existing frameworks \citep{benaim2005stochastic,davis2020stochastic,bolte2021conservative} can only ensure the convergence to stationary points in the sense of conservative field. As demonstrated in \citep{bolte2021nonsmooth,bianchi2022convergence}, these results fail to guarantee the convergence to meaningful stationary points of $f$.    To this end, \citet{bianchi2022convergence} establishes that under mild assumptions with randomized initial points and stepsizes, SGD can find Clarke stationary points for nonsmooth neural networks almost surely. However, their analysis is limited to SGD without any momentum term, and how to extend their results to Adam-family methods remains an open question.

\subsection{Challenges from Heavy-tailed Evaluation Noises}

Another challenge for solving \ref{Prob_Ori} lies in the noises when evaluating the stochastic subgradient of the objective function. The evaluation noises in a great number of existing works are assumed to have finite second-order moment or even uniformly bounded, for the sake of convenience when analyzing their theoretical properties. However, in various machine learning tasks, such as classification models \citep{mahoney2019traditional,simsekli2019tail,simsekli2020fractional,camuto2021asymmetric,wan2023implicit} and language models \citep{zhang2019gradient,zhang2020adaptive}, some recent works \citep{simsekli2019tail,zhang2020adaptive} illustrate that the evaluation noises of the stochastic subgradients could be heavy-tailed (i.e., only have bounded $s$-order moment for some $s \in [1,2)$ \citep{zhang2019gradient}).   
As illustrated in \citet{zhang2019gradient}, the heavy-tailed evaluation noises have a higher probability of producing extreme values or outliers when compared to normal distributions, and hence may undermine the performance of SGD in these tasks. Even in finite-sum settings, the frequently occurred extreme values in the evaluation of stochastic subgradients can result in extremely large variance \citep{simsekli2019tail}. These results explain the empirical observations in training neural networks, including the long-standing failure cases of SGD methods in training recurrent neural networks \citep{pascanu2012understanding}, and the superior performance of adaptive methods over SGD methods in training language models \citep{zhang2020adaptive}.

To address the challenges in solving \ref{Prob_Ori} with heavy-tailed evaluation noises, the gradient clipping technique has been developed. Gradient clipping normalizes the stochastic gradient, thus preventing extreme values in evaluating the stochastic gradients that can cause instability or divergence in the optimization algorithms. With the gradient clipping technique, some recent works \citep{zhang2019gradient} show that SGD converges when the evaluation noises are bounded in $L^s$ (i.e., the noises $\{\xi_k\}$ satisfy $\sup_{k\geq 0}\mathbb{E}[||\xi_k||^s]<+\infty$) for some $s \in (1,2)$. Table \ref{Table_intro_2} exhibits the related works on the convergence properties of stochastic subgradient methods with gradient clipping  \citep{zhang2019gradient,gorbunov2020stochastic,zhang2020adaptive,mai2021stability,qian2021understanding,elesedy2023u,reisizadeh2023variance}.  As illustrated in Table \ref{Table_intro_2}, all of these existing works rely on the weak convexity of the objective function $f$, hence they are not applicable for training nonsmooth neural networks.

Moreover, most of these existing works focus on the standard SGD method without the momentum term. Although \citet{zhang2020adaptive,pan2023toward} introduce stochastic Adam-family methods with gradient clipping, they did not provide any convergence guarantee for their proposed method.  Therefore, a significant gap exists between the existing 
theoretical analysis \citep{zhang2019gradient,gorbunov2020stochastic,zhang2020adaptive,mai2021stability,qian2021understanding,elesedy2023u,reisizadeh2023variance} and practical implementations for stochastic subgradient methods with heavy-tailed noise, and how to fill that gap is challenging and remains unexplored. 

\begin{table}[tb]
	\centering
	\tiny
	\begin{tabular}{c|cccccc}
		\hline
		& Beyond differentiability & Assumption on noises  & Adaptive stepsize &  Heavy-ball mo. & Nesterov mo. & Convergence \\ \hline
		Our work                & \cmark   &  $L^{1}$       & \cmark  & \cmark &    \cmark &  \cmark   \\
		\citet{zhang2019gradient}& \xmark   &  $L^{\infty}$  & \xmark  & \xmark &    \xmark &  \cmark   \\
		\citet{gorbunov2020stochastic}& \xmark   &  $L^{2}$  &\xmark  & \xmark &    \xmark &  \cmark   \\
		\citet{zhang2020adaptive}& \xmark   &   $L^{s}$ for $s \in(1,2]$  &\xmark  & \xmark &    \xmark &  \cmark   \\
		\citet{zhang2020adaptive}& \xmark   &  No convergence result  &\cmark  & \xmark &    \xmark &  \xmark   \\
		\citet{mai2021stability}&  $\text{\cmark}^{(a)}$   &  $L^{2}$       & \xmark  & \xmark &    \xmark &  \cmark   \\   
		\citet{qian2021understanding}& \xmark  & $L^{\infty}$ & \xmark  & \xmark &    \xmark &  \cmark   \\
		\citet{elesedy2023u}   & \xmark  & $L^{2}$ & \xmark  & \xmark &    \xmark &  \cmark   \\
		\citet{reisizadeh2023variance}& \xmark  & $L^{2}$ & \xmark  & \xmark &    \xmark &  \cmark   \\ \hline
	\end{tabular}
	\caption{\small Comparison of existing convergence results of stochastic (sub)gradient methods with gradient clipping techniques. Here ``mo.'' is the abbreviation for ``momentum''. (a): The proof techniques in \citet{mai2021stability} relies on weak convexity on $f$ and cannot be applied to non-regular cases. }
	\label{Table_intro_2}
\end{table}

\subsection{Contributions}
In this paper, we aim to establish the convergence properties of Adam-family methods for nonsmooth optimization, especially in the context of training nonsmooth neural networks. To this end, we employ the concept of the conservative field to characterize how the objective function $f$ is differentiated,  and consider the following set-valued mapping  $\ca{G}: \Rn \times \Rn \times \Rn \rightrightarrows \Rn \times \Rn \times \Rn$,
\begin{equation}
	\label{Eq_mapping_G}
	\ca{G}(x, m, v) := 
	 \left\{  \left[\begin{matrix}
		(|v| + \varepsilon )^{-\gamma} \odot \left(   m + \alpha d \right)\\
		\tau_1 m - \tau_1 d\\
		\tau_2 v - \tau_2  u\\
	\end{matrix}\right]  : d \in \D_f(x), ~u \in \ca{U}(x,m, v)
        \right\}.
\end{equation}
Here $\D_f$ refers to the conservative field that characterizes how we differentiate the objective function $f$, and $\alpha, \gamma, \varepsilon, \tau_1, \tau_2$ are hyper-parameters. Moreover, $\odot$ and $(\cdot)^{\gamma}$ refer to the element-wise multiplication and power,  respectively. Furthermore,  $\ca{U}: \Rn\times \Rn \times \Rn \rightrightarrows \Rn$ is a set-valued mapping that determines how the estimator $\vk$ is updated.  
Then we propose the following generalized framework for Adam-family methods (\ref{Eq_framework}), 
\begin{equation}
	\label{Eq_framework}
	\tag{AFM}
	(\xkp, m_{k+1}, v_{k+1}) = (x_k, m_k, v_k) - \eta_k (d_{x, k}, d_{m, k}, d_{v, k}) -  \theta_k (\xi_{x, k}, \xi_{m, k}, \xi_{v, k}).
\end{equation}
In \eqref{Eq_framework}, $(d_{x, k}, d_{m, k}, d_{v, k})$ denotes the updating direction, which is an approximated evaluation for $\ca{G}(\xk, \mk, \vk)$, while $(\xi_{x, k}, \xi_{m, k}, \xi_{v, k})$ refers to the evaluation noise. Moreover,  $\{\eta_k\}$ and $\{\theta_k\}$ are the two-timescale stepsizes for updating directions and evaluation noises respectively,  in the sense that they may satisfy $\eta_k /\theta_k \to 0$ as $k\to\infty.$

We prove that under mild conditions, any cluster point of the sequence ${\xk}$ generated by our proposed framework \eqref{Eq_framework} with stepsizes in the order of $o(1/\log(k))$ is a $\D_f$-stationary point of $f$. 
Furthermore, we establish that under mild conditions with randomly chosen initial points and stepsizes, almost surely, any cluster point of the sequence ${\xk}$ is a Clarke stationary point of $f$, independent of the chosen conservative field $\D_f$.

Based on our proposed framework \eqref{Eq_framework}, we demonstrate that our proposed framework can be employed to analyze the convergence properties for a class of Adam-family methods with diminishing stepsize, including  Adam, AdaBelief, AMSGrad, NAdam, and Yogi. We prove that these Adam-family methods converge to stationary points of $f$ in both senses of conservative field and Clarke subdifferential under mild conditions,  thus providing theoretical guarantees for their performance in training nonsmooth neural networks with AD algorithms. 

Another application of our proposed framework \eqref{Eq_framework} lies in investigating the convergence properties of stochastic subgradient methods that incorporate the gradient clipping technique. We prove that under heavy-tailed evaluation noises that are only assumed to be integrable,  our proposed gradient clipping methods conform to the proposed framework \eqref{Eq_framework}. As a result, the convergence properties of these gradient clipping methods directly follow those established for our proposed framework \eqref{Eq_framework}, under mild conditions. 

Furthermore, we perform extensive numerical experiments to evaluate the performance of our proposed Adam-family methods. By comparing with the implementations of Adam-family methods in PyTorch that utilize fixed stepsize in updating their momentum terms and variance estimators, we demonstrate that our proposed Adam-family methods achieve similar accuracy and training loss. Moreover, when the evaluation noises are heavy-tailed, the numerical examples demonstrate that our proposed Adam-family methods outperform existing approaches in terms of training efficiency and robustness.

\subsection{Organization}
The rest of this paper is organized as follows.  In  Section \ref{Section_2}, we define the notations used throughout the paper and present some essential concepts related to probability theory, nonsmooth analysis and differential inclusion. In Section \ref{Section_3}, we focus on the analysis of the convergence properties of our proposed framework \eqref{Eq_framework}, in both senses of the conservative field and Clarke subdifferential. Section \ref{Section_4} illustrates the application of our framework  \eqref{Eq_framework} in establishing the convergence properties for a class of Adam-family methods, including Adam, AdaBelief, AMSGrad, NAdam, and Yogi, under practical settings with mild conditions. Section \ref{Section_5} demonstrates another application of our framework \eqref{Eq_framework} by illustrating the convergence properties of stochastic subgradient methods with gradient clipping technique under heavy-tailed evaluation noises.  In Section 6, we present the results of our numerical experiments that investigate the performance of our proposed Adam-family methods for training nonsmooth neural networks. Finally, we conclude the paper in the last section.

\section{Preliminary}
\label{Section_2}
\subsection{Basic Notations}
For any vectors $x$ and $y$ in $\Rn$ and $\delta \in \bb{R}$, we denote $x\odot y$, $x^\delta$, $x/y$, $|x|$, $x+\delta$ as the vectors whose $i$-th entries are respectively given by $x_iy_i$, $x_i^{\delta}$, $x_i/y_i$, $|x_i|$ and $x_i + \delta$. Moreover, for any sets $\ca{X}, \ca{Y} \subset \Rn$, we denote $\ca{X}\odot \ca{Y}:= \{x\odot y: x \in \ca{X}, y\in \ca{Y} \}$, $(\ca{X})^p := \{x^p: x\in \ca{X}\}$ and  $|\ca{X}|:= \{|x|: x \in \ca{X}\}$. In addition, for any $z \in \Rn$, we denote $z + \ca{X} := \{z\} + \ca{X}$ and $z \odot \ca{X} := \{z\} \odot\ca{X}$. 

We define the set-valued mappings $\mathrm{sign}:\Rn \rightrightarrows \Rn$  and  $\widetilde{\mathrm{sign}}:\Rn \rightrightarrows \Rn$ as follows: For any $x \in \Rn$,
\begin{equation*}
	\small
	\left(\mathrm{sign}(x)\right)_i = \begin{cases}
		\{-1\} & x_i<0;\\
		[-1, 1] & x_i = 0; \\
		\{1\} & x_i > 0.
	\end{cases}
	,\quad \text{and} \quad 
	\left(\widetilde{\mathrm{sign}}(x)\right)_i = \begin{cases}
		\{-1\} & x_i<0;\\
		\{0\} & x_i = 0; \\
		\{1\} & x_i > 0.
	\end{cases}
\end{equation*}
Then it is easy to verify that $\widetilde{\mathrm{sign}}(x) \odot \mathrm{sign}(x) = (\widetilde{\mathrm{sign}}(x))^2$ holds for any $x \in \Rn$. 

In addition, we denote $\Rn_+:=\{ x\in \Rn: x_i\geq 0 \text{ for any } 1\leq i\leq n \}$. Moreover, $\mu^d$ refers to the Lebesgue measure on $\bb{R}^d$, and when the dimension $d$ is clear from the context, we write the Lebesgue measure as $\mu$ for brevity. Furthermore, we say a measurable set $A$ is zero-measure if $\mu(A) = 0$, and $A$ is full-measure if $\mu(A^c) = 0$.

\subsection{Probability Theory}

In this subsection, we present some essential concepts from probability theory, which are necessary for the proofs in this paper. 
\begin{defin}
	Let $(\Omega, \ca{F}, \mathbb{P})$ be a probability space. We say $\{\ca{F}_k\}_{k \in \bb{N}}$ is a filtration  if  $\{\ca{F}_k\}$ is a collection of $\sigma$-algebras that
	 satisfies $
	\ca{F}_0 \subseteq \ca{F}_1 \subseteq \cdots \subseteq \ca{F}_{\infty} \subseteq \ca{F}$. 
\end{defin}

\begin{defin}
	We say that a stochastic series $\{\xi_k\}$ is a martingale if the following conditions hold.
	\begin{itemize}
		\item The sequence of random vectors $\{\xi_k\}$ is adapted to the filtration $\{ \ca{F}_{k} \}$. That is, for any $k \geq 0$, $\xi_k$ is measurable with respect to the $\sigma$-algebra $\ca{F}_{k}$.
		\item The equation $\bb{E}[\xi_{k+1}| \ca{F}_k] = \xi_{k}$ holds almost surely for every $k\geq 0$. 
	\end{itemize}
\end{defin}

\begin{defin}
	We say that a stochastic series $\{\xi_k\}$ is a supermartingale if $\{\xi_k\}$ is adapted to the filtration $\{ \ca{F}_{k} \}$ and $\bb{E}[\xi_{k+1}| \ca{F}_k] \leq \xi_{k}$ holds almost surely for every $k\geq 0$. 
\end{defin}

\begin{defin}
	We say that a stochastic series $\{\xi_k\}$ is a martingale difference sequence if the following conditions hold.
	\begin{itemize}
		\item The sequence of random vectors $\{\xi_k\}$ is adapted to the filtration $\{ \ca{F}_{k} \}$. 
		\item For each $k \geq 1$, almost surely, it holds that $\bb{E}[|\xi_k|] < +\infty$ and $\bb{E}\left[ \xi_k | \ca{F}_{k-1} \right] = 0$. 
	\end{itemize}
\end{defin}

The following proposition plays an important role in establishing the convergence properties for our proposed framework \eqref{Eq_framework}. In this proposition, we improve the results in \citep[Proposition 4.4]{benaim2006dynamics} and  demonstrate that with appropriately chosen $\{\eta_k\}$ and $\{\theta_k\}$, the uniform boundedness of the martingale difference sequence $\{\xi_k\}$ leads to the validity of the regularity conditions in \citep[Section 1.5]{benaim2005stochastic}. 

\begin{prop}
	\label{Prop_UB_martingale_difference_sequence}
	Suppose $\{\eta_k\}$ and $\{\theta_k\}$ are two diminishing positive sequences of real numbers that satisfy 
	\begin{equation*}
		\lim_{k \to +\infty} \frac{\theta_k^2}{\eta_k} \log(k) = 0.
	\end{equation*}
	Let $\lambda_0 := 0$, $\lambda_i := \sum_{k = 0}^{i-1} \eta_k$, and $\Lambda(t) := \sup  \{k \geq 0: t\geq \lambda_k\} $. Then for any $T > 0$, and any uniformly bounded martingale difference sequence $\{\xi_k\}$, almost surely, it holds that
	\begin{equation*}
		\lim_{s \to +\infty} \sup_{s\leq i \leq \Lambda(\lambda_s + T)}\norm{ \sum_{k = s}^{i} \theta_k \xi_k} =0 . 
	\end{equation*}
\end{prop}
\begin{proof}
	Since the martingale difference sequence $\{\xi_k\}$ is  uniformly bounded, $\xi_k$ is sub-Gaussian for any $k\geq 0$. Then there exists a constant $M>0$ such that for any $w \in \Rn$, it holds for any $k\geq 0$ that 
	\begin{equation*}
		\bb{E}\left[ \exp\left( \inner{w, \xi_{k+1}} \right) | \ca{F}_k \right] \leq \exp\left( \frac{M}{2}\norm{w}^2 \right),
	\end{equation*}
        holds almost surely. Therefore, for any $w \in \Rn$ and any $C > 0$, let 
        \begin{equation*}
            Z_i := \exp\left[  \inner{Cw, \sum_{k = s}^i \theta_k \xi_k} - \frac{MC^2}{2}\sum_{k = s}^i \theta_k^2 \norm{w}^2  \right].
        \end{equation*} 
        Then  for any $i\geq 0$, we have that $\bb{E}[Z_{i+1} | \ca{F}_i] \leq Z_{i}$. 
	Hence for any $\delta > 0$, and any $C > 0$, it holds that 
	\begin{equation*}
		\begin{aligned}
			&\bb{P}\left( \sup_{s\leq i \leq \Lambda(\lambda_s + T)} \inner{w, \sum_{k = s}^i \theta_k \xi_k} > \delta \right)={}\bb{P}\left( \sup_{s\leq i \leq \Lambda(\lambda_s + T)} \inner{Cw, \sum_{k = s}^i \theta_k \xi_k} > C\delta \right)\\
			\leq{}& \bb{P}\left( \sup_{s\leq i \leq \Lambda(\lambda_s + T)} Z_i > \exp\left( C\delta - \frac{MC^2}{2} \sum_{k = s}^{\Lambda(\lambda_s + T)} \theta_k^2 \norm{w}^2 \right) \right)\\
			\leq{}& \exp\left( \left(\frac{M}{2}\norm{w}^2 \sum_{k = s}^{\Lambda(\lambda_s + T)} \theta_k^2\right)C^2   - \delta C \right).
		\end{aligned}
	\end{equation*}
	Here the second inequality holds since $\{Z_k\}$  is a nonnegative super-martingale and $\bb{E}[Z_{s}] \leq 1$. Then from the arbitrariness of $C$, set $C = \frac{\delta}{M\norm{w}^2 \sum_{k = s}^{\Lambda(\lambda_s + T)}\theta_k^2}$, it holds that 
	\begin{equation*}
		\bb{P}\left( \sup_{s\leq i \leq \Lambda(\lambda_s + T)} \inner{w, \sum_{k = s}^i \theta_k \xi_k} > \delta \right) \leq \exp \left(\frac{-\delta^2}{2M\norm{w}^2 \sum_{k = s}^{\Lambda(\lambda_s + T)}\theta_k^2 }\right). 
	\end{equation*}
	From the arbitrariness of $w$, there exists constants $C_1$ and $C_2$ that only depend on $n$ such that 
	\begin{equation*}
		\bb{P}\left( \sup_{s\leq i \leq \Lambda(\lambda_s + T)} \norm{\sum_{k = s}^i \theta_k \xi_k} > \delta \right) \leq  C_1 \exp\left(\frac{-\delta^2}{2M C_2\sum_{k = s}^{\Lambda(\lambda_s + T)} \theta_k^2 }\right) \leq C_1\exp\left(\frac{-\delta^2}{2M C_2T\frac{\theta_{k'}^2}{\eta_{k'}} }\right),
	\end{equation*}
	holds for some $k' \in [s, \Lambda(\lambda_s + T)]$. Here $\{\etak\}$ refers to the diminishing sequence of real numbers as defined in the condition of this proposition.

	Therefore, for any $j \geq 0$, there exists $k_j\in [\Lambda(jT), \Lambda((j+1)T) ]$, such that
	\begin{equation*}
		\begin{aligned}
			&\sum_{j = 0}^{+\infty} \bb{P}\left(\sup_{\Lambda(jT)\leq i \leq \Lambda( jT+T)}\norm{ \sum_{k = s}^{i}\theta_k \xi_k} \geq \delta \right) \\
			\leq{}& \sum_{j=0}^{+\infty} C_1\exp\left( \frac{-\delta^2}{2MC_2T \eta_{k_j}^{-1}\theta_{k_j}^2} \right) \leq \sum_{k=0}^{+\infty} 2 C_1\exp\left( \frac{-\delta^2}{2MC_2T\frac{\theta_k^2}{\eta_k}} \right) < +\infty. 
		\end{aligned}
	\end{equation*}
	Here the last inequality holds from the fact that $\lim_{k \to +\infty}\frac{\theta_k^2}{\eta_k} \log(k) = 0 $. 
	Therefore, we can conclude that 
	\begin{equation*}
		\lim_{j\to +\infty}\bb{P}\left(\sup_{\Lambda(jT)\leq i \leq \Lambda( jT+T)}\norm{ \sum_{k = \Lambda(jT)}^{i}\theta_k \xi_k} \geq \delta\right) = 0,
	\end{equation*}
	holds almost surely for any $\delta >0$. Then the arbitrariness of $\delta$ illustrates that almost surely, we have
	\begin{equation*}
		\lim_{j \to +\infty} \sup_{\Lambda(jT)\leq i \leq \Lambda( jT+T)}\norm{ \sum_{k = \Lambda(jT)}^{i}\theta_k \xi_k} = 0. 
	\end{equation*}
	Finally, notice that for any $jT \leq s\leq jT+T$, it holds that 
	\begin{equation*}
		\begin{aligned}
		    &\sup_{s\leq i \leq \Lambda( \lambda_s+T)}\norm{ \sum_{k = \Lambda(jT)}^{i}\theta_k \xi_k} \\
        \leq{}& 2\sup_{\Lambda(jT)\leq i \leq \Lambda( jT+T)}\norm{ \sum_{k = \Lambda(jT)}^{i}\theta_k \xi_k} + \sup_{\Lambda((j+1)T)\leq i \leq \Lambda( (j+2)T)}\norm{ \sum_{k = \Lambda(jT + T)}^{i}\theta_k \xi_k}. 
		\end{aligned}
	\end{equation*}
	Then we achieve that 
	\begin{equation*}
		\lim_{s \to +\infty} \sup_{s\leq i \leq \Lambda( \lambda_s+T)}\norm{ \sum_{k = s}^{i}\theta_k \xi_k} = 0,
	\end{equation*}
        holds almost surely. 
	Hence we complete the proof. 
\end{proof}

\subsection{Nonsmooth Analysis}
\label{Section_Nonsmooth_Analysis}

\subsubsection{Clarke Subdifferential}
In this part, we introduce the concept of Clarke subdifferential \citep{clarke1990optimization}, which plays an important role in characterizing the stationarity and designing efficient algorithms for nonsmooth optimization problems.

\begin{defin}[\citet{clarke1990optimization}]
	\label{Defin_Subdifferential}
	For any given locally Lipschitz continuous function $f: \Rn \to \bb{R}$ and any $x \in \Rn$, the generalized directional derivative of $f$ at $x$  in the direction $d \in \Rn$, denoted by $f^\circ(x; d)$, is defined as 
	\begin{equation*}
		f^\circ(x; d) := \mathop{\lim\sup}_{\tilde{x}\to x, ~t \downarrow  0}~ \frac{f(\tilde{x} + td) - f(\tilde{x})}{t}. 
	\end{equation*}
	Then the generalized gradient or the Clarke subdifferential of $f$ at $x$, denoted by $\partial f(x)$, is defined as 
	\begin{equation*}
		\begin{aligned}
			\partial f(x) := &\left\{ w \in \Rn :  \inner{w, d} \leq f^\circ (x; d), \text{ for all } d \in \Rn \right\}.
		\end{aligned}
	\end{equation*}
\end{defin}

Then based on the concept of generalized directional derivative, we present the definition of (Clarke) regular functions. 
\begin{defin}[\citet{clarke1990optimization}]
	\label{Defin_Clarke_regular}
	For any given locally Lipschitz continuous function $f: \Rn \to \bb{R}$ and any $x \in \Rn$, we say that $f$ is (Clarke) regular at $x \in \Rn$ if for every direction $d\in \Rn$, the one-sided directional derivative 
	\begin{equation*}
		f^\star(x;d) := \lim_{t \downarrow 0} \frac{f(x + td) - f(x)}{t} 
	\end{equation*}
	exists and $f^\star(x; d) = f^\circ(x;d)$.
\end{defin}

\subsubsection{Conservative Field}
In this part, we present a brief introduction on the conservative field, which can be applied to characterize how the nonsmooth neural networks are differentiated by AD algorithms. 

\begin{defin}
	A set-valued mapping $\ca{D}: \Rn \rightrightarrows \bb{R}^s$ is a mapping from $\Rn$ to a collection of subsets of $\bb{R}^s$. $\D$ is said to have closed graph (or $\D$ is graph-closed) if the graph of $\ca{D}$, defined by
	\begin{equation*}
		\mathrm{graph}(\D) := \left\{ (w,z) \in \Rn \times \bb{R}^s: w \in \Rn, z \in \D(w) \right\},
	\end{equation*}
	is a closed subset of $\Rn \times \bb{R}^s$.  
\end{defin}

\begin{defin}
	A set-valued mapping $\ca{D}: \Rn \rightrightarrows \bb{R}^s$ is said to be locally bounded if, for any $x \in \Rn$, there is a neighborhood $V_x$ of $x$ such that $\cup_{y \in V_x}\ca{D}(y)$ is bounded. 
\end{defin}

\begin{defin}[Aumann’s integral]
    \label{Defin_Aumann_integral}
    Let $(\Theta, \ca{F}, P)$ be a measurable space, and $\D : \Rn  \times \Theta \rightrightarrows \Rn$ be a measurable set-valued mapping. Then for all $x \in \Rn$, the integral of $\ca{D}$ with respect to $P$ is defined as 
    \begin{equation*}
        \bb{E}_{s \sim P}\left[ \D(x, s)  \right] := \left\{ \int_{\Theta} \chi(x, s) ~\mathrm{d}P(s): \text{  $\chi(x, \cdot)$  is integrable, and $\chi(x, s) \in \D(x, s)$ for any $s \in \Theta$}  \right\}. 
    \end{equation*}
\end{defin}

The following lemma illustrates that the composition of two locally bounded graph-closed set-valued mappings is locally bounded and graph-closed. Therefore, we can easily verify the graph-closeness for the composition of set-valued mappings, which plays an important role in our theoretical analysis. 

\begin{lem}[Lemma 2.5 in \citep{xiao2023convergence}]
	\label{Le_closed_graph}
	Suppose $\D_1: \Rn \rightrightarrows \bb{R}^s$ and $\D_2: \bb{R}^d \rightrightarrows \Rn$  are two  locally bounded graph-closed set-valued mappings, then their composition $\D_1 \circ \D_2$ is locally bounded and graph-closed. 
\end{lem}

In the following definitions, we present the definition for the conservative field and its corresponding potential function.

\begin{defin}
	An absolutely continuous curve is a continuous mapping $\gamma: \bb{R}_+ \to \Rn $ whose derivative $\gamma'$ exists almost everywhere in $\bb{R}_+$ and $\gamma(t) - \gamma(0)$ equals to the Lebesgue integral of $\gamma'$ between $0$ and $t$ for all $t \in \bb{R}_+$, i.e.,
	\begin{equation*}
		\gamma(t) = \gamma(0) + \int_{0}^t \gamma'(u) \mathrm{d} u, \qquad \text{for all $t \in \bb{R}_+$}.
	\end{equation*}
\end{defin}

\begin{defin}
	\label{Defin_conservative_field}
	Let $\ca{D}$ be a set-valued mapping from $\bb{R}^{n}$ to subsets of $\bb{R}^{n}$. Then we call $\ca{D}$ as a conservative field whenever it has closed graph,  nonempty compact valued, and for any absolutely continuous curve $\gamma: [0,1] \to \bb{R}^{n} $ satisfying $\gamma(0) = \gamma(1)$, we have
	\begin{equation}
		\label{Eq_Defin_Conservative_mappping}
		\int_{0}^1 \max_{v \in \ca{D}(\gamma(t)) } \inner{\gamma'(t), v} \mathrm{d}t = 0, 
	\end{equation}
	where the integral is understood in the Lebesgue sense. 
\end{defin}

It is important to note that any conservative field is locally bounded  \citep[Remark 3]{bolte2021conservative}. We now introduce the definition of the potential function corresponding to the conservative field.

\begin{defin}
	\label{Defin_conservative_field_path_int}
	Let $\ca{D}$ be a conservative field in $\bb{R}^{n}$. Then with any given $x_0 \in \bb{R}^{n}$, we can define a function $f: \Rn \to \bb{R}$ through the path integral
	\begin{equation}
		\label{Eq_Defin_CF}
		\begin{aligned}
			f(x) = &{} f(x_0) + \int_{0}^1 \max_{v \in \ca{D}(\gamma(t)) } \inner{\gamma'(t), v} \mathrm{d}t
			= f(x_0) + \int_{0}^1 \min_{v \in \ca{D}(\gamma(t)) } \inner{\gamma'(t), v} \mathrm{d}t,
		\end{aligned}
	\end{equation} 
	for any absolutely continuous curve $\gamma$ that satisfies $\gamma(0) = x_0$ and $\gamma(1) = x$.  Then $f$ is called a potential function for $\ca{D}$, and we also say $\ca{D}$ admits $f$ as its potential function, or that $\ca{D}$ is a conservative field for $f$. 
\end{defin}

The following two lemmas characterize the relationship between the conservative field and the Clarke subdifferential. 

\begin{lem}[Theorem 1 in \citet{bolte2021conservative}]
	\label{Le_Conservative_as_gradient}
	Let $f:\Rn \to \bb{R}$ be a potential function that admits $\D_f$ as its conservative field. Then $\D_f(x) = \{\nabla f(x)\}$ almost everywhere.  
\end{lem}

\begin{lem}[Corollary 1 in \citet{bolte2021conservative}]
	Let $f:\Rn \to \bb{R}$ be a potential function that admits $\D_f$ as its conservative field. Then $\partial f$ is a conservative field for $f$, and for all $x \in \Rn$, it holds that 
	\begin{equation*}
		\partial f(x) \subseteq \conv\left(\D_f(x)\right). 
	\end{equation*} 
\end{lem}

\begin{defin}
    Let $f:\Rn \to \bb{R}$ be a potential function that admits $\D_f$ as its conservative field, then we say $x$ is a $\D_f$-stationary point of $f$ if $0 \in \D_f(x)$. In particular, we say $x$ is a $\partial f$-stationary point of $f$ if $0 \in \partial f(x)$. 
\end{defin}

It is worth mentioning that the class of potential functions is general enough to cover the objectives in a wide range of real-world problems. As shown in \citet[Section 5.1]{davis2020stochastic}, any Clarke regular function is a potential function. Another important function class is the definable functions (i.e. the functions whose graphs are definable in an $o$-minimal structure) \citep[Definition 5.10]{davis2020stochastic}. As demonstrated in \citet{van1996geometric}, any definable function is also a potential function \citep{davis2020stochastic,bolte2021conservative}. To characterize the definable functions, the Tarski–Seidenberg theorem \citep{bierstone1988semianalytic} shows that any semi-algebraic function is definable. Moreover, \citet{wilkie1996model} shows there exists an $o$-minimal structure that contains both the graph of the exponential function and all semi-algebraic sets. As a result, numerous common activation and loss functions, including sigmoid, softplus, ReLU, $\ell_1$-loss, MSE loss, hinge loss, logistic loss, and cross-entropy loss, are all definable. Additionally, \citet{bolte2021nonsmooth} reveals that parameterized solutions in a broad class of optimizations are definable.

Additionally, it should be noted that definability is preserved under finite summation and composition \citep{davis2020stochastic,bolte2021conservative}. As a result, for any neural network built from definable blocks, its loss function is definable and thus is a potential function. Moreover, the Clarke subdifferential of definable functions are definable \citep{bolte2021conservative}. Therefore, for any neural network built from definable blocks, the conservative field corresponding to the AD algorithms is definable. The following proposition shows that the definability of $f$ and $\mathcal{D}_f$ leads to the nonsmooth Morse–Sard property \citep{bolte2007clarke} for \ref{Prob_Ori}. 
\begin{prop}[Theorem 5 in \citet{bolte2021conservative}]
	\label{Prop_definable_regularity}
	Let $f$ be a potential function that admits $\D_f$ as its conservative field. Suppose both $f$ and $\D_f$ are definable over $\Rn$, then the set $\{f(x): 0\in \D_f(x)\}$ is finite. 
\end{prop}

\subsection{Differential Inclusion and Stochastic Subgradient Methods}
In this subsection, we introduce some fundamental concepts related to the differential equation (i.e., differential inclusion) that are essential for the proofs presented in this paper.
\begin{defin}
	For any locally bounded set-valued mapping $\ca{D}: \Rn \rightrightarrows \Rn$ that is nonempty compact convex valued and has closed graph.  We say the absolutely continuous path $x(t)$ in $\Rn$ is a solution for the differential inclusion 
	\begin{equation}
		\label{Eq_def_DI}
		\frac{\mathrm{d} x}{\mathrm{d}t} \in \ca{D}(x),
	\end{equation}
	with initial point $x_0$ if $x(0) = x_0$, and $\dot{x}(t) \in \ca{D}(x(t))$ holds for almost every $t\geq 0$. 
\end{defin}

\begin{defin}
    \label{Defin_delta_expansion}
    For any given set-valued mapping $\D: \Rn \rightrightarrows \Rn$ and any constant $\delta \geq 0$,  the set-valued mapping $\D^{\delta}$ is defined as
    \begin{equation}
        \D^{\delta}(x) := \{ w \in \Rn: \exists z \in \bb{B}_{\delta}(x), \, \mathrm{dist}(w, \ca{D}(z))\leq \delta \}.
    \end{equation}
\end{defin}

For the differential inclusion \eqref{Eq_def_DI}, we introduce the concept on its perturbed solution in the following definition.
\begin{defin}[Definition II in \citet{benaim2005stochastic}]
	\label{Defin_perturbed_solution}
	We say an absolutely continuous function $\gamma$
	 is a perturbed solution to \eqref{Eq_def_DI}  if there exists a locally integrable function $u: \bb{R}_+ \to \Rn$, such that 
	\begin{itemize}
		\item For any $T>0$, it holds that $\lim\limits_{t \to +\infty} \sup\limits_{0\leq l\leq T} \norm{\int_{t}^{t+l} u(s) ~\mathrm{d}s} = 0$. 
		\item There exists $\delta: \bb{R}_+ \to \bb{R}$ such that $\lim\limits_{t \to +\infty} \delta(t) = 0$ and 
		\begin{equation*}
			\dot{\gamma}(t) - u(t) \in \D^{\delta(t)}(\gamma(t)). 
		\end{equation*}
	\end{itemize}
\end{defin}

Now consider the sequence $\{\xk\}$ generated by the  following updating scheme,  
\begin{equation}
	\label{Eq_def_Iter}
	\xkp = \xk + \eta_k(d_k + \xi_k),
\end{equation}
where $\{\eta_k\}$ is a diminishing positive sequence of real numbers. 
We define the  (continuous-time) interpolated process of $\{\xk\}$ generated by \eqref{Eq_def_Iter} as follows. 
\begin{defin}
	The  (continuous-time) interpolated process of $\{\xk\}$ generated by \eqref{Eq_def_Iter} is the mapping $w: \bb{R}_+ \to \Rn$ such that 
	\begin{equation*}
		w(\lambda_i + s) := x_{i-1} + \frac{s}{\lambda_i - \lambda_{i-1}} \left( x_{i} - x_{i-1} \right), \quad s\in[0, \eta_i). 
	\end{equation*}
	Here $\lambda_0 := 0$, and $\lambda_i := \sum_{k = 0}^{i-1} \eta_k$.
\end{defin}

The following lemma is an extension of \citep[Proposition 1.3]{benaim2005stochastic}. Compared with \citep[Proposition 1.3]{benaim2005stochastic}, Lemma \ref{Le_interpolated_process} allows for inexact evaluations of the set-valued mapping $\D$, and shows that the interpolated process of $\{\xk\}$ from \eqref{Eq_def_Iter} is a perturbed solution of the differential inclusion \eqref{Eq_def_DI}.   
\begin{lem}
	\label{Le_interpolated_process}
	Let $\ca{D}: \Rn \rightrightarrows \Rn$ be a locally bounded set-valued mapping that is nonempty compact convex valued with closed graph.
	Suppose the following conditions hold in \eqref{Eq_def_Iter}:
	\begin{enumerate}
		\item For any $T> 0$, it holds that
		\begin{equation*}
			\lim_{s \to +\infty} \sup_{s\leq i \leq \Lambda(\lambda_s + T)}\norm{ \sum_{k = s}^{i} \eta_k \xi_k} =0.
		\end{equation*}
		\item There exist a positive sequence $\{\delta_k\}$  such that $\lim_{k\to +\infty} \delta_k = 0$ and $d_k \in \D^{\delta_k}(\xk)$.
		\item $\sup_{k \geq 0} \norm{\xk}<+\infty$, $\sup_{k \geq 0} \norm{d_k} < +\infty$. 
	\end{enumerate}
	Then the interpolated process of $\{\xk\}$ is a perturbed solution for \eqref{Eq_def_DI}. 
\end{lem}
\begin{proof}
    Let $w: \bb{R}_+\to \Rn$ denote the interpolated process for \eqref{Eq_def_Iter}. Then define $u: \bb{R}_+\to \Rn$ as 
	\begin{equation*}
		u(\lambda_j + s):= \xi_{j}, \quad \text{for any $j\geq 0$}. 
	\end{equation*}
	
	Therefore, for any $t > 0$, it holds that
	\begin{equation*}
		\dot{w}(t) = u(t) + d_{\Lambda(t)}. 
	\end{equation*}
	Let 
	\begin{equation*}
		\delta(t) := \norm{w(t) - x_{\Lambda(t)}} + \sup_{k \geq \Lambda(t)} \delta_k.
	\end{equation*}
	Then it holds that 
	\begin{equation*}
		\dot{w}(t) - u(t) \in  \D^{\delta(t)}(w(t)) .
	\end{equation*}
	From the definition of $\delta(t)$, it holds that
	\begin{equation*}
		\limsup_{t \to +\infty} \delta(t) \leq \limsup_{t \to +\infty} \left(\eta_{\Lambda(t)}  \norm{\xi_{\Lambda(t)} + d_{\Lambda(t)}}  + \sup_{k \geq \Lambda(t)} \delta_k\right) = 0. 
	\end{equation*}
	In addition, from the definition of $u(t)$, we achieve
	\begin{equation*}
		\lim_{t \to +\infty} \sup_{0\leq l\leq T} \norm{\int_{t}^{t+l} u(s) \mathrm{d}s} = \lim_{s \to +\infty} \sup_{s\leq i \leq \Lambda(\lambda_s + T)}\norm{ \sum_{k = s}^{i} \eta_k \xi_k} =0.
	\end{equation*}
	Therefore, from Definition \ref{Defin_perturbed_solution}, we can conclude that $w$ is a perturbed solution for \eqref{Eq_def_DI}. This completes the proof. 
\end{proof}

\section{A General Framework for Convergence Properties}
\label{Section_3}

\subsection{Convergence to $\D_f$-stationary Points}
\label{Subsection_31}
In this subsection, we aim to show the convergence properties of the proposed abstract framework \eqref{Eq_framework}.  We first make the following assumptions on $f$.
\begin{assumpt}
	\label{Assumption_f}
	For the problem \ref{Prob_Ori}, we assume $f$ is locally Lipschitz continuous, bounded from below. Moreover, there exists a compact convex valued mapping $\D_f:  \Rn \rightrightarrows \Rn$ that has closed graph such that  
	\begin{enumerate}
		\item $f$ is a potential function that admits $\D_f$ as its conservative field. 
		\item The set $\{f(x): 0\in \D_f(x)\}$ has empty interior in $\bb{R}$. That is, the complementary of $\{f(x): 0\in \D_f(x)\}$ is dense in $\bb{R}$.  
	\end{enumerate}
\end{assumpt} 
As discussed in Section \ref{Section_Nonsmooth_Analysis}, Assumption \ref{Assumption_f}(1) is  satisfied in a wide range of applications of \ref{Prob_Ori}. Moreover, Assumption \ref{Assumption_f}(2) is the weak Sard's theorem, which has been shown to be a mild assumption, as demonstrated in  \citep{davis2020stochastic,castera2021inertial}.

Furthermore,  we make the following assumptions on the framework \eqref{Eq_framework}.
\begin{assumpt}
	\label{Assumption_alg}
	\begin{enumerate}
		\item The parameters satisfy $\alpha, \gamma \geq 0$, and $ \varepsilon, \tau_1, \tau_2 > 0$. 
		\item The sequences of iterates $\{\xk\}$, $\{\mk\}$ and $\{\vk\}$ are almost surely bounded, i.e., 
		\begin{equation*}
				\sup_{k\geq 0} \norm{\xk} + \norm{\mk} + \norm{\vk}  <+\infty
		\end{equation*}
		holds almost surely. 
		\item $\ca{U}$ is a locally bounded set-valued mapping 
		that is convex compact valued with closed graph. Moreover, there exists a constant $\kappa \geq 0$  such that for any $x, m, v \in \Rn$, it holds that 
		\begin{equation*}
			\widetilde{\mathrm{sign}}(v) \odot \ca{U}(x, m, v)  \geq \kappa |v| \geq 0.  
		\end{equation*}
		\item $(d_{x, k}, d_{m, k}, d_{v, k})$ is an approximated evaluation for $\ca{G}(\xk, m_k, v_k)$ in the sense that there exists a positive sequence of real numbers $\{\delta_k\}$ such that $\lim_{k\to +\infty} \delta_k = 0$ and 
		\begin{equation*}
                (d_{x, k}, d_{m, k}, d_{v, k}) \in \ca{G}^{\delta_k}(\xk, \mk, \vk).
		\end{equation*}
		\item $\{(\xi_{x, k}, \xi_{m, k}, \xi_{v, k})\}$ is a uniformly bounded martingale sequence. That is, almost surely, it holds for any $k \geq 1$ that 
		\begin{equation*}
			\bb{E}[(\xi_{x, k}, \xi_{m, k}, \xi_{v, k}) | \ca{F}_{k-1}] = 0, \quad \text{and} \quad \sup_{k \geq 0} ~\norm{(\xi_{x, k}, \xi_{m, k}, \xi_{v, k})} <+\infty. 
		\end{equation*}
		\item The stepsizes $\{\eta_k\}$ and $\{\theta_k\}$ are positive and satisfy
		\begin{equation*}
			\sum_{k = 0}^{+\infty} \eta_k = +\infty, ~ \sum_{k = 0}^{+\infty} \theta_k = +\infty,~  \lim_{k\to +\infty} \eta_k \log(k) = 0, ~ \text{and} ~ \lim_{k\to +\infty} \frac{\theta_k^2}{\eta_k} \log(k) = 0. 
		\end{equation*}
	\end{enumerate}
\end{assumpt}

Here are some comments for Assumption \ref{Assumption_alg}. Assumption \ref{Assumption_alg}(2) assumes that the generated sequence $\{(\xk, \mk, \vk)\}$ and the updating directions $\{(d_{x,k}, d_{m, k}, d_{v,k})\}$ are uniformly bounded, which is a common assumption in various existing works \citep{benaim2005stochastic,benaim2006dynamics,davis2020stochastic,bolte2021conservative,castera2021inertial}. 
Assumption \ref{Assumption_alg}(3) enforces regularity conditions on the set-valued mapping $\ca{U}$, which are satisfied in a wide range of adaptive stochastic gradient methods such as  Adam, AdaBelief, AMSGrad, NAdam, Yogi, as discussed later in Section \ref{Section_4}. Assumption \ref{Assumption_alg}(4) illustrates how $(d_{x,k}, d_{m, k}, d_{v,k})$ approximates $\ca{G}(\xk, \mk, \vk)$, which is a mild assumption commonly used in existing works  \citep{benaim2005stochastic,benaim2006dynamics,bolte2021conservative,castera2021inertial}. In addition, Assumption \ref{Assumption_alg}(5) is a prevalent assumption in various existing works \citep{bolte2021conservative,castera2021inertial}.

Furthermore, Assumption \ref{Assumption_alg}(6) allows for a flexible choice of the stepsize ${\eta_k}$ in \eqref{Eq_framework}, enabling it to be set in the order of $o(1/\log(k))$. It is easy to verify that a simple choice of $\theta_k = \eta_k$ satisfies Assumption \ref{Assumption_alg}(6). Hence the framework \eqref{Eq_framework} includes those cases where the evaluation noises are uniformly bounded.   More importantly, Assumption \ref{Assumption_alg}(6) allows for a two-timescale scheme for \eqref{Eq_framework} in the sense that we can choose $\{\theta_k\}$ satisfying $\theta_k/\eta_k \to +\infty$, since we can always set $\theta_k =\eta_k \left(\eta_k \log(k)\right)^{-s}$ for any $s \in (0, \frac{1}{2})$. As shown in Section \ref{Section_4}, the two-timescale framework is crucial for developing adaptive subgradient methods with gradient clipping techniques when the evaluation noises are only assumed to be integrable. The two-timescale updating scheme assumed in Assumption \ref{Assumption_alg} distinguishes our proposed framework \eqref{Eq_framework} from existing frameworks  in \citep{benaim2005stochastic,davis2020stochastic,bolte2021conservative}.

\begin{rmk}
    It is worth mentioning that some existing works have investigated the conditions to ensure the sequence of iterates to be uniformly bounded almost surely, including the proximal stochastic subgradient descent method \citep{davis2020stochastic}, the SGD method with constant stepsize \citep{bianchi2022convergence}, and noiseless heavy-ball SGD method \citep{josz2023global}. However, their analysis is limited within specific subgradient methods, making it challenging to establish similar results for  adaptive methods.   

    On the other hand, when $f$ is assumed to be differentiable, 
 \citet{barakat2021convergence,gadat2022asymptotic,li2022unified} utilize the Robbins-Siegmund theorem \citep{robbins1971convergence} to prove that the function values of their merit functions are uniformly bounded almost surely. Then based on the coercivity of their employed merit functions, they establish the uniform boundedness of the sequence generated by adaptive methods.  However, when $f$ is assumed to be nonsmooth and nonconvex, it is challenging to estimate the decrease of the objective function over the iterations. As a result,  their proof techniques cannot be applied to prove the uniform boundedness of $\{(\xk, \mk, \vk)\}$ within the context of the framework \eqref{Eq_framework} under nonsmooth settings. How to establish the uniform boundedness for the sequence $\{(\xk, \mk, \vk)\}$ under the framework \eqref{Eq_framework} still remains open.

\end{rmk}

To prove the convergence properties of \eqref{Eq_framework}, we consider the following differential inclusion
\begin{equation}
	\label{Eq_DI}
	\left(\frac{\mathrm{d} x}{\mathrm{d}t},\frac{\mathrm{d} m}{\mathrm{d}t},\frac{\mathrm{d} v}{\mathrm{d}t}\right)
	\in -\ca{G}(x, m, v). 
\end{equation}

Let the function $\phi: \Rn\times \Rn \times \Rn \to \bb{R}$ be defined as, 
\begin{equation*}
	\phi(x, m, v) := f(x) + \frac{1}{2\tau_1} \inner{m, (|v| + \varepsilon )^{-\gamma} \odot m},
\end{equation*}
and let the set $\ca{B}$ be chosen as $\ca{B}:= \{ (x, m, v) \in \Rn \times \Rn \times \Rn: 0 \in \D_f(x), m = 0 \}$. Then we have the following lemma to illustrate the relationship between $\D_f$-stationary points of $f$ and $\ca{B}$. The proof for Lemma \ref{Le_stable_set} is straightforward, hence we omit it for simplicity. 
\begin{lem}
	\label{Le_stable_set}
	The function $\phi$ is locally Lipschitz continuous over $\Rn\times \Rn \times \Rn$, and  $\{f(x): 0\in \D_f(x)\} = \{\phi(x, m, v): (x, m, v) \in \ca{B}\}$. 
\end{lem}

    

In the following lemma, we illustrate that $\phi$ is a potential function whenever $f$ is a potential function for $\D_f$, and investigate the expression of the conservative field for $\phi$. Lemma \ref{Le_lyapunov_conservative_field} directly follows from the expressions for $\phi$ and the validity of the chain rule for the conservative field \citep{bolte2021conservative}, hence we omit the proof for Lemma \ref{Le_lyapunov_conservative_field} for simplicity.

\begin{lem}
	\label{Le_lyapunov_conservative_field}
	Suppose $f$ is a potential function that admits $\D_f$ as its conservative field, then $\phi$ is a potential function that admits the conservative field $\D_{\phi}$ defined by
	\begin{equation*}
		\D_{\phi}(x, m, v) = 
		\left[\begin{matrix}
			\D_f(x)\\
			\frac{1}{\tau_1}(|v| + \varepsilon)^{-\gamma} \odot m\\
			-\frac{\gamma}{2\tau_1 } m^2 \odot (|v| + \varepsilon)^{-\gamma - 1} \odot \mathrm{sign}(v) 
		\end{matrix}\right]. 
	\end{equation*}
\end{lem}

The following proposition illustrates that $\phi$ is a Lyapunov function for $\ca{B}$ with respect to the differential inclusion \eqref{Eq_DI}. When $f$ is assumed to be differentiable, similar Lyapunov functions have been proposed in \citep{barakat2021stochastic,barakat2021convergence,gadat2022asymptotic}.

\begin{prop}
	\label{Prop_Lyapunov}
	Suppose Assumption \ref{Assumption_f} and Assumption \ref{Assumption_alg} hold with 
	$(1-\kappa)\gamma \tau_2 \leq 2\tau_1$. For any $(x_0, m_0, v_0) \notin \ca{B}$, let $(x(t), m(t), v(t))$ be any trajectory of the differential inclusion \eqref{Eq_DI} with initial point $(x_0, m_0, v_0)$. Then for any $t > 0$, it holds that 
	\begin{equation*}
		\phi(x(t), m(t), v(t))  <\phi(x(0), m(0), v(0)).
	\end{equation*}
	That is, $\phi$ is the Lyapunov function for $\ca{B}$ with respect to the differential inclusion \eqref{Eq_DI}.
\end{prop}
\begin{proof} 
	From the definition of \eqref{Eq_DI}, for any $(x(t), m(t), v(t))$ that is a trajectory of the differential inclusion \eqref{Eq_DI},   there exists measurable functions $l_{f}(s)$ and $l_{v}(s)$ such that for almost every $s \geq 0$,  $l_{f}(s) \in \D_f(x(s))$,  $l_v(s) \in \ca{U}(x(s), m(s), v(s))$,  and 
	\begin{equation*}
		(\dot{x}(s), \dot{m}(s), \dot{v}(s)) = -
		\left[\begin{matrix}
			(|v(s)| + \varepsilon )^{-\gamma} \odot \left(   m(s) + \alpha l_{f}(s) \right)\\
			\tau_1 m(s) - \tau_1 l_f(s)\\
			\tau_2 v(s) - \tau_2  l_v(s)\\
		\end{matrix}\right].
	\end{equation*}

	Therefore, from the expression of $\D_{\phi}$, we can conclude that 
	\begin{equation*}
		\begin{aligned}
			&\inner{(\dot{x}(s), \dot{m}(s), \dot{v}(s)), \ca{D}_{\phi}(x(s), m(s), v(s))}\\
			={}& \inner{\D_f(x(s)), -(|v(s)| + \varepsilon )^{-\gamma} \odot \left( m(s) + \alpha  l_f(s)\right)} \\
			&+ \frac{1}{\tau_1} \inner{ (|v(s)|+ \varepsilon )^{-\gamma} \odot m(s), -\tau_1 m(s) + \tau_1 l_f(s)}\\
			& + \frac{-\gamma}{2\tau_1} \inner{m(s)^2 \odot (|v(s)|+\varepsilon)^{-\gamma-1} \odot \widetilde{\mathrm{sign}}(v(s)), (-\tau_2 v(s) + \tau_2 l_v(s) )}\\
			\ni{}& -\alpha \inner{l_{f}(s), (|v(s)| + \varepsilon )^{-\gamma}\odot l_{f}(s)}  \\
			&  - \inner{m(s), (|v(s)| + \varepsilon )^{-\gamma} \odot m(s)} +\frac{\gamma \tau_2 }{2\tau_1}\inner{m(s)^2, (|v(s)|+\varepsilon)^{-\gamma-1} \odot|v(s)|}\\ 
			&- \frac{\gamma\tau_2 }{2\tau_1} \inner{m(s)^2  \odot (|v(s)|+\varepsilon)^{-\gamma-1},\widetilde{\mathrm{sign}}(v(s))  \odot l_v(s)}.
		\end{aligned}
	\end{equation*}

	Notice that for any $v \in \Rn$, 
	\begin{equation*}
		(|v|+\varepsilon)^{-\gamma - 1} \odot|v| = (|v| + \varepsilon)^{-\gamma} - \varepsilon(|v|+\varepsilon)^{-\gamma-1},
	\end{equation*}
	we have that  
	\begin{equation}
		\label{Eq_Prop_Lyapunov_0}
		\begin{aligned}
			&- \inner{m(s), (|v(s)| + \varepsilon )^{-\gamma} \odot m(s)} +\frac{\gamma \tau_2 }{2\tau_1}\inner{m(s)^2, (|v(s)|+\varepsilon)^{-\gamma-1} \odot|v(s)|}\\
			={}& -\left( 1 - \frac{\gamma \tau_2}{2\tau_1} \right)\inner{m(s), (|v(s)| + \varepsilon )^{-\gamma} \odot m(s)} 
			- \frac{\varepsilon \gamma \tau_2}{2\tau_1}\inner{m(s), (|v(s)| + \varepsilon )^{-\gamma-1} \odot m(s)}. 
		\end{aligned}
	\end{equation}
	Moreover, Assumption \ref{Assumption_alg}(3) illustrates that 
	$\widetilde{\mathrm{sign}}(v(s))  \odot U(x(s),m(s),v(s)) \geq \kappa |v(s)|$.
	Hence for any $s>0$, we have
	\begin{equation}
		\label{Eq_Prop_Lyapunov_1}
	 \begin{aligned}
	 	&\inner{m(s)^2  \odot (|v(s)|+\varepsilon)^{-\gamma-1},\widetilde{\mathrm{sign}}(v(s))  \odot l_v(s) } 
		\\
		\geq{} &
		\kappa 
		\inner{m(s)^2  \odot (|v(s)|+\varepsilon)^{-\gamma-1},|v(s)|} \geq 0. 
	 \end{aligned}
	\end{equation}	
        Additionally, under Assumption \ref{Assumption_alg}, let the positive constant $\delta_{\gamma}$ be defined as $$\delta_{\gamma} = \begin{cases}
		1, & \gamma = 0;\\
		\frac{\gamma\tau_2}{2\tau_1}, & \gamma > 0. 
	\end{cases}$$
        Then when $\gamma = 0$, it holds that
	\begin{equation*}
		\begin{aligned}
			&\left( 1 - \frac{(1-\kappa)\gamma \tau_2}{2\tau_1} \right)\inner{m(s), (|v(s)| + \varepsilon )^{-\gamma} \odot m(s)}= \inner{m(s), (|v(s)| + \varepsilon )^{-\gamma} \odot m(s)}\\
			\geq{}& \varepsilon \inner{m(s), (|v(s)| + \varepsilon )^{-\gamma-1} \odot m(s)} = \varepsilon \delta_{\gamma} \inner{m(s), (|v(s)| + \varepsilon )^{-\gamma-1} \odot m(s)}
		\end{aligned} 
	\end{equation*}
        On the other hand, when $\gamma >0$, it holds from the definition of $\delta_{\gamma}$ that 
        \begin{equation*}
		\begin{aligned}
			\frac{\varepsilon \gamma \tau_2}{2\tau_1}\inner{m(s), (|v(s)| + \varepsilon )^{-\gamma-1} \odot m(s)} =  \varepsilon \delta_{\gamma} \inner{m(s), (|v(s)| + \varepsilon )^{-\gamma-1} \odot m(s)}
		\end{aligned}.
	\end{equation*}
	Therefore, we can conclude that under Assumption \ref{Assumption_alg},  it holds for any $\gamma \geq 0$ that 
	\begin{equation}
		\label{Eq_Prop_Lyapunov_2}
		\begin{aligned}
			&\left( 1 - \frac{(1-\kappa)\gamma \tau_2}{2\tau_1} \right)\inner{m(s), (|v(s)| + \varepsilon )^{-\gamma} \odot m(s)} + \frac{\varepsilon \gamma \tau_2}{2\tau_1}\inner{m(s), (|v(s)| + \varepsilon )^{-\gamma-1} \odot m(s)}\\
			\geq{}& \varepsilon \delta_{\gamma}\inner{m(s), (|v(s)| + \varepsilon )^{-\gamma-1} \odot m(s)}. 
		\end{aligned}
	\end{equation}
	As a result, for any $s \geq 0$, we have 
	\begin{equation*}
		\begin{aligned}
			&\inf_{d_{\phi} \in \D_{\phi}(x(s), m(s), v(s))}   \inner{ (\dot{x}(s), \dot{m}(s), \dot{v}(s)) ,  d_{\phi}}\\
			\leq{}& -\alpha \inner{l_{f}(s), (|v(s)| + \varepsilon )^{-\gamma}\odot l_{f}(s)}  \\
			&  - \inner{m(s), (|v(s)| + \varepsilon )^{-\gamma} \odot m(s)} +\frac{\gamma \tau_2 }{2\tau_1}\inner{m(s)^2, (|v(s)|+\varepsilon)^{-\gamma-1} \odot|v(s)|}\\ 
			&- \frac{\gamma\tau_2 }{2\tau_1} \inner{m(s)^2  \odot (|v(s)|+\varepsilon)^{-\gamma-1},\widetilde{\mathrm{sign}}(v(s))  \odot l_v(s)}\\
			\leq{}&    - \inner{m(s), (|v(s)| + \varepsilon )^{-\gamma} \odot m(s)} +\frac{\gamma \tau_2 }{2\tau_1}\inner{m(s)^2, (|v(s)|+\varepsilon)^{-\gamma-1} \odot|v(s)|}\\ 
			&- \frac{\gamma\tau_2 }{2\tau_1} \inner{m(s)^2  \odot (|v(s)|+\varepsilon)^{-\gamma-1},\widetilde{\mathrm{sign}}(v(s))  \odot l_v(s)}\\
			\leq{}&-  \varepsilon\delta_{\gamma} \inner{m(s), (|v(s)| + \varepsilon )^{-\gamma-1} \odot m(s)}. 
		\end{aligned}
	\end{equation*}
	Here the last inequality directly follows \eqref{Eq_Prop_Lyapunov_0}, \eqref{Eq_Prop_Lyapunov_1}, and \eqref{Eq_Prop_Lyapunov_2}.

	Then for any $t > 0$, from Definition \ref{Defin_conservative_field_path_int},  we have 
	\begin{equation*}
		\begin{aligned}
			&\phi(x(t), m(t), v(t)) - \phi(x(0), m(0), v(0)) \\
			={}& \int_{0}^t \inf_{d_{\phi} \in \D_{\phi}(x(s), m(s), v(s))}   \inner{ (\dot{x}(s), \dot{m}(s), \dot{v}(s)) ,  d_{\phi}} \mathrm{d}s\\
			\leq{}&  -\int_{0}^t  \varepsilon\delta_{\gamma} \inner{m(s), (|v(s)| + \varepsilon )^{-\gamma-1} \odot m(s)} \mathrm{d} s \leq 0.
		\end{aligned}
	\end{equation*}
	Therefore, for any $s_1>s_2 \geq  0$, we have
	\begin{equation}
		\phi(x(s_1), m(s_1), v(s_1)) \leq  \phi(x(s_2), m(s_2), v(s_2)),
	\end{equation}
	which illustrates that $\phi(x(t), m(t), v(t))$ is non-increasing for any $t > 0$. 
	
	Now we prove that $\phi$ is a Lyapunov function for $\ca{B}$. When $(x, m, v) \notin \ca{B}$, we first consider the cases where $m(0) \neq 0$. From the continuity of the path $(x(t), m(t), v(t))$, there exists $T>0$ such that $m(t) \neq 0$ for any $t \in [0, T]$. Therefore, we can conclude that for any $t>0$, 
	\begin{equation*}
		\begin{aligned}
			&\phi(x(t), m(t), v(t)) - \phi(x(0), m(0), v(0))  \\
			\leq{}& 
			- \int_{0}^t \varepsilon\delta_{\gamma}\inner{m(s), (|v(s)| + \varepsilon )^{-\gamma-1} \odot m(s)}
			\\
			\leq{}&  - \int_{0}^{\min\{t, T\}} \varepsilon\delta_{\gamma}\inner{m(s), (|v(s)| + \varepsilon )^{-\gamma-1} \odot m(s)} <  0.
		\end{aligned}
	\end{equation*}
	
	On the other hand, when $(x(0), m(0), v(0)) \notin \ca{B}$ with $m(0) = 0$, then it holds that $0\notin \D_f(x(0))$. Therefore, by the hyperplane separation theorem, there exists $w \in \Rn$ such that $\norm{w}=1$ and $\inf_{d \in \D_f(x(0))} \inner{d, w} >0$. Then from the outer-semicontinuity of $\D_f$ and the continuity of $m(t)$, there exists a constant $c >0$ and  time $T>0$ such that 
	\begin{equation*}
		\inf_{d \in \D_f(x(t))} \inner{d, w} >c, \quad \text{and} \quad \norm{m(t)} \leq \frac{c}{2}, \quad \forall t \in [0, T]. 
	\end{equation*}
	As a result, for any $t \in [0, T]$, we have 
	\begin{equation*}
		\begin{aligned}
			&\inner{m(t), w} = \int_{0}^t \inner{\dot{m}(s), w} \mathrm{d}s 
			= \int_{0}^t  \tau_1\inner{ m(s) - \D_f(x(s)), w}  \mathrm{d}s
			\\
			\leq{}& \int_{0}^t -\frac{\tau_1 c}{2} \mathrm{d} s = -\frac{\tau_1 c t}{2} < 0. 
		\end{aligned}
	\end{equation*}
	Then for any $t > 0$, it holds that
	\begin{equation*}
		\int_{0}^t \norm{m(s)}^2 \mathrm{d}s \geq \int_{0}^t \inner{m(s), w}^2 \mathrm{d} s > \int_{0}^t \left(\frac{\tau_1 c s}{2}\right)^2 \mathrm{d} s > 0. 
	\end{equation*}
	As a result, let $M:= \sup_{t \in [0, T]} \norm{|v(t)| + \varepsilon}$, we achieve 
	\begin{equation*}
		\begin{aligned}
			&\phi(x(t), m(t), v(t)) - \phi(x(0), m(0), v(0)) \\
			\leq{}& -\int_{0}^{\min\{t, T\}} \varepsilon \delta_{\gamma}\inner{m(s), (|v(s)| + \varepsilon )^{-\gamma-1} \odot m(s)}
			\\
			\leq{}& 
			-
			\frac{\varepsilon \delta_{\gamma}}{M^{\gamma+1}}
			\int_{0}^{\min\{t, T\}}  \norm{m(s)}^2 \mathrm{d}s \; < \; 0.
		\end{aligned}
	\end{equation*}
	Thus we can conclude that 
	\begin{equation*}
		\phi(x(t), m(t), v(t))  <\phi(x(0), m(0), v(0)),
	\end{equation*}
	holds whenever $(x, m, s) \notin \ca{B}$. Hence, $\phi$ is a Lyapunov function for the set $\ca{B}$ with respect to the differential inclusion \eqref{Eq_DI}. This completes the proof. 
\end{proof}

Proposition \ref{Prop_Lyapunov} requires the parameters in the framework \eqref{Eq_framework} to satisfy $(1-\kappa)\gamma\tau_2 \leq 2\tau_1$. When $\kappa$ is $0$, the Assumption \ref{Assumption_alg}(3) becomes $\widetilde{\mathrm{sign}}(v) \odot \ca{U}(x,m,v)\geq 0$ for any $x,m,v \in \Rn$. As demonstrated in Section \ref{Section_4}, this condition can be satisfied by numerous popular Adam-family methods, including Adam, AdaBelief, AMSGrad, NAdam, and Yogi. Moreover, it can be proven that the corresponding set-valued mapping $\ca{U}$ in AMSGrad satisfies Assumption \ref{Assumption_alg}(3) with $\kappa = 1$. Consequently, as discussed later in Corollary \ref{Coro_AMSGrad_Convergence}, AMSGrad converges with any combinations of $\tau_1, \tau_2 >0$ within the framework \eqref{Eq_framework}.

Based on Proposition \ref{Prop_Lyapunov} and \citep[Proposition 3.27]{benaim2005stochastic}, the following theorem illustrates the  global almost sure convergence properties for \eqref{Eq_framework}. 
\begin{theo}
	\label{The_convergence}
	For  any sequence $\{(\xk, \mk, \vk)\}$ generated by \eqref{Eq_framework}. Suppose Assumption \ref{Assumption_f} and Assumption \ref{Assumption_alg} hold, and the parameters in \eqref{Eq_framework} satisfy $(1-\kappa)\gamma\tau_2  \leq 2\tau_1$. Then almost surely,  any cluster point of $\{\xk\}$ lies in $\{x \in \Rn: 0\in \D_f(x)\}$, and the sequence $\{f(\xk)\}$ converges. 
\end{theo}
\begin{proof}
    Firstly, we aim to prove the uniform boundedness of the sequence $\{(d_{x,k}, d_{m,k}, d_{v,k})\}$.
     From Assumption \ref{Assumption_alg}(4), it is easy to verify that $\lim_{k\to +\infty} \mathrm{dist}(\mk - d_{m,k}, \tau_1\D_f^{\delta_k}(\xk)) = 0$. Moreover, Assumption \ref{Assumption_alg}(2) assumes that the sequence of iterates  $\{(\xk, \mk, \vk)\}$ is uniformly bounded almost surely. Then as the local boundedness of $\D_f$ and $\ca{U}$ implies the local boundedness of $\ca{G}$, we can conclude that the sequence $\{(d_{x,k}, d_{m,k}, d_{v,k})\}$ is uniformly bounded almost surely. In addition, as $\D_f$ and $\ca{U}$ are assumed to be convex-valued, it holds that the set-valued mapping $\ca{G}$ is also convex-valued. 
    
    Moreover, Proposition \ref{Prop_UB_martingale_difference_sequence} controls the summation of the noises in the framework \eqref{Eq_framework} (i.e., $\sum_{k = j}^i \theta_k(\xi_{x, k}, \xi_{m, k}, \xi_{v, k})$). From Proposition \ref{Prop_UB_martingale_difference_sequence} and Assumption \ref{Assumption_alg}(5),  for any $T > 0$, almost surely, we have that,
	\begin{equation*}
		\lim_{j \to +\infty} \sup_{j\leq i \leq \Lambda(\lambda_j + T)} \norm{\sum_{k = j}^i \theta_k(\xi_{x, k}, \xi_{m, k}, \xi_{v, k})  } = 0. 
	\end{equation*}
	Together with Assumption \ref{Assumption_alg}(2), we can conclude that for almost every $\omega \in \Omega$, the interpolation for $\{(\xk(\omega), \mk(\omega), \vk(\omega))\}$ is a perturbed solution for the differential inclusion \eqref{Eq_DI}. Then Proposition \ref{Prop_Lyapunov} illustrates that $\phi$ is a Lyapunov function for the set $\ca{B}$ with respect to the differential inclusion \eqref{Eq_DI}. 
	
	Let $M_{m,\omega} := \sup_{k\geq 0} \norm{\mk(\omega)} $, $M_{v,\omega} := \sup_{k\geq 0} \norm{\vk(\omega)}$, $M_{x,\omega} := \sup_{k\geq 0} \norm{\xk(\omega)}$ and ${\ca{C}}_{\omega}:= \bb{B}_{M_{x,\omega}}(0)  \times \bb{B}_{M_{m,\omega}}(0) \times \bb{B}_{M_{v,\omega}}(0)$, then the set $\ca{B} \cap {\ca{C}}_{\omega}$ is a compact set. Then for almost every $\omega \in \Omega$,  \citep[Proposition 3.27]{benaim2005stochastic} illustrates that any cluster point of $\{(\xk(\omega), \mk(\omega), \vk(\omega))\}$ lies in $\ca{B} \cap {\ca{C}}_{\omega}$, and $\{\phi(\xk(\omega), \mk(\omega), \vk(\omega))\}$ converges. From Lemma \ref{Le_stable_set}, we can conclude that for almost every $\omega \in \Omega$, any limit point of $\{\xk(\omega)\}$ lies in $\{x: 0\in \D_f(x)\}$, and $\{f(\xk(\omega))\}$ converges. Hence we complete the proof. 
\end{proof}

\subsection{Convergence to $\partial f$-stationary Points with Random Initialization}

In training nonsmooth neural networks, the conservative fields associated with AD algorithms may introduce infinitely many spurious stationary points \citep{bolte2020mathematical,bolte2021nonsmooth,bianchi2022convergence}. 
To address these issues, several existing works \citet{bolte2020mathematical,bianchi2022convergence} demonstrate that the SGD method can avoid the spurious stationary points introduced by conservative field almost surely. For the SGD method with diminishing stepsizes and mini-batch random sampling, \citet{bolte2020mathematical} prove that it can avoid these spurious stationary points almost surely.  Moreover, \cite{bianchi2022convergence} prove that with randomly chosen initial points and stepsizes, the vanilla SGD method with constant stepsizes converges to a neighborhood of the $\partial f$-stationary points of \ref{Prob_Ori}. Their results guarantee that the vanilla SGD method is able to yield meaningful stationary points in training nonsmooth neural works, regardless of the chosen conservative field.

In this subsection, to establish similar properties for Adam-family methods, we adopt the techniques from \citep{bolte2020mathematical,bianchi2022convergence} and extend their results to analyze the following abstract framework that employs diminishing stepsizes,  
\begin{equation}
	\label{Eq_abstract_AS_framework}
	\zkp -\zk \in - cs_k\ca{Q}_k(\zk, \omega_k),
\end{equation}
where $\zk \in \bb{R}^d$ refers to the iteration points,  $s_k \in \bb{R}$ refers to the stepsizes, $c \in \bb{R}$ is a scaling parameter for the stepsizes, $\Xi$ is a probability space and  $\{\omega_k\} \subset \Xi$ characterizes the stochasticity in \eqref{Eq_abstract_AS_framework}. Moreover, for any $k\geq 0$, $\ca{Q}_k:\bb{R}^d\times \Xi \rightrightarrows \bb{R}^d$ is a set-valued mapping. Furthermore, for almost every $\omega$, we assume the set-valued mapping $\ca{Q}_k(\cdot, \omega):\bb{R}^d \rightrightarrows \bb{R}^d$ is almost everywhere $\ca{C}^1$ for any $k\geq 0$. 

\begin{defin}
	\label{Defin_AS_C1}
    A measurable mapping $q: \bb{R}^d \to \bb{R}^d$ is almost everywhere $\ca{C}^1$ if for almost every $z \in \bb{R}^d$, $q$ is locally continuously differentiable in 
    a neighborhood of $z$. 
 
    Moreover, a set-valued mapping $\ca{Q} : \bb{R}^d \rightrightarrows \bb{R}^d$ is almost everywhere $\ca{C}^1$ if there exists an almost everywhere $\ca{C}^1$ mapping  $q:\bb{R}^d \to \bb{R}^d$ such that for almost every $z \in \bb{R}^d$, $\ca{Q}(z) = \{q(z)\}$. 
\end{defin}

For any $k\geq 0$, let $q_k:\bb{R}^d \times \Xi \to \bb{R}^d$ be the mapping that is locally continuously differentiable for almost every $z \in \bb{R}^d$ and $\ca{Q}_k(z, \omega) = \{q_k(z, \omega)\}$ 
holds for almost every $\omega \in \Xi$. Then we define the mapping  $T_{s,\omega, k}$ as
\begin{equation*}
	T_{s,\omega, k}(z) = z - s q_k(z, \omega).
\end{equation*}
The following proposition illustrates that for almost every $s\in \bb{R}$,  the mapping $T_{s,\omega, k}^{-1}$ maps zero-measure subsets into the zero-measure subsets over $\bb{R}^d$.

\begin{prop}
	\label{Prop_AS}
	Suppose $q_k(\cdot, \omega)$ is almost everywhere $\ca{C}^1$ for almost every $\omega \in \Xi$ and any $k\geq 0$. Then for almost every $\omega \in \Xi$, there exists a full-measure subset $S_k$ of $\bb{R}$ such that for any $s \in S_k$ and any zero-measure set $A \subset \bb{R}^d$, the subset
	\begin{equation}
		\{ z \in \bb{R}^d : T_{s, \omega, k}(z) \in A \}
	\end{equation}
	is zero-measure. 
\end{prop}
\begin{proof}
	From Definition \ref{Defin_AS_C1},  there exists a full-measure subset $\Gamma_{\phi} \subseteq \bb{R}^d$ such that for almost every $\omega \in \Omega$ and any $z \in \Gamma_{\phi}$, $q_k(\cdot, \omega)$ is continuously differentiable in a neighborhood of $z\in \Gamma_{\phi}$.  
	
	For any fixed $k\geq 0$, $\omega \in \Omega$, and $z \in \Gamma_{\phi}$, denote the Jacobian of $q_k$ with respect to $z$ as $J_{q_k}$. Then the Jacobian of $T_{s, \omega, k}$ can be expressed as 
	\begin{equation*}
		J_{T_{s, \omega, k}}(z) = I_n - sJ_{q_k}(z, \omega). 
	\end{equation*} 
	Notice that for any fixed $z \in \Gamma_{\phi}$, $\mathrm{det}(J_{T_{s, \omega, k}}(z))$ is a non-trivial $d$-th order polynomial of $s$ in $\bb{R}$, hence its roots are zero-measure in $\bb{R}$. Therefore, we can conclude that for any $\omega \in \Omega$ and any $z \in \Gamma_{\phi}$, 
	\begin{equation*}
		\{ s \in \bb{R}: \mathrm{det}(J_{T_{s, \omega, k}}(z)) = 0\}
	\end{equation*}
	is zero-measure in $\bb{R}$. From Fubini's theorem, we can conclude that for almost every given $\omega \in \Omega$, there exists a full-measure subset $S_k$ of $\bb{R}$ such that for any $s \in S_k$, 
	\begin{equation}
		\label{Eq_Prop_AS_0}
		\Gamma_{\phi, s, \omega, k} := \{z \in \Gamma_{\phi}: \mathrm{det}(J_{T_{s, \omega, k}}(z)) \neq 0\}
	\end{equation} is full-measure in $\bb{R}^d$.

    By inverse function theorem, for any $k\geq 0$, almost every given $\omega \in \Omega$, $s \in S_k$ and  $z \in \Gamma_{\phi, s, \omega, k}$, the mapping $T_{s, \omega, k}$ is a local diffeomorphism  in a neighborhood of $z$. Let $\tilde{V}_{z}:= \bb{B}_{z}(\tilde{\delta}_z)$, where $\tilde{\delta}_z >0$ and $T_{s, \omega, k}$ is a local diffeomorphism in $\ca{B}_{z}(\tilde{\delta}_z)$ (i.e., $T_{s, \omega, k}$ is continuously differentiable and has non-singular Jacobian over $\ca{B}_{z}(\tilde{\delta}_z)$). Therefore, $V_z := T_{s, \omega, k}(\ca{B}_{z}(\tilde{\delta}_z))$ is an open set for any $z\in \Gamma_{\phi, s, \omega, k}$. 
    
    Notice that $\{V_z\}_{z \in \Gamma_{\phi, s, \omega, k}}$ is an open cover for $\Gamma_{\phi, s, \omega, k}$. 
	Based on Lindelof's lemma \citep{kelley2017general} and the fact that $\bb{R}^d$ is a second-countable space, there exists $\{z_i\}_{i\in \bb{N}_+} \subset \Gamma_{\phi, s, \omega, k}$ such that $\Gamma_{\phi, s, \omega, k}  \subseteq \bigcup_{i \in \bb{N}_+} V_{z_i}  \subseteq \bigcup_{i \in \bb{N}_+} \tilde{V}_{z_i}$. Given any zero-measure set $A \subset \bb{R}^d$, for any $i \in \bb{N}_+$, since $T_{s, \omega, k}$ is a local diffeomorphism in $\tilde{V}_{z_i}$, we can conclude that the set $\{ z \in \tilde{V}_{z_i}: T_{s, \omega, k}(z) \in A \}$ 
	is zero-measure. Then the set 
	\begin{equation*}
		\bigcup_{i\in \bb{N}_+} \{ z \in \tilde{V}_{z_i}: T_{s, \omega, k}(z) \in A \}
	\end{equation*}
	is zero-measure, hence the set $\{ z \in \Gamma_{\phi, s, \omega, k}: T_{s, \omega, k}(z) \in A  \}$ is zero-measure. Combined with the fact that $\Gamma_{\phi, s, k}$ is a full-measure subset of $\bb{R}^d$, we can conclude that  the set $\{ z \in \bb{R}^d : T_{s, \omega, k}(z) \in A  \}$ satisfies 
	\begin{equation*}
		\{ z \in \bb{R}^d : T_{s, \omega, k}(z) \in A  \} \subset
		\Gamma_{\phi, s, \omega, k}^c \cup  \{ z \in \Gamma_{\phi, s, \omega, k}: T_{s, \omega, k}(z) \in A  \}.
	\end{equation*}
	 Then from the definition of $\Gamma_{\phi, s, \omega, k}$ in \eqref{Eq_Prop_AS_0}, for almost every $\omega \in \Xi$,  for any $s \in S_k$, the set $\Gamma_{\phi, s, \omega, k}^c$ is zero-measure. 
	This completes the proof. 
\end{proof}

Based on Proposition \ref{Prop_AS}, the following proposition illustrates that for almost every prefixed $\omega \in \Xi$ and $s \in \bb{R}_+$, the mapping $T_{s,\omega, 0}^{-1} \circ T_{s,\omega, 1}^{-1} \circ \cdots T_{s,\omega, k}^{-1}$ maps zero-measure subsets of $\bb{R}^d$ into zero-measure subsets for all $k\geq 0$. 	
\begin{prop}
	\label{Prop_AS_multistep}
	Suppose $\ca{Q}_k(\cdot, \omega)$ is almost everywhere $\ca{C}^1$ for almost every $\omega \in \Xi$ and any $k\geq 0$. Then for any given zero-measure set $A \subset \bb{R}^d$, for almost every $\omega \in \Xi$, there exists a full-measure subset $S_{init, \omega}$ of $\bb{R}^d \times \bb{R}$ such that for any sequence $\{z_k\}$ that employ the following update scheme
	\begin{equation}
		\label{Eq_Prop_AS_multistep_0}
		z_{k+1} = z_k - s \ca{Q}_k(z_k, \omega),
	\end{equation}
	with $(z_0, s) \in S_{init, \omega}$, 
	we have that 
	\begin{equation*}
		\{z_k: k\geq 0\} \bigcap A = \emptyset. 
	\end{equation*}
\end{prop}
\begin{proof}
	According to Proposition \ref{Prop_AS}, given any zero-measure subset $A \subset \bb{R}^d$, for any $k\geq 0$ and $\omega \in \Xi$, there exists a full-measure subset $S_{\omega,k} \subseteq \bb{R}_+$ such that for any $s \in S_{\omega,k}$ and any zero-measure subset $A$ of $\bb{R}^d$, the set $\{z \in \bb{R}^d: z - s\ca{Q}_k(z, \omega) \cap A \neq \emptyset  \}$ is a zero-measure subset of $\bb{R}^d$.
	
	Then let $S_{\omega} := \cup_{k\geq 0} \ca{S}_{\omega,k}$, it is easy to verify that for any $s \in S_{\omega}$, the set $\{z \in \bb{R}^d: z - s\ca{Q}_i(z) \cap A \neq \emptyset \text{ for some $i \geq 0$} \}$ is a zero-measure subset of $\bb{R}^d$. As a result, let $\tilde{Y}_{0, \omega, s} = A$, and recursively define
	\begin{equation*}
		\tilde{Y}_{k+1, \omega, s} = \{z \in \bb{R}^d: (z - s\ca{Q}_i(z, \omega)) \cap (\tilde{Y}_{k, \omega, s} \cup A) \neq \emptyset \text{ for some $i \geq 0$} \}.
	\end{equation*}
	Then Proposition \ref{Prop_AS} illustrates that $\tilde{Y}_{k, \omega, s}$ is a zero-measure subset of $\bb{R}^d$ for any $k\geq 0$.
	Moreover, from the definition of $\tilde{Y}_{k,\omega, s}$, we can conclude that for any $ j\geq 0$, any $s \in S_{\omega}$,  and for any sequence $\{z_k\}$ that follows equation \eqref{Eq_Prop_AS_multistep_0} with an initial condition $z_0 \notin \tilde{Y}_{j, \omega, s}$, it holds true that $\{z_k: k\leq j\} \cap A = \emptyset$.

	 Let $Y_{\omega, s} = (\cup_{k\geq 0} \tilde{Y}_{k,\omega, s})^c$, then  for any $k\geq 0$, any $s \in S_{\omega}$, and any $z_0 \in Y_{\omega, s}$, we have that $Y_{\omega, s}$ is a zero-measure subset of $\bb{R}^d$ and the sequence $\{z_k\} \cap  (\tilde{Y}_{k, \omega, s} \cup A) = \emptyset$. Then from Fubini's theorem, the subset $\{(z_0, s): s \in S_{\omega}^c \text{ or } z_0 \in Y_{\omega, s}^c\}$ is a zero-measure subset of $\bb{R}^d \times \bb{R}_+$. As a result,  $S_{init, \omega}=\{(z_0, s): s \in S_{\omega}, z_0 \in Y_{\omega, s}\}$ is a full-measure subset of $\bb{R}^d \times \bb{R}_+$. Moreover, from the choices of $S_{\omega}$ and $Y_{\omega, s}$,  for any $(z_0, s) \in S_{init, \omega}$, the sequence $\{z_k\}$ that follows the scheme in \eqref{Eq_Prop_AS_multistep_0} satisfies $\{z_k\} \cap A = \emptyset$. This completes the proof. 
\end{proof}

Next, we present the following theorem to illustrate that under mild assumptions, with random initialization for the initial point $z_0$ and the scaling parameter $c$,  the sequence generated by \eqref{Eq_abstract_AS_framework} can avoid any zero-measure subset $A$ of $\bb{R}^d$. As shown later in Section \ref{Section_4} and Section \ref{Section_5}, Theorem \ref{The_convergence_AS} directly implies the almost sure convergence to $\partial f$-stationary points of \ref{Prob_Ori} for the analyzed stochastic subgradient methods.

\begin{theo}
	\label{The_convergence_AS}
	Suppose $\ca{Q}_k(\cdot, \omega)$ is almost everywhere $\ca{C}^1$ for almost every $\omega \in \Xi$ and any $k\geq 0$. Then for any zero-measure subset $A \subset \bb{R}^d$, there exists a full-measure subset $S_{init} \subseteq \bb{R}^d \times \bb{R}$ and $S_{\omega} \subseteq \Omega$, such that for any $(z_0, c) \in S_{init}$, almost surely in $\Xi$, it holds that the sequence $\{\zk\}$ generated by \eqref{Eq_abstract_AS_framework} satisfies $\{\zk\} \subset A^c$. 
\end{theo}
\begin{proof}
	As illustrated in Proposition \ref{Prop_AS_multistep}, for almost every $\omega \in \Xi$, there exists a full-measure subset $S_{init, \omega}$ of $\bb{R}^d \times \bb{R}_+$ such that for any $(z_0, c) \in S_{init, \omega}$, almost surely in $\Xi$, it holds that the sequence $\{\zk\}$ generated by \eqref{Eq_abstract_AS_framework} satisfies $\{\zk\} \subset A^c$. 
	
	Notice that for almost every $\omega \in \Xi$, the set $S_{init,\omega}$ is a full-measure subset of $\bb{R}^d \times \bb{R}_+$. Then by applying Fubini's theorem, the set $T = \{(\omega, z_0, s): \omega \in \Xi, ~ (z_0,s) \in S_{init,\omega}\}$ is a full-measure subset of $\Xi \times \bb{R}^d \times \bb{R}_+$. Then by applying Fubini's theorem again, we can conclude that there exists a full-measure subset $S_{init}$ of $\bb{R}^d$ such that for any $(z_0, s) \in S_{init}$, the subset $\{\omega: (\omega, z_0, s) \in T\}$ is a full-measure subset of $\Xi$. This completes the proof.  
\end{proof}

\section{Applications: Convergence Guarantees for Adam-family Methods}
\label{Section_4}

In this section, we establish the convergence properties of ADAM, AMSGrad, Yogi and AdaBelief for solving \ref{Prob_Ori} based on our proposed framework when the objective function $f$ takes the following form, 
\begin{equation}
	f(x) := \bb{E}_{s \sim P} [f_{s}(x)].
\end{equation}
Here $(\Theta, \ca{F}, P)$ is a probability space, where $\Theta$ refers to the sample space, $\ca{F}$ is the corresponding $\sigma$-algebra, and $P$ is the probability distribution. 
Throughout this section, we make the following assumptions on $f$. 
\begin{assumpt}
	\label{Assumption_Implementation_f}
            There exists a measurable set-valued mapping $\D: \Rn \times \Theta \to \Rn$ such that 
            \begin{enumerate}
                \item The mapping $ (x, s) \mapsto f_s(x)$ is measurable over $\Rn \times \Theta$;
                \item For almost every $s \in \Rn$, $x \mapsto \D(x, s)$ is a definable conservative field that admits $f_s$ as its potential function. Moreover, there exists a measurable mapping $\chi: \Rn \times \Theta \to \Rn$ such that $\chi(x, s) \in \D(x, s)$ for all $x \in \Rn$ and almost every $s \in \Theta$. 
                \item  There exists an integrable function $p_{\Theta}:\Theta \to \bb{R}_+$ such that for all $x \in \Rn $ and $s \in \Theta$, it holds that
                \begin{equation*}
                    \sup_{d \in \D(x, s)}  \norm{d} \leq p_{\Theta}(s). 
                \end{equation*}
                \item The set $\{ f(x): 0 \in \conv\left( \bb{E}_{s \sim P}[\D(x, s)] \right) \}$ has empty interior in $\bb{R}$.
          \end{enumerate}
		
\end{assumpt} 

Based on the results from \citep[Theorem 3.10]{bolte2022subgradient}, the integral of $\D$ with respect to the measure $P$ (i.e., $\bb{E}_{s \sim P}[\D(x, s)]$ as defined in Definition \ref{Defin_Aumann_integral}) is a conservative field that admits $f$ as its potential function. As a result, in this section, we choose the conservative field $\D_f$ in the framework \eqref{Eq_framework} as 
\begin{equation*}
    \D_f(x) = \conv\left(\bb{E}_{s \sim P}[\D(x, s)] \right). 
\end{equation*}
Moreover, the mapping $\chi$ defined in Assumption \ref{Assumption_Implementation_f}(2) is called a (measurable) selection of the set-valued mapping $\D$. From Definition \ref{Defin_Aumann_integral}, it is easy to verify that $\bb{E}_{s \sim P} [\chi(x, s)] \in \D_f(x)$ holds for all $x \in \Rn$.

\begin{rmk}
	\label{Rmk_Application_finite_sum}
	It is worth mentioning that Assumption \ref{Assumption_Implementation_f}  is mild in practice.  For the loss function of any neural network that is built from definable blocks,  \citet{bolte2021conservative} show that the results returned by the AD algorithms are contained within a definable conservative field. This illustrates that Assumption \ref{Assumption_Implementation_f}(1)-(3) are easy to be satisfied in practice. 

    Moreover, \citet[Theorem 5]{bolte2021conservative} illustrates that the set $\{f(x): 0\in \D_{f}(x)\}$ is finite whenever both $f$ and  $\D_f$ are definable. As discussed in \citep{bolte2021conservative}, when the set $\Theta$ is finite, the $f$ can be expressed in a finite-sum formulation. Under such settings, both $f$ and $\D_f$ are definable whenever both $f_s$ and $\D(\cdot, s)$ are definable for any $s \in \Theta$. On the other hand, for the cases where $\Theta$ contains infinitely many elements,    \citet[Theorem 4.8]{bolte2022subgradient} guarantees the definability of $f$ and $\D_f$ under appropriate conditions. In particular, \citet{bolte2022subgradient} shows that when we assume $\Theta = \bb{R}^q$ for some $q >0$, the probability measure $P$ has a semi-algebraic density function, and $\D$ is assumed to be convex-valued and semi-algebraic, then $\D_f$ is semi-algebraic. These results illustrate that Assumption \ref{Assumption_Implementation_f}(4) is also mild in practice. 
	
\end{rmk}

Inspired by the pioneering works \citep{barakat2021convergence,barakat2021stochastic,gadat2022asymptotic}, we consider a class of Adam-family methods with diminishing stepsizes for minimizing $f$ over $\Rn$.  The detailed algorithm is presented in Algorithm \ref{Alg:ADAM}.  In step 6 of Algorithm \ref{Alg:ADAM}, different Adam-family methods employ different schemes for updating the estimator $\{\vk\}$, which is characterized by a specific mapping $R_{\ca{U}}: \Rn \times \Rn \times \Rn \to \Rn$. Table \ref{Table_Adaptive_algs} summarizes the updating rules for Adam, AdaBelief, AMSGrad, NAdam and Yogi, their corresponding set-valued mappings $\ca{U}$ in the framework \eqref{Eq_framework}, and the settings for the parameters $\alpha$ and $\kappa$. 
\begin{algorithm}[htbp]
	\begin{algorithmic}[1]  
		\Require Initial point $x_0 \in \Rn$, $m_0 \in \Rn$ and $v_0 \in \Rn_+$, parameters $\alpha \geq 0$, and $\tau_1, \tau_2, \varepsilon > 0$, and $\chi$ as a selection  of stochastic subgradients;
		\State Set $k=0$;
		\While{not terminated}
		\State Independently sample $s_k \sim P$, and compute $g_k = \chi(\xk, s_k)$;
		\State Choose the stepsize $\eta_k$;
		\State Update the momentum term by $m_{k+1} = (1-\tau_1  \eta_k) m_k + \tau_1   \eta_k g_k$;
		\State Update the estimator $\vkp$ from $g_k$, $\mkp$ and $\vk$ by 
            \begin{equation*}
                \vkp = \vk - \tau_2\eta_k R_{\ca{U}}(g_k, \mkp, \vk). 
            \end{equation*}
		\State Compute the scaling parameters  $\rho_{m,k+1}$ and $\rho_{v,k+1}$;
		\State Update $\xk$ by 
		\begin{equation*}
			\xkp = \xk - \eta_k (\rho_{v,k+1}|v_{k+1}| + \varepsilon )^{-\frac{1}{2}} \odot ( \rho_{m,k+1} m_{k+1} +  \alpha g_k );
		\end{equation*}
		\State $k = k+1$;
		\EndWhile
		\State Return $\xk$.
	\end{algorithmic}  
	\caption{Stochastic subgradient-based Adam for nonsmooth optimization problems.}  
	\label{Alg:ADAM}
\end{algorithm}

To establish the convergence properties for Algorithm \ref{Alg:ADAM}, we make some mild assumptions  as follows. 
\begin{assumpt}
	\label{Assumption_Implementation}
	\begin{enumerate}
		\item The sequence $\{\xk\}$ is almost surely bounded. That is, 
		\begin{equation*}
			\begin{aligned}
				\sup_{k\geq 0} \norm{\xk}  <+\infty
			\end{aligned}
		\end{equation*}
		holds almost surely. 
		\item The sequence of stepsizes $\{\eta_k\}$ is positive and satisfies 
		\begin{equation*}
			\sum_{k = 0}^{+\infty} \eta_k = +\infty, \quad \lim_{k \to +\infty} \eta_k \log(k) = 0. 
		\end{equation*} 
		\item The scaling parameters $\{\rho_{m,k}\}$ and $\{\rho_{v,k}\}$ satisfy 
		\begin{equation*}
			\lim_{k \to +\infty} \rho_{m,k} = 1, \quad \lim_{k \to +\infty} \rho_{v,k} = 1.
		\end{equation*}
            \item There exists a constant $M_{\Theta}>0$ such that $p_{\Theta}(s) \leq M_{\Theta}$ holds for almost every $s \in \Theta$. Here $p_{\Theta}(s)$ is the auxiliary function  defined in Assumption \ref{Assumption_Implementation_f}(3). 
	\end{enumerate}
\end{assumpt}

\begin{table}[tb]
	\centering
	\tiny
	\begin{tabular}{c|c|c|c}
		\hline
		Adam-family methods & Expression of $R_{\ca{U}}$ for updating $\{\vk\}$ & Corresponding $\ca{U}$ in \eqref{Eq_framework} & $\alpha$ and $\kappa$ \\ \hline
		Adam \citep{kingma2014adam}& $R_{\ca{U}}(g,m,v) = v - g^2$ & $ \mathrm{sign}(v) \odot \ca{S}(x)$ & $\alpha = 0, \kappa = 0$\\ \hline
		AdaBelief \citep{zhuang2020adabelief} & $R_{\ca{U}}(g,m,v) = v - (g-m)^2$ & $ \mathrm{sign}(v) \odot \tilde{\ca{S}}(x, m, v)$ & $\alpha = 0, \kappa = 0$\\ \hline
		AMSGrad \citep{reddi2019convergence}& $R_{\ca{U}}(g,m,v) = v - \max\{v, g^2\}$ & $\mathrm{sign}(v) \odot \left( \bb{E}_{s\sim P} [\max\{|v|, \ca{S}(x, s)\}]\right)$ & $\alpha = 0, \kappa = 1$\\ \hline
		NAdam \citep{dozat2016incorporating}& $R_{\ca{U}}(g,m,v) = v - g^2$ & $ \mathrm{sign}(v) \odot \ca{S}(x)$ & $\alpha > 0, \kappa =0$\\ \hline
		Yogi \citep{zaheer2018adaptive}& $R_{\ca{U}}(g,m,v) = v - \mathrm{sign}(v - g^2) \odot g^2$ & $\mathrm{sign}(v) \odot 
		\left( \bb{E}_{s\sim P} [ \{ |v| - \mathrm{sign}(|v| - d)\odot d: d \in \ca{S}(x, s) \} ] \right)$ & $\alpha = 0, \kappa = 0$\\\hline
	\end{tabular}
	
	\caption{Different updating schemes for $\{\vk\}$ in Step 6 of Algorithm \ref{Alg:ADAM} that describe Adam, AdaBelief, AMSGrad, NAdam and Yogi. Here $\ca{S}(x, s) := \conv(\left\{ d\odot d: d\in \D(x, s) \right\})$, $\ca{S}(x) := \bb{E}_{s \sim P}[ \ca{S}(x, s)]$, and  $\tilde{\ca{S}}(x, m, v):= \bb{E}_{s \sim P}[\{ (d-m)^2: d \in \D(x, s) \}]$.}
	\label{Table_Adaptive_algs}
\end{table}

\begin{rmk}
    For the set-valued mapping $\ca{S}$ in Table \ref{Table_Adaptive_algs}, it is worth mentioning that under Assumption \ref{Assumption_Implementation}(4), for any $x \in \Rn$ and almost every $s \in \Theta$, we have $\sup_{d \in \ca{S}(x, s)} \norm{d} \leq (p_{\Theta}(s)  )^2 $, and $p_{\Theta}(x)^2$ is integrable in $\Theta$. Therefore, based on \citep[Theorem 2]{shapiro2007uniform}, we can conclude that $\bb{E}_{s \sim P}[ \ca{S}(x, s)]$ has closed graph and is compact valued. Hence the set-valued mapping $\ca{U}$ corresponding to Adam has closed graph and is compact valued. Similarly, based on \citep[Theorem 2]{shapiro2007uniform} and Assumption \ref{Assumption_Implementation}(4), it is easy to verify that all the set-valued mappings $\ca{U}$ listed in Table \ref{Table_Adaptive_algs} has closed graph and is compact valued, and $\bb{E}_{s\sim P} [R_{\ca{U}}(\D(x, s),m, v)] \in v - \ca{U}(x,m, v )$ holds for any $(x,m,v) \in \Rn \times \Rn \times \Rn_+$.  
\end{rmk}

The following lemma illustrates that the sequence $\{(\xk, \mk, \vk)\}$ is uniformly bounded under Assumption \ref{Assumption_Implementation}. 
\begin{lem}
	\label{Le_UB_ADAM}
	For any sequence $\{\xk\}$ generated by Algorithm \ref{Alg:ADAM}. Suppose Assumption \ref{Assumption_Implementation_f} and Assumption \ref{Assumption_Implementation} hold, and the sequence $\{\vk\}$ follows the schemes in Table \ref{Table_Adaptive_algs}. Then almost surely,  it holds that 
	\begin{equation*}
		\begin{aligned}
			\sup_{k\geq 0} \norm{\xk} +\norm{\mk} +\norm{\vk}  <+\infty.
		\end{aligned}
	\end{equation*}
\end{lem}
\begin{proof}
        Based on Assumption \ref{Assumption_Implementation_f}(3), Assumption \ref{Assumption_Implementation}(1) and Assumption \ref{Assumption_Implementation}(4), we have
        \begin{equation} 
            \label{Eq_Le_UB_ADAM_0}
            \sup_{k\geq 0} \norm{g_k} \leq \sup_{k\geq 0}\sup_{d \in \D(\xk, s_k)} \norm{d} \leq \sup_{k\geq 0}p_{\Theta}(s_k) \leq M_{\Theta}. 
        \end{equation}
	From the updating rule in Algorithm \ref{Alg:ADAM}, it holds for any $k\geq 0$ that 
	\begin{equation}
		\label{Eq_Le_UB_ADAM_1}
		\norm{\mkp} \leq (1-\tau_1  \eta_k) \norm{\mk} + \tau_1   \eta_k \norm{g_k} \leq \max\left\{\norm{m_0}, \sup_{k \geq 0}\norm{g_k}\right\}. 
	\end{equation} 
	Therefore, we can conclude that $\sup_{k \geq 0}\norm{\mk} <+\infty$. 
	
	Next, we prove that the sequence $\{\vk\}$ is uniformly bounded for all the updating schemes in Table \ref{Table_Adaptive_algs}. 
	
	{\bf Adam and NAdam:}
	For any $k\geq 0$, it holds that 
	\begin{equation*}
		\norm{\vkp} \leq (1-\tau_2 \eta_k) \norm{\vk} + \tau_2  \eta_k \norm{g_k^2} \leq \max\left\{\norm{v_0}, \sup_{k \geq 0}\norm{g_k^2}\right\}.
	\end{equation*}
	Therefore, we can conclude that $\sup_{k \geq 0}\norm{\vk} <+\infty$.
	
	{\bf AdaBelief:} 
	For any $k\geq 0$, it holds that 
	\begin{equation*}
		\norm{\vkp} \leq (1-\tau_2 \eta_k) \norm{\vk} + \tau_2  \eta_k \norm{(g_k-\mkp)^2} \leq \max\left\{\norm{v_0}, \sup_{k \geq 0}\norm{(g_k-\mkp)^2}\right\}.
	\end{equation*}
	Therefore, we can conclude that $\sup_{k \geq 0}\norm{\vk} <+\infty$.
	
	{\bf AMSGrad:}
	For any $k\geq 0$, it holds that 
	\begin{equation*}
		\begin{aligned}
		    \sup_{k \geq 0} \norm{\vkp} = \norm{  \vk + \tau_2 \eta_k \max\{ 0, g_k^2 - \vk \}  } \leq  \max\{ \norm{v_0} , \sup_{k \geq 0}\norm{g_k^2} \} <+\infty. 
		\end{aligned}     
	\end{equation*}
	
	{\bf Yogi:}
	For any $k\geq 0$, it holds that
	\begin{equation*}
		\norm{\vkp} \leq \max\left\{ \norm{\vk} , (1+\tau_2 \eta_k)\norm{g_k^2 } \right\}.
	\end{equation*}
	Therefore, we can conclude that 
	\begin{equation*}
		\sup_{k \geq 0} \norm{\vkp} \leq \max\left\{ \norm{v_0} , \sup_{k \geq 0}  (1+\tau_2 \eta_k)\norm{g_k^2 } \right\} <+\infty. 
	\end{equation*}

	Combined with the above inequalities, we can conclude that for any of the updating schemes in Table \ref{Table_Adaptive_algs}, it holds that 
	\begin{equation*}
		\sup_{k\geq 0} \norm{\xk} +\norm{\mk} +\norm{\vk}  <+\infty.
	\end{equation*}
	This completes the proof. 
\end{proof}

Next, we establish the convergence properties for Algorithm \ref{Alg:ADAM} by relating it to the framework \eqref{Eq_framework}. The following corollary demonstrates that Algorithm \ref{Alg:ADAM} fits the framework \eqref{Eq_framework} when choosing the updating scheme for the estimators $\{\vk\}$ specified in Table \ref{Table_Adaptive_algs}. Consequently, the convergence properties of Algorithm \ref{Alg:ADAM} directly follow from Theorem \ref{The_convergence}.
\begin{coro}
	\label{Coro_ADAM_Convergence}
	For any sequence $\{\xk\}$ generated by Algorithm \ref{Alg:ADAM}. Suppose Assumption \ref{Assumption_Implementation_f} and Assumption \ref{Assumption_Implementation} hold, the sequence $\{\vk\}$ follows the schemes in Table \ref{Table_Adaptive_algs}, and $(1-\kappa)\tau_2 \leq 4\tau_1$. Then almost surely,  every cluster point of $\{\xk\}$ is a $\D_f$-stationary point of $f$ and the sequence $\{f(\xk)\}$ converges. 
\end{coro}
\begin{proof}
	We first check the validity of Assumption \ref{Assumption_alg}. Lemma \ref{Le_UB_ADAM} implies that Assumption \ref{Assumption_alg}(2) holds. Moreover, as discussed in \citep[Theorem 3.10]{bolte2022subgradient}, $\D_f$ is a conservative field that admits $f$ as its potential function. 
	Then it follows from Lemma \ref{Le_closed_graph} that $\ca{U}$ has closed graph  and is locally bounded, hence Assumption \ref{Assumption_alg}(3) holds with $\ca{U}$ chosen as the formulations in Table \ref{Table_Adaptive_algs} and $\gamma = \frac{1}{2}$. In addition, Lemma \ref{Le_UB_ADAM} illustrates that $\norm{\mkp - \mk} + \norm{\vkp - \vk}$ converges to $0$ as $k$ goes to infinity. 
	Then combined with Assumption \ref{Assumption_Implementation}(3), we can conclude that  Assumption \ref{Assumption_alg}(4) holds. Furthermore, Assumption \ref{Assumption_alg}(5) directly follows from \eqref{Eq_Le_UB_ADAM_0}, and Assumption \ref{Assumption_alg}(6)  directly follows from Assumption \ref{Assumption_Implementation}(2) by choosing $\theta_k = \eta_k$. 
 
    On the other hand, as the results in \citep[Theorem 3.10]{bolte2022subgradient} show that under Assumption \ref{Assumption_Implementation_f}(1)-(3), $f$ is a potential function that admits $\D_f$ as its conservative field. Then the validity of Assumption \ref{Assumption_f} directly follows from Assumption \ref{Assumption_Implementation_f}.

	Therefore, we can conclude that Algorithm \ref{Alg:ADAM} follows the framework in \eqref{Eq_framework}. Then from Theorem \ref{The_convergence}, we can conclude that almost surely,  every cluster point of $\{\xk\}$ is a $\D_f$-stationary point of $f$ and the sequence $\{f(\xk)\}$ converges. This completes the proof. 
\end{proof}

Since $\kappa$ can be set to $1$ in AMSGrad, the following corollary illustrates that AMSGrad converges with any combination of the parameters $\tau_1, \tau_2 >0$. This improves the results in \citet{reddi2019convergence}, where $f$ is assumed to be differentiable, while the stepsizes are chosen as  $\eta_k = \ca{O}(k^{-\frac{1}{2}})$. 
\begin{coro}
	\label{Coro_AMSGrad_Convergence}
	For any sequence $\{\xk\}$ generated by Algorithm \ref{Alg:ADAM} with AMSGrad updating scheme in Table \ref{Table_Adaptive_algs}. Suppose Assumption \ref{Assumption_Implementation_f} and Assumption \ref{Assumption_Implementation} hold. Then almost surely,  every cluster point of $\{\xk\}$ is a $\D_f$-stationary point of $f$ and the sequence $\{f(\xk)\}$ converges. 
\end{coro}

Finally, the following corollary demonstrates that under mild assumptions, with almost every initial point and stepsize in Algorithm \ref{Alg:ADAM}, the generated sequence $\{\xk\}$ can find the stationary points in the sense of Clarke subdifferential almost surely. Therefore, although AD algorithms may introduce spurious stationary points for $f$, Algorithm \ref{Alg:ADAM} can avoid such spurious stationary points for almost every choice of initial points and stepsizes. 

\begin{coro}
    \label{Coro_ADAM_Convergence_AS}
    Suppose Assumption \ref{Assumption_Implementation_f} holds. Moreover,  for any sequence $\{\xk\}$ generated by Algorithm \ref{Alg:ADAM} with the update schemes in Table \ref{Table_Adaptive_algs}, we assume that 
    \begin{enumerate}
        \item There exists a prefixed positive sequence $\{\nu_{k}\}$ and parameters $0<c_{\min}<c_{\max}$, such that the stepsizes $\{\eta_k\}$ in Algorithm \ref{Alg:ADAM} are set as $\eta_k = c\nu_k$ for any $k\geq 0$ with some $c \in (c_{\min},c_{\max})$.
        \item There exists a non-empty open subset $K \subset \Rn \times \Rn \times \Rn_+$ such that Assumption \ref{Assumption_Implementation} holds with any $(x_0, m_0, v_0, c) \in K \times (c_{\min},c_{\max})$.
    \end{enumerate}
    Then  for almost every $(x_0, m_0, v_0,c) \in K \times  (c_{\min}, c_{\max})$,  it holds almost surely that every cluster point of $\{\xk\}$ is a $\partial f$-stationary point of $f$  and the sequence $\{f(\xk)\}$ converges. 
\end{coro}

\begin{proof}
	Let  $A = \{x \in \Rn: \D_f(x) \neq \partial f(x)\}$. Then it holds from \citet{bolte2021conservative} that $\A$ is a zero-measure subset of $\Rn$. Therefore, the set $\{(x, m, v) \in \Rn\times \Rn \times \Rn_+: x \in A\}$ is zero-measure in $\Rn\times \Rn \times \Rn_+$. 

    Let 
    \begin{equation*}
        \small
        Q_k^{(1)}(x, m, v, s_k) = 
        \left[\begin{matrix}
            0\\
            \tau_1 m - \tau_1  \D(x, s_k)\\
            0\\
        \end{matrix}\right], \quad Q_k^{(2)}(x, m, v, s_k) = 
        \left[\begin{matrix}
            0\\
            0\\
             \tau_2 R_{\ca{U}}(\D(x, s_k), m, v) 
            \\
        \end{matrix}\right],
    \end{equation*}
    and 
    \begin{equation*}
        Q_k^{(3)}(x, m, v, s_k) = 
        \left[\begin{matrix}
            (\rho_{v, k+1}|v| + \varepsilon)^{-\frac{1}{2}} \odot (\rho_{m, k+1} m + \alpha \D(x, s_k))\\
            0\\
            0\\
        \end{matrix}\right].
    \end{equation*}
 For almost every $s \in \Theta$, $f_s$ is a definable function that admits $x\mapsto \D(x, s)$ as its conservative field, \citet[Theorem 4]{bolte2021conservative} illustrates that the set-valued mapping $x\mapsto \D(x, s)$ is almost everywhere $\ca{C}^1$. Moreover, from the update schemes from Table \ref{Table_Adaptive_algs}, we can conclude that all the listed $R_{\ca{U}}$ is semi-algebraic.  Therefore, the set-valued mappings $\ca{Q}_k^{(1)}$,  $\ca{Q}_k^{(2)}$ and $\ca{Q}_k^{(3)}$ are almost everywhere $\ca{C}^1$ for almost every $s \in \Theta$ and any $k\geq 0$. Moreover, Algorithm \ref{Alg:ADAM} can be expressed as 
 \begin{equation*}
     \begin{aligned}
         &\left(x_{k+\frac{1}{3}}, m_{k+\frac{1}{3}}, v_{k+\frac{1}{3}} \right) \in (\xk, \mk, \vk) - \eta_k Q_k^{(1)}(\xk, \mk, \vk, s_k),\\
         &\left(x_{k+\frac{2}{3}}, m_{k+\frac{2}{3}}, v_{k+\frac{2}{3}} \right) 
         \in (x_{k+\frac{1}{3}}, m_{k+\frac{1}{3}}, v_{k+\frac{1}{3}}) - \eta_k Q_k^{(2)}\left(x_{k+\frac{1}{3}}, m_{k+\frac{1}{3}}, v_{k+\frac{1}{3}}, s_k\right),\\
         &(\xkp, \mkp, \vkp) \in \left(x_{k+\frac{2}{3}}, m_{k+\frac{2}{3}}, v_{k+\frac{2}{3}}\right) - \eta_k  Q_k^{(3)}\left(x_{k+\frac{2}{3}}, m_{k+\frac{2}{3}}, v_{k+\frac{2}{3}}, s_k\right). 
     \end{aligned} 
 \end{equation*}
 
 From Theorem \ref{The_convergence_AS} we can conclude that there exists a full-measure subset $S_{init}$ of $K\times [c_{\min}, c_{\max}]$ such that for any $(x_0, m_0, v_0, c) \in S_{init}$, almost surely, it holds that $\{ (\xk, \mk, \vk): k=0, \frac{1}{3}, \frac{2}{3}, 1,  ... \} \subset A^c$, implying that $ \bb{E}_{s_k \sim P}[g_k] = \bb{E}_{s_k \sim P}[\chi(\xk, s_k)]  \in \partial f(\xk)$ holds for any $k\geq 0$. Therefore, by fixing $\D_f$ as $\partial f$ in Theorem \ref{The_convergence}, we can conclude that every cluster point of $\{\xk\}$ is a $\partial f$-stationary point of $f$ and the sequence $\{f(\xk)\}$ converges almost surely.  This completes the proof. 
\end{proof}

\section{Applications: Gradient Clipping for Stochastic Subgradient Methods}
\label{Section_5}
In this section, we present stochastic subgradient methods with gradient clipping technique to solve \ref{Prob_Ori}, under the assumption that the evaluation noises for subgradients are heavy-tailed. Then we prove the convergence properties for our proposed methods based on the framework \eqref{Eq_framework}. 

For a clearer and more comprehensive presentation of our proposed methods, we follow the settings and notations in Section \ref{Section_4} throughout this section. In particular, we assume the objective function $f$ in \ref{Prob_Ori} can be expressed as 
\begin{equation*}
    f(x) = \bb{E}_{s \sim P}[f_{s}(x)],
\end{equation*}
where $(\Theta, \ca{F}, P)$ is the probability space.

For any given compact and convex subset $\ca{S} \subset \Rn$ such that $0$ lies in its interior, we define the clipping mapping $\mathrm{Clip}_{(\cdot)}(\cdot): \bb{R}_+ \times \Rn \to \Rn$ as,
\begin{equation}
	\label{Eq_defin_clipping_mapping}
	\mathrm{Clip}_C(g) := \mathop{\arg\min}_{x \in C \ca{S}} \norm{x -  g},
\end{equation}
where $C\ca{S} = \{ C s\,:\, s\in \ca{S}\}$.
Intuitively, the clipping mapping avoids the extreme values in evaluating the gradients by restricting them in a compact region, and hence helps enforce the convergence of stochastic subgradient methods with heavy-tailed evaluation noises.

The explicit expression of the clipping mapping depends on the choice of the convex compact set $\ca{S}$. When $\ca{S}$ is chosen as the $n$-dimensional hypercube $[-1,1]^{n}$, the corresponding clipping mapping $\mathrm{Clip}_C$ can be expressed as $\mathrm{Clip}_C(g) = \min\{\max\{g, -C\}, C\}$, which is a coordinate-wise mapping and can be computed easily. Furthermore, when $\ca{S}$ is chosen as the unit ball in $\Rn$, the corresponding clipping mapping becomes $\mathrm{Clip}_C(g) =  g \cdot \min\left\{ 1, \frac{C}{\norm{g}} \right\}$, which coincides with the clipping mapping employed in \citep{zhang2019gradient}. Note that it is not a coordinate-wise mapping. It is worth mentioning that the convergence properties of our analyzed stochastic subgradient methods are independent of the specific choice of $\ca{S}$. Hence we do not specify the choice of $\ca{S}$ in the clipping mapping $\mathrm{Clip}_C(\cdot)$ in the rest of this section.

\subsection{SGD with Gradient Clipping}
\label{Subsection_51}
In this subsection,   we consider a general framework of SGD that incorporates the gradient clipping technique to deal with heavy-tailed evaluating noises in its stochastic subgradients: 
\begin{equation}
	\label{Eq_SGDM}
	\tag{SGD-C}
        \boxed{
	\begin{aligned}
            &\text{Sample $s_k \sim P$ and choose $g_k = \chi(\xk, s_k)$},\\
		&\hat{g}_k= \mathrm{Clip}_{C_k}(g_k), \\
		&m_{k+1} = (1-\tau_1 \eta_k) m_k + \tau_1  \eta_k \hat{g}_k,\\
		&\xkp = \xk -\eta_k (   m_{k+1} +  \alpha \hat{g}_k ).\\
	\end{aligned}
        }
\end{equation}
Here $\chi$ is a selection of the set-valued mapping $\D$, as defined in Assumption \ref{Assumption_Implementation_f}(2). Moreover,  $\tau_1$ and $\alpha$ refer to the parameters for heavy-ball momentum and Nesterov momentum, respectively. Therefore, compared to existing works that concentrate on the convergence of standard SGD without any momentum term, the updating scheme in \ref{Eq_SGDM} encompasses various popular variants of SGD, including SGD \citep{bolte2021conservative, bianchi2022convergence}, and its momentum accelerated variants \citep{nesterov2003introductory, loizou2017linearly, castera2021inertial}.

To establish the convergence properties for \ref{Eq_SGDM} method, we make the following assumptions.
\begin{assumpt}
	\label{Assumption_Clipping}
	\begin{enumerate}
		\item The parameters satisfy $\alpha \geq 0$, $\tau_1 > 0$. 
		\item The sequence $\{\xk\}$ is almost surely bounded. That is, 
		\begin{equation*}
			\sup_{k\geq 0} \norm{\xk} <  +\infty
		\end{equation*}
		holds almost surely. 
		\item The stepsizes $\{\eta_k\}$ and clipping parameters $\{C_k\}$ are positive and satisfy
		\begin{equation*}
			\sum_{k = 0}^{+\infty} \eta_k = +\infty, \quad  \lim_{k\to +\infty} \eta_k \log(k) = 0, \quad\lim_{k\to +\infty} C_k = +\infty, \quad  \text{and} \quad \lim_{k\to +\infty} C_k^2 \eta_k\log(k) = 0. 
		\end{equation*}
	\end{enumerate}
\end{assumpt}

Different from the existing works, in this section, we only assume the evaluation noises to be integrable in Assumption \ref{Assumption_Implementation_f}(3), without any further assumptions of the uniform boundedness such as Assumption \ref{Assumption_Implementation}(4). As far as we know, such an assumption is among the weakest ones in the relevant literature  \citep{zhang2019gradient,gorbunov2020stochastic,zhang2020adaptive,mai2021stability,qian2021understanding,elesedy2023u,reisizadeh2023variance}. Moreover,  Assumption \ref{Assumption_Clipping}(3) is mild, as we can always choose $C_k = C_0\left( \eta_k \log(k) \right)^{-\frac{1}{3}}$ in \ref{Eq_SGDM}.

 Let the $\sigma$-algebras $\{\ca{F}_k\}$ be chosen as $\ca{F}_k = \sigma(x_j, m_j, g_j, s_j: j< k)$, $d_k = \bb{E}_{s_k \sim P}[\hat{g}_k]$, and $\xi_k = \frac{\hat{g}_k - d_k}{C_k}$. Then $\hat{g}_k = d_k + C_k \xi_k$. Hence the update scheme in \ref{Eq_SGDM} can be rewritten as 
 \begin{equation*}
    \begin{aligned}
        m_{k+1} ={}& (1-\tau_1 \eta_k) m_k + \tau_1  \eta_k (d_k + C_k \xi_k),\\
		\xkp ={}& \xk -\eta_k (   m_{k+1} +  \alpha (d_k + C_k \xi_k) ).\\
    \end{aligned}
 \end{equation*}
 Here $(-(\mkp + \alpha d_k), \tau_1(-\mk + d_k))$ is regarded as the noiseless update direction for $(\xk, \mk)$, while $(\alpha C_k \xi_k, \tau_1 C_k \xi_k)$ can be interpreted as the corresponding evaluation noises.  We first present Lemma \ref{Le_clipping_convergence_dk} to exhibit some basic properties of the sequence $\{d_k\}$ and $\{\xi_k\}$. 
\begin{lem}
    \label{Le_clipping_convergence_dk}
    Suppose Assumption \ref{Assumption_Implementation_f} and Assumption \ref{Assumption_Clipping}(1) hold, then there exists a sequence of nonnegative constants $\{\delta_k\}$ such that $\lim_{k \to +\infty} \delta_k = 0$,  
    \begin{equation*}
         d_k \in \D_f^{\delta_k}(\xk), \quad \forall k\geq 0, 
    \end{equation*}
    and $\{\xi_k\}$ is a uniformly bounded martingale difference sequence.
\end{lem}
\begin{proof}
    Let $\varepsilon_{\ca{S}} = \mathop{\min}_{x \notin \ca{S}} \norm{x}$, and $M_{\ca{S}} = \mathop{\max}_{x \in \ca{S}} \norm{x}$. Then from the definition of $\hat{g}_k$, it holds for any $C>0$ and $k\geq 0$ that 
    \begin{equation*}
        \begin{aligned}
            &\bb{E}_{s_k \sim P}[\norm{g_k - \mathrm{Clip}_C(g_k)}] = \bb{E}_{s_k \sim P}\left[\norm{g_k - \mathrm{Clip}_C(g_k)} \cdot \mathbbm{1}_{\{ g_k \notin C\ca{S} \}}\right] \\
            \leq{}& \bb{E}_{s_k \sim P}\left[\norm{g_k } \cdot \mathbbm{1}_{\{\norm{g_k} \geq C\varepsilon_{\ca{S}} \}}\right] \leq \bb{E}_{s_k \sim P}\left[\norm{g_k } \cdot \mathbbm{1}_{\{ p_{\Theta}(s_k) \geq C \varepsilon_{\ca{S}}\}}\right] \\
            \leq{}& \bb{E}_{s_k \sim P}\left[ p_{\Theta}(s_k) \cdot \mathbbm{1}_{\{ p_{\Theta}(s_k) \geq {C \varepsilon_{\ca{S}}}\}}\right]. 
        \end{aligned}
    \end{equation*}
    As a result, from the fact that $p_{\Theta}$ is integrable over $\Theta$, it holds that 
    \begin{equation*}
        \lim_{C\to +\infty}\sup_{k \geq 0} \bb{E}_{s_k \sim P}\left[\norm{g_k - \mathrm{Clip}_{C}(g_k)}\right] = 0.
    \end{equation*}
    Therefore, let $\delta_k =  \bb{E}_{s_k \sim P}\left[\norm{g_k - \mathrm{Clip}_{C_k}(g_k)}\right]$,  then it is easy to verify that $\lim_{k\to +\infty}\delta_k = 0$. 
    Moreover, from the definition of $\delta_k$, we have
    \begin{equation*}
        \begin{aligned}
            & \mathrm{dist}\left( d_k, \D_f(\xk) \right) \leq  \norm{\bb{E}_{s_k \sim P}[\hat{g}_k] - \bb{E}_{s_k \sim P}[g_k]} = \norm{\bb{E}_{s_k \sim P}\left[g_k - \mathrm{Clip}_{C_k}(g_k)\right]} \\
            \leq{}&  \bb{E}_{s_k \sim P}\left[\norm{g_k - \mathrm{Clip}_{C_k}(g_k)}\right] = \delta_k.
        \end{aligned}
    \end{equation*}

    Furthermore, from the definition of $\{\xi_k\}$, it holds for any $k\geq 0$ that $\norm{\hat{g}_k} \leq C_k M_{\ca{S}}$ almost surely. Then we can conclude that almost surely, $\sup_{k\geq 0} \norm{\xi_k} \leq M_{\ca{S}}$. Moreover, as $s_k$ is chosen independently from $\{s_0,..., s_{k-1}\}$, it holds that $\bb{E}[\xi_k|\ca{F}_{k-1}] =   \bb{E}_{s_k\sim P}[\xi_k] = 0$. Therefore, we can conclude that $\{\xi_k\}$ is a uniformly bounded martingale difference sequence. This completes the proof.  
\end{proof}

Next, we show that the sequence $\{\mk\}$ is almost surely uniformly bounded in the following proposition, even if the corresponding evaluation noises are not uniformly bounded. The proof for Proposition \ref{Prop_Clipping_mk_uniformly_bounded} is presented in Appendix A. 
\begin{prop}
    \label{Prop_Clipping_mk_uniformly_bounded}
    Suppose Assumption \ref{Assumption_Implementation_f} and Assumption \ref{Assumption_Clipping} hold, then we have 
    \begin{equation*}
        \sup_{k\geq 0} \norm{\mk} <+\infty.
    \end{equation*}
\end{prop}

In the following theorem, we demonstrate that the framework \ref{Eq_SGDM} conforms to the framework \eqref{Eq_framework}, and directly establish its convergence properties based on Theorem \ref{The_convergence} and Theorem \ref{The_convergence_AS}.

\begin{theo}
	\label{The_Gradient_Clipping_Convergence}
	Let $\{\xk\}$ be the sequence generated by \ref{Eq_SGDM}. Suppose Assumption \ref{Assumption_Implementation_f} and Assumption \ref{Assumption_Clipping} hold. Then, almost surely, any cluster point of $\{\xk\}$ is a $\D_f$-stationary point of $f$, and the sequence $\{f(\xk)\}$ converges.
\end{theo}
\begin{proof}
	Let $d_k = \bb{E}_{s_k \sim P}[\hat{g}_k]$, then from Lemma \ref{Le_clipping_convergence_dk}, it holds that $\lim_{k\to +\infty}\mathrm{dist}\left( d_k, \D_f(\xk) \right) = 0$, and $\{\xi_k\}$ is a uniformly bounded martingale difference sequence. 
 
    Then we set $\theta_k = C_k \eta_k$ in the framework \eqref{Eq_framework}.  With $\gamma =0$, $\varepsilon = 1$, and $\ca{U}(x,m,v) = \{0\}$ in \eqref{Eq_mapping_G}, \ref{Eq_SGDM} can be reformulated as the following scheme, 
	\begin{equation*}
		(\xkp, \mkp, \vkp) = (\xk, \mk, \vk) - \eta_k (d_{x, k}, d_{m,k}, d_{v,k}) - \theta_k( \alpha\xi_k, -\tau_1 \xi_k, 0),
	\end{equation*}
	where $(d_{x, k}, d_{m,k}, d_{v,k}) = (m_k + \alpha d_k, \tau_1 \mk - \tau_1 d_k, \tau_2 v)$.

    Next, we check the validity of Assumption \ref{Assumption_alg}. Assumption \ref{Assumption_Clipping}(1)-(2) and Proposition \ref{Prop_Clipping_mk_uniformly_bounded} imply Assumption \ref{Assumption_alg}(1)-(2), while Assumption \ref{Assumption_alg}(3) holds as we set $\ca{U}(x,m,v) = \{0\}$ in \eqref{Eq_framework}.  
    Moreover, from the uniform boundedness of $\{(\xk, \mk)\}$ and Lemma \ref{Le_clipping_convergence_dk},  it is easy to verify that there exists a diminishing sequence $\{\delta_k\}$ such that  $ (d_{x, k}, d_{m,k}, d_{v,k})  \in  \ca{G}^{\delta}(\xk, \mk, \vk)$ with $\gamma =0$, $\varepsilon = 1$, and $\ca{U}(x,m,v) = \{0\}$ in \eqref{Eq_mapping_G}. Then the validity of 
    Assumption \ref{Assumption_alg}(4) is guaranteed. In addition, Assumption \ref{Assumption_alg}(5) follows from the fact that $\{\xi_k\}$ is a martingale difference sequence, as illustrated in Lemma \ref{Le_clipping_convergence_dk}.  Furthermore, Assumption \ref{Assumption_alg}(6) directly follows 
    Assumption \ref{Assumption_Clipping}(4) with $\theta = C_k \eta_k$.   Therefore, from Theorem  \ref{The_convergence}, we can conclude that any cluster point of $\{\xk\}$ is a $\D_f$-stationary point of $f$, and the sequence $\{f(\xk)\}$ converges. This completes the proof. 
\end{proof}

The following theorem illustrates that under mild assumptions with almost every initial points and stepsizes, any sequence generated by \ref{Eq_SGDM} converges to $\partial f$-stationary points of $f$, hence avoids the spurious stationary points introduced by conservative field $\D_f$. 
\begin{theo}
	\label{The_Gradient_Clipping_Convergence_AS}
	Suppose Assumption \ref{Assumption_Implementation_f} holds. Moreover, for the sequence  $\{\xk\}$  generated by \ref{Eq_SGDM}, we assume that  
        \begin{enumerate}
            \item There exists a prefixed positive sequence $\{\nu_k\}$ and the parameters $0<c_{\min}<c_{\max}$ such that the stepsizes $\{\eta_k\}$ are chosen as $\eta_k = c\nu_k$ for any $k\geq 0$ with some $c \in (c_{\min}, c_{\max})$. 
            \item There exists a non-empty open subset $K$ of $\Rn \times \Rn$ such that Assumption \ref{Assumption_Clipping}  holds with any $(x_0, m_0, c) \in K \times (c_{\min}, c_{\max})$. 
        \end{enumerate}
        Then  for almost every $(x_0, m_0, c)  \in K \times (c_{\min}, c_{\max})$,  it holds almost surely that every cluster point of $\{\xk\}$ is a $\partial f$-stationary point of $f$  and the sequence $\{f(\xk)\}$ converges. 
\end{theo}

\begin{proof}
    For almost every $s \in \Theta$, since $f_s$ is definable, we can conclude that  the set-valued mapping $x\mapsto \partial f_s(x)$ is almost everywhere $\ca{C}^1$. Then we consider the following set-valued mappings
    \begin{equation*}
        \ca{Q}^{(1)}_k(x, m, s_k) = \left[\begin{matrix}
			0\\[3pt]
			-\tau_1 m + \tau_1 \mathrm{Clip}_{C_k}\left( \D(x, s_k)  \right)
		\end{matrix}\right]
    \end{equation*}
    and 
    \begin{equation*}
        \ca{Q}^{(2)}_k(x, m, s_k) = \left[\begin{matrix}
			-m - \alpha \mathrm{Clip}_{C_k}\left( \D(x, s_k)\right) \\[3pt]
			0
		\end{matrix}\right].
    \end{equation*}
    From the definability of $f_s$ and $\D(\cdot, s)$, $\D(\cdot, s)$ is almost everywhere $\ca{C}^1$ over $\Rn$ \citep[Theorem 4]{bolte2021conservative}. As a result, from the continuity of clipping mapping $\mathrm{Clip}_{C}$, we can conclude that for any $k\geq 0$ and almost every $s \in \Theta$, both $\ca{Q}^{(1)}(\cdot, \cdot,  s)$ and $\ca{Q}^{(2)}(\cdot, \cdot,  s)$ are almost everywhere $\ca{C}^1$ over $\Rn \times \Rn$.

    From the expression of $\ca{Q}^{(1)}$ and $\ca{Q}^{(2)}$, the update scheme in \ref{Eq_SGDM} can be reshaped as 
    \begin{equation*}
        \begin{aligned}
            &(x_{k+\frac{1}{2}},m_{k+\frac{1}{2}}) \in (\xk, \mk) - \eta_k \ca{Q}_k^{(1)}(\xk, \mk, s_k),\\
            &(\xkp, \mkp) \in (x_{k+\frac{1}{2}},m_{k+\frac{1}{2}}) - \eta_k \ca{Q}_k^{(2)}(x_{k+\frac{1}{2}},m_{k+\frac{1}{2}}, s_k).
        \end{aligned}
    \end{equation*}
        
    Notice that the set $A := \{ (x, m) \in \Rn \times \Rn: \D_f(x) \neq  \partial f(x) \}$ is zero-measure in $\Rn \times\Rn$ as illustrated in \citet{bolte2021conservative}. 
	Therefore, Theorem \ref{The_convergence_AS} illustrates that for almost every $(x_0, m_0, c) \in K\times (c_{\min}, c_{\max})$, it holds almost surely that the sequence $\{(\xk, \mk)\}$ generated by \ref{Eq_SGDM} satisfies $\{(\xk, \mk)\} \subset A^c$.

    From the definition of $A$, we can conclude that $\D_f(x) = \partial f(x)$ holds for any $x \in A^c$. Therefore, for any sequence $\{(\xk, \mk)\}$ with those initial points $(x_0, m_0) \in K$ and scaling parameter $c \in (c_{\min}, c_{\max})$, 
    the corresponding conservative field $\D_f$ can be directly chosen as $\partial f$ since $\D_f(\xk) = \partial f(\xk)$ holds for any $k\geq 0$.  Therefore, Theorem \ref{The_convergence} illustrates that with those initial points $(x_0, m_0) \in K$ and scaling parameter $c \in (c_{\min}, c_{\max})$, any cluster point of $\{\xk\}$ is a $\partial f$-stationary point and the sequence $\{f(\xk)\}$ converges. This completes the proof.  
\end{proof}

\subsection{Adam-family Method with Gradient Clipping}
\label{Subsection_52}
In this subsection, we consider developing an Adam-family method (\ref{Eq_ADAMC}) that employs the gradient clipping technique for solving \ref{Prob_Ori} under heavy-tailed evaluation noises. Then we establish its convergence properties based on our proposed framework \eqref{Eq_framework}. The detailed method is summarized by the following update scheme.
\begin{equation}
	\label{Eq_ADAMC}
	\tag{ADAM-C}
        \boxed{
	\begin{aligned}
		&\text{Sample $s_k \sim P$ and choose $g_k = \chi(\xk, s_k)$},\\ 
		&\hat{g}_k= \mathrm{Clip}_{C_k}(g_k), \\
		&m_{k+1} = (1-\tau_1 \eta_k) m_k + \tau_1  \eta_k \hat{g}_k,\\
		&\vkp =  (1-\tau_2 \eta_k) \vk + \tau_2  \eta_k |\hat{g}_k|,\\
		&\xkp = \xk - \eta_k (\rho_{v,k+1}|v_{k+1}| + \varepsilon )^{-1} \odot ( \rho_{m,k+1} m_{k+1} +  \alpha g_k ).\\
	\end{aligned}
        }
\end{equation}
Here the $\chi$ is a selection of the set-valued mapping $\D$, as defined in Assumption \ref{Assumption_Implementation_f}(2). Moreover, the estimator ${\vk}$ is updated for tracking the first-order moment of $|g_k|$.

It is worth mentioning that the estimators $\{\vk\}$ in \ref{Eq_ADAMC} adopt a different update scheme as those in the original Adam \citep{kingma2014adam}, since the evaluation noises are assumed to be heavy-tailed.   In the original Adam, the estimators ${\vk}$ estimate the noise level of each coordinate by tracking the second-order moment of the stochastic subgradients ${g_k}$. 
However, when the evaluation noise of $g_k$ is assumed to be heavy-tailed, $\bb{E}[\mathrm{Clip}_C(g_k)^2]$ may diverge to infinity  as $C\to +\infty$. As a result, the sequence $\{\vk\}$ may not be uniformly bounded in Adam under heavy-tailed noises, leading to the absence of convergence guarantees.

To estimate the noise level of each coordinate under heavy-tailed evaluation noises, we consider tracking the first-order moment of $\{|g_k|\}$ by $\{\vk\}$, and employ the  $(|v_{k+1}| + \varepsilon )^{-1}$ as the coordinate-wise adaptive stepsizes. Under Assumption \ref{Assumption_Implementation_f}, $\bb{E}_{s_k \sim P}[|g_k|]$ exists and takes finite values almost surely. As a result, the estimators $\{\vk\}$ in \ref{Eq_ADAMC} can be proved to be uniformly bounded, which is crucial in establishing the convergence properties for \ref{Eq_ADAMC} based on the framework \eqref{Eq_framework}.


To establish the convergence properties for \ref{Eq_ADAMC}, we make the following assumptions.
\begin{assumpt}
	\label{Assumption_Clipping_ADAM}
	\begin{enumerate}
		\item The parameters satisfy $\alpha \geq 0$, $\tau_1, \tau_2,  \varepsilon > 0$ and $\tau_2 \leq 2\tau_1$. 
		\item The sequence $\{\xk\}$ is  almost surely bounded, i.e., 
		\begin{equation*}
			\sup_{k\geq 0} \norm{\xk} <+\infty
		\end{equation*}
		holds almost surely. 
		\item The stepsizes $\{\eta_k\}$ and clipping parameters $\{C_k\}$ are positive and satisfy
		\begin{equation*}
			\sum_{k = 0}^{+\infty} \eta_k = +\infty, \quad  \lim_{k\to +\infty} \eta_k \log(k) = 0, \quad\lim_{k\to +\infty} C_k = +\infty, \quad  \text{and} \quad \lim_{k\to +\infty} C_k^2 \eta_k\log(k) = 0. 
		\end{equation*}
		\item The scaling parameters $\{\rho_{m,k}\}$ and $\{\rho_{v,k}\}$ satisfy 
		\begin{equation*}
			\lim_{k \to +\infty} \rho_{m,k} = 1, \quad \lim_{k \to +\infty} \rho_{v,k} = 1.
		\end{equation*}
	\end{enumerate}
\end{assumpt}

We first present Proposition \ref{Prop_ClippingAdam_mk_uniformly_bounded_} to illustrate that the uniform boundedness of $\{\xk\}$ implies the uniformly boundedness of $\{\mk\}$ and $\{\vk\}$. The proof of Proposition \ref{Prop_ClippingAdam_mk_uniformly_bounded_} follows the same techniques as in Proposition \ref{Prop_Clipping_mk_uniformly_bounded}, hence we omit its proof for simplicity.  
\begin{prop}
    \label{Prop_ClippingAdam_mk_uniformly_bounded_}
    Suppose Assumption \ref{Assumption_Implementation_f} and Assumption \ref{Assumption_Clipping_ADAM} hold. Then we have $\sup_{k\geq 0} \norm{\mk} + \norm{\vk} <+\infty$.
\end{prop}

Next, we present the following theorem that illustrates the convergence properties of \ref{Eq_ADAMC}.
\begin{theo}
	\label{The_Gradient_Clipping_Convergence_Adam}
	Let $\{\xk\}$ be the sequence generated by \ref{Eq_ADAMC}. Suppose Assumption \ref{Assumption_Implementation_f} and Assumption \ref{Assumption_Clipping_ADAM} hold. Then, almost surely, any cluster point of $\{\xk\}$ is a $\D_f$-stationary point of $f$, and the sequence $\{f(\xk)\}$ converges. 
\end{theo}
\begin{proof}
	Let $\ca{W}(x) = \conv \left(\bb{E}_{s}[|\D(x, s) |]  \right)$. Then following Assumption \ref{Assumption_Implementation_f}(3), we can conclude that for any $x \in \Rn$ and almost every $s \in \Theta$, it holds that $\sup_{d \in |\D(x, s) |} \norm{d} \leq p_{\Theta}(s) $. Then together with \citep[Theorem 2]{shapiro2007uniform} and the fact that the set-valued mapping $x\mapsto |\D(x, s) |$ is locally bounded and graph closed for almost every $s \in \Theta$, we can conclude that $\ca{W}(x)$ is convex compact valued, graph closed and locally bounded. 
 
    Then with $\ca{U}(x,m,v) = \ca{W}(x)$ in \eqref{Eq_mapping_G}, we can show that \ref{Eq_ADAMC} fits in the framework \eqref{Eq_framework}. Moreover, similar to the proof in Theorem \ref{The_Gradient_Clipping_Convergence}, we can verify the validity of Assumption \ref{Assumption_alg}. Then from Theorem \ref{The_convergence} we can conclude that any cluster point of $\{\xk\}$ is a $\D_f$-stationary point of $f$ and the sequence $\{f(\xk)\}$ converges. This completes the proof. 
\end{proof}

Similar to Theorem \ref{The_Gradient_Clipping_Convergence_AS}, we can show that under mild assumptions with almost every initial point and stepsize, any sequence generated by \ref{Eq_ADAMC} is capable of finding $\partial f$-stationary points of $f$, regardless of the chosen conservative field in \ref{Eq_ADAMC}. 
\begin{theo}
	\label{The_Gradient_Clipping_Convergence_Adam_AS}
        Suppose Assumption \ref{Assumption_Implementation_f} holds. Moreover, for the sequence $\{(\xk, \mk, \vk)\}$ generated by \ref{Eq_ADAMC}, we assume that 
        \begin{enumerate}
            \item There exists a prefixed sequence $\{\nu_k\}$ and the  parameters $0< c_{\min}<c_{\max}$ such that the stepsizes $\{\eta_k\}$ in \ref{Eq_ADAMC} are chosen as $\eta_k = c\nu_k$ for any $k\geq 0$ with some $c \in (c_{\min}, c_{\max})$. 
            \item There exists a non-empty open subset $K$ of $\Rn \times \Rn\times \Rn_+$ such that Assumption \ref{Assumption_Clipping_ADAM} holds with any $(x_0, m_0, v_0, c) \in K\times (c_{\min}, c_{\max})$.
        \end{enumerate}
	Then  for almost every $(x_0, m_0, v_0, c) \in K\times (c_{\min}, c_{\max})$, it holds almost surely that  every cluster point of $\{\xk\}$ is a $\partial f$-stationary point of $f$  and the sequence $\{f(\xk)\}$ converges . 
\end{theo}
\begin{proof}
	Let $A := \{ (x, m, v) \in \Rn \times \Rn \times \Rn: \D_f(x) \neq  \partial f(x) \}$, and for any $k\geq 0$, we define  the set-valued mappings
        \begin{equation*}
            \small
            \ca{Q}_k^{(1)}(x, m, v, s) := 
            \left[\begin{matrix}
			0\\[3pt]
			\tau_1 m - \tau_1 \mathrm{Clip}_{C_k}(\D(x, s))\\[3pt]
			0
		\end{matrix}\right], \quad 
            \ca{Q}_k^{(2)}(x, m, v, s) := 
            \left[\begin{matrix}
			0\\[3pt]
			0\\[3pt]
			\tau_2 v - \tau_2 |\mathrm{Clip}_{C_k}(\D(x, s))|
		\end{matrix}\right],
        \end{equation*}
        and 
        \begin{equation*}
            \ca{Q}_k^{(3)}(x, m, v, s) := 
            \left[\begin{matrix}
			(\rho_{v,k+1}|v| + \varepsilon )^{-1} \odot \left( \rho_{m,k+1} m +  \alpha \mathrm{Clip}( \D(x, s) ) \right)\\[3pt]
			0\\[3pt]
			0
		\end{matrix}\right].
        \end{equation*}
        Then for any $k\geq 0$ and almost every $s \in \Theta$, $\ca{Q}_k^{(1)}(\cdot, \cdot, \cdot, s)$, $\ca{Q}_k^{(2)}(\cdot, \cdot, \cdot, s)$, and $\ca{Q}_k^{(3)}(\cdot, \cdot, \cdot, s)$  are almost everywhere $\ca{C}^1$ in $\Rn \times \Rn\times \Rn$. More importantly, the update scheme in \eqref{Eq_ADAMC} can be reshaped as 
        \begin{equation*}
         \begin{aligned}
             &\left(x_{k+\frac{1}{3}}, m_{k+\frac{1}{3}}, v_{k+\frac{1}{3}} \right) \in (\xk, \mk, \vk) - \eta_k Q_k^{(1)}(\xk, \mk, \vk, s_k),\\
             &\left(x_{k+\frac{2}{3}}, m_{k+\frac{2}{3}}, v_{k+\frac{2}{3}} \right) 
             \in (x_{k+\frac{1}{3}}, m_{k+\frac{1}{3}}, v_{k+\frac{1}{3}}) - \eta_k Q_k^{(2)}\left(x_{k+\frac{1}{3}}, m_{k+\frac{1}{3}}, v_{k+\frac{1}{3}}, s_k\right),\\
             &(\xkp, \mkp, \vkp) \in \left(x_{k+\frac{2}{3}}, m_{k+\frac{2}{3}}, v_{k+\frac{2}{3}}\right) - \eta_k  Q_k^{(3)}\left(x_{k+\frac{2}{3}}, m_{k+\frac{2}{3}}, v_{k+\frac{2}{3}}, s_k\right). 
         \end{aligned} 
        \end{equation*}

        Notice that the set $A$ is zero-measure in $\Rn \times\Rn \times \Rn$, following the same techniques as Theorem \ref{The_Gradient_Clipping_Convergence_AS}, we can prove that for almost every $(x_0, m_0, v_0, c) \in K \times (c_{\min}, c_{\max})$, it holds that $\{(\xk, \mk, \vk): k = 0,\frac{1}{3}, \frac{2}{3}, 1,  ...\} \subset A^c$. As a result, for almost every $(x_0, m_0, v_0, c) \in K \times (c_{\min}, c_{\max})$, 
        we can choose the conservative field $\D_f$ in \ref{Eq_ADAMC} as $\partial f$ in Theorem \ref{The_Gradient_Clipping_Convergence_Adam}. Therefore, Theorem \ref{The_convergence} illustrates that every cluster point of $\{\xk\}$ is a $\partial f$-stationary point of $f$  and the sequence $\{f(\xk)\}$ converges almost surely. This completes the proof. 
\end{proof}

\begin{rmk}
	Following the updating schemes in Table \ref{Table_Adaptive_algs}, we can also choose the updating scheme for the estimators $\{\vk\}$ in \ref{Eq_ADAMC} as one of the followings.
	\begin{itemize}
		\item {\bf AdaBelief-C:} $\vkp =  (1-\tau_2 \eta_k) \vk + \tau_2  \eta_k |\hat{g}_k - \mkp|$;
		\item {\bf AMSGrad-C:} $\vkp = \vk + \tau_2 \eta_k \max\{ 0, |\hat{g}_k|-\vk \}$;
		\item {\bf Yogi-C:} $\vkp =  \vk - \tau_2  \eta_k \mathrm{sign}(\vk - |\hat{g}_k|) \odot |\hat{g}_k|$.
	\end{itemize}
	Then for these stochastic adaptive subgradient methods with gradient clipping,  we can establish the same convergence properties by following the same proof routines as those in Theorem \ref{The_Gradient_Clipping_Convergence_Adam} and Corollary \ref{The_Gradient_Clipping_Convergence_Adam_AS}, hence we omit these proofs for simplicity. 
\end{rmk}

\section{Numerical Experiments}

In this section, we evaluate the numerical performance of our analyzed Adam-family methods for training nonsmooth neural networks. All the numerical experiments in this section are conducted on a server equipped with an Intel Xeon 6342 CPU and a NVIDIA GeForce RTX 3090 GPU, running Python 3.8 and PyTorch 1.9.0. 

\subsection{Comparison with Implementations in PyTorch}
\label{Section_numerical_adaptive}
In this subsection, we evaluate the numerical performance of Algorithm \ref{Alg:ADAM} by comparing it with the Adam-family methods available in PyTorch and torch-optimizer packages. In view of the great popularity of Adam-family methods in training nonsmooth neural networks, we aim to investigate whether we can preserve their high efficiency while providing convergence guarantees with minimal modifications to their implementations.

It is important to note that the Adam-family methods in PyTorch can be viewed as Algorithm \ref{Alg:ADAM} with a fixed $\eta_k = \eta_0$ in updating the momentum terms $\{\mk\}$ and estimators $\{\mk\}$ (i.e., Steps 5-6 in Algorithm \ref{Alg:ADAM}). Moreover, $\beta_1 := 1-\tau_1\eta_0$ and $\beta_2:= 1-\tau_2 \eta_0$ are commonly referred to as the momentum parameters for these Adam-family methods in PyTorch. To our best knowledge, these Adam-family methods with constant stepsizes in updating the momentum terms $\{\mk\}$ and estimators $\{\mk\}$ do not have any convergence guarantees in training nonsmooth neural networks. More importantly, some existing works \citep{reddi2019convergence,zhang2022adam} illustrate that Adam may diverge when $\beta_1 < \sqrt{\beta_2}$ and $f$ is assumed to be differentiable. 

In our numerical experiments, we investigate the performance of these compared Adam-family methods on training ResNet-50 \citep{he2016deep} for image classification tasks on the CIFAR-10 and CIFAR-100 data sets \citep{krizhevsky2009learning}. We set the batch size to 128 for all test instances and select the regularization parameter $\varepsilon$ as $\varepsilon = 10^{-15}$. Furthermore, at the $k$-th epoch, we choose the stepsize as 
$\eta_k = \frac{\eta_0}{\sqrt{k+1}}$ for all the tested algorithms. Following the settings in \citet{castera2021inertial}, we use a grid search to find a suitable initial stepsize $\eta_0$ and parameters $\tau_1, \tau_2$ for the Adam-family methods provided in PyTorch. We select the initial stepsize $\eta_0$ from $\{k_1 \times 10^{-k_2}: k_1 = 1,3,5,7,9, ~k_2 = 3,4,5\}$, and choose the parameters $\tau_1$, $\tau_2$ from $\{0.1/\eta_0, 0.05/\eta_0, 0.01/\eta_0, 0.005/\eta_0, 0.001/\eta_0\}$, to find a combination of $(\eta_0, \tau_1, \tau_2)$ that yields the most significant increase in accuracy after $20$ epochs.  All other parameters for these Adam-family methods in PyTorch remain fixed at their default values. 

For our proposed algorithms (i.e., Adam-family methods with diminishing stepsizes for $\{\mk\}$ and $\{\vk\}$ as in Algorithm \ref{Alg:ADAM}), we keep all other parameters the same as those available in PyTorch, as we aim to perform minimal modifications to their released counterparts. Moreover, to investigate the performance of our proposed Adam-family methods with the Nesterov momentum term, in each test instance, we choose the Nesterov momentum parameter $\alpha$ as $0$ and $0.1$, respectively. We run each test instance five times with different random seeds. In each test instance, all compared methods are tested using the same random seed and initialized with the same random weights by the default initialization function in PyTorch.

The numerical results are presented in Figure \ref{Fig_Test2_Cifar10} and Figure \ref{Fig_Test2_Cifar100}. These figures demonstrate that our proposed Adam-family methods with diminishing stepsizes exhibit the same performance as the existing Adam-family methods available in PyTorch and torch-optimizer packages.   These empirical results highlight the effectiveness of our proposed Adam-family methods, as they achieve  comparable performance to their widely used counterparts in the community. Furthermore, we note that the integration of Nesterov momentum can potentially lead to increased accuracy and reduced test loss across all tested Adam-family methods in our numerical experiments, especially in the classification tasks on the CIFAR-10 data set. These empirical results, together with our presented theoretical analysis, demonstrate that by simply choosing diminishing stepsizes for the momentum terms and estimators in existing Adam-family methods, we can preserve their high performance in practice while benefiting from the convergence guarantees in training nonsmooth neural networks. 
\begin{figure}[tb]
	\centering
	\subfigure[\scriptsize Test accuracy]{
		\begin{minipage}[t]{0.24\linewidth}
			\centering
			\includegraphics[width=\linewidth, height=0.12\textheight]{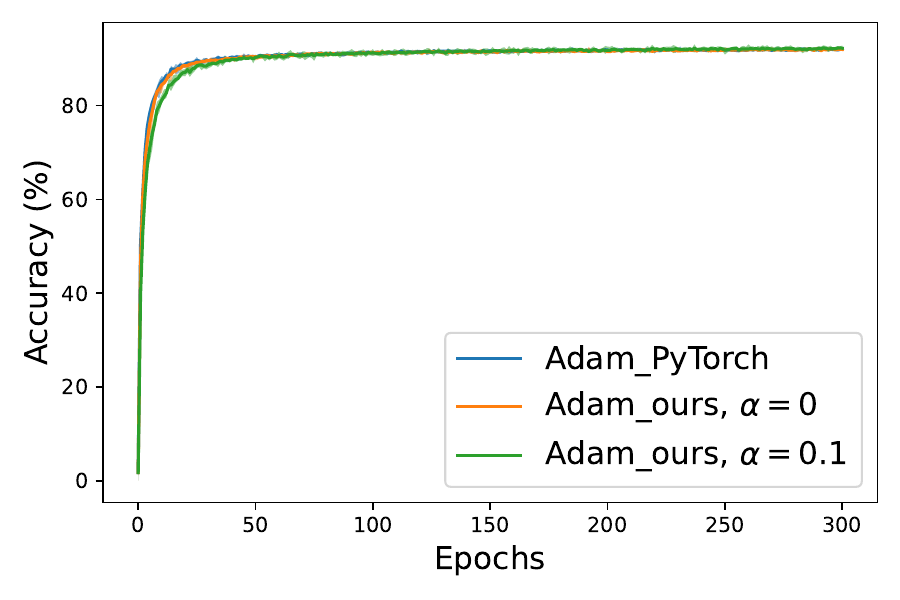}
			\label{Fig:Fig_Test2_Cifar10_Adam_acc}
		\end{minipage}%
	}%
	\subfigure[\scriptsize Test acc. after $200$ epochs]{
		\begin{minipage}[t]{0.24\linewidth}
			\centering
			\includegraphics[width=\linewidth, height=0.12\textheight]{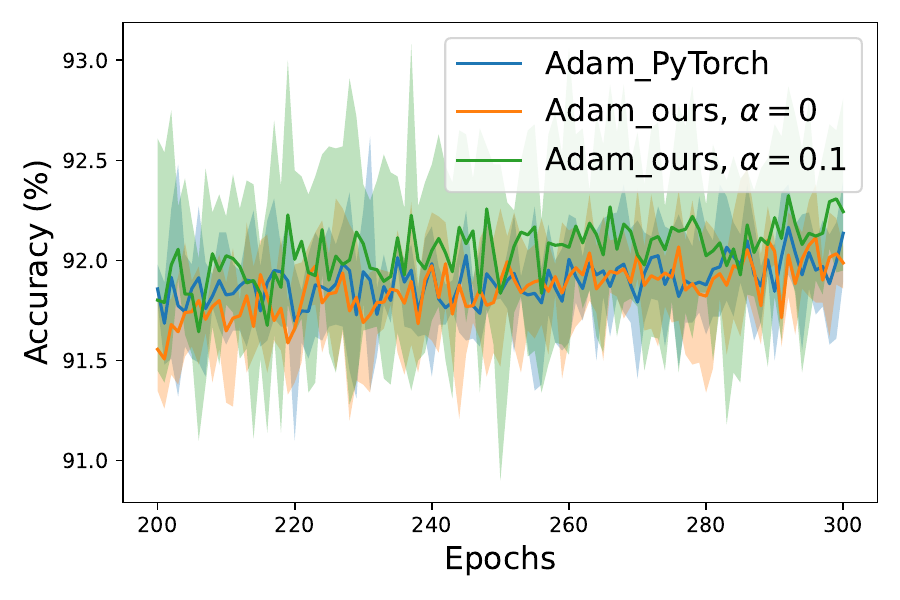}
			\label{Fig:Fig_Test2_Cifar10_Adam_acclater}
		\end{minipage}%
	}%
	\subfigure[\scriptsize Train loss]{
		\begin{minipage}[t]{0.24\linewidth}
			\centering
			\includegraphics[width=\linewidth, height=0.12\textheight]{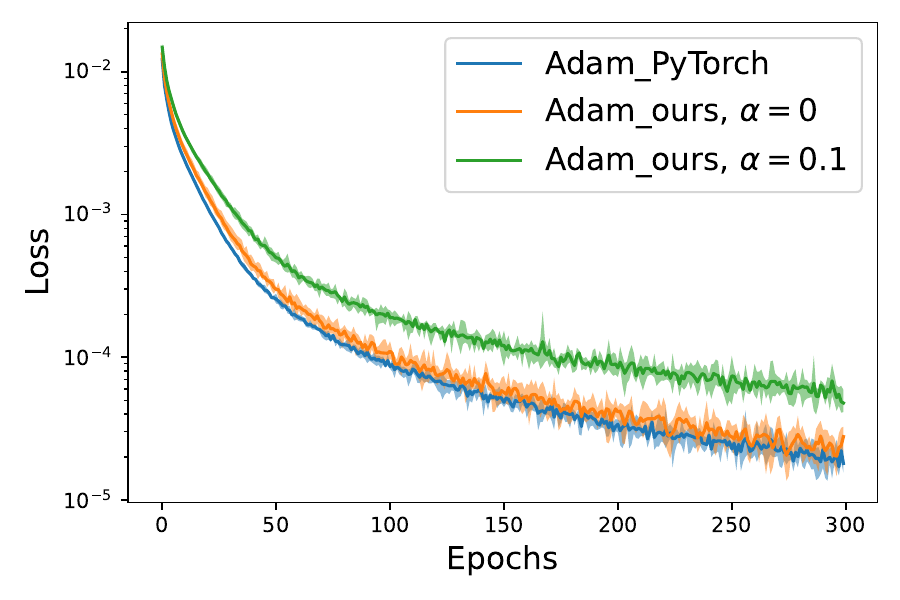}
			\label{Fig:Fig_Test2_Cifar10_Adam_trainloss}
		\end{minipage}%
	}%
	\subfigure[\scriptsize Test loss]{
		\begin{minipage}[t]{0.24\linewidth}
			\centering
			\includegraphics[width=\linewidth, height=0.12\textheight]{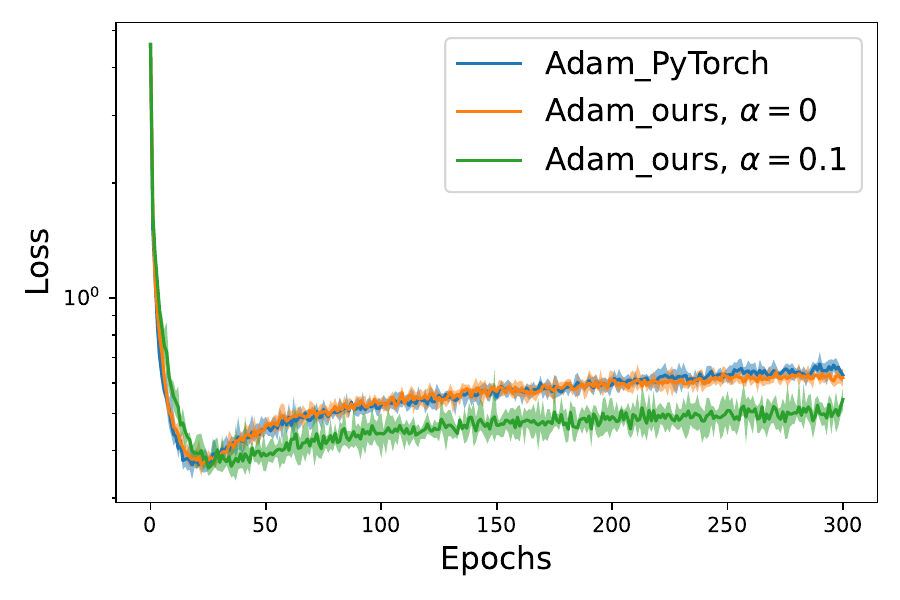}
			\label{Fig:Fig_Test2_Cifar10_Adam_testloss}
		\end{minipage}%
	}%

	\subfigure[\scriptsize Test accuracy]{
		\begin{minipage}[t]{0.24\linewidth}
			\centering
			\includegraphics[width=\linewidth, height=0.12\textheight]{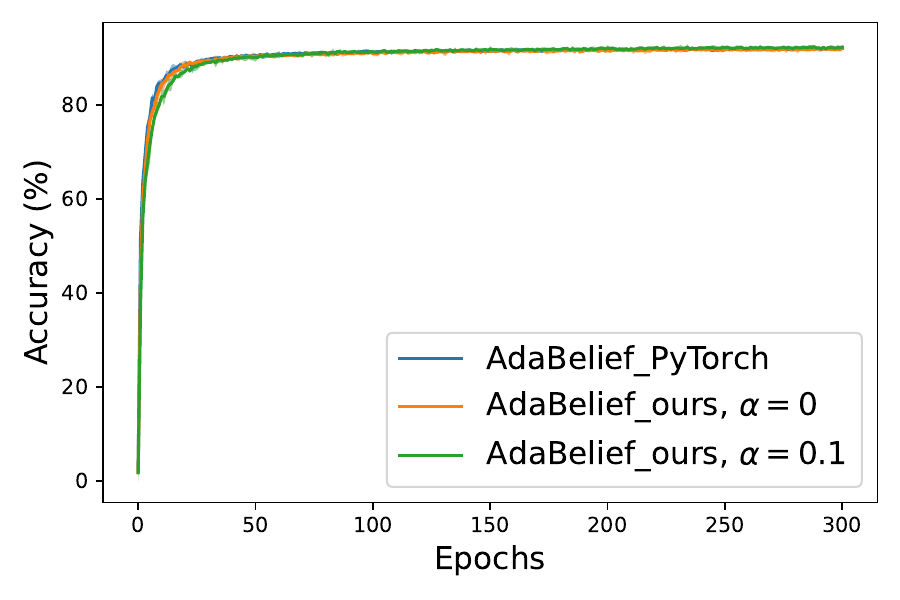}
			\label{Fig:Fig_Test2_Cifar10_AdaBelief_acc}
		\end{minipage}%
	}%
	\subfigure[\scriptsize Test acc. after $200$ epochs]{
		\begin{minipage}[t]{0.24\linewidth}
			\centering
			\includegraphics[width=\linewidth, height=0.12\textheight]{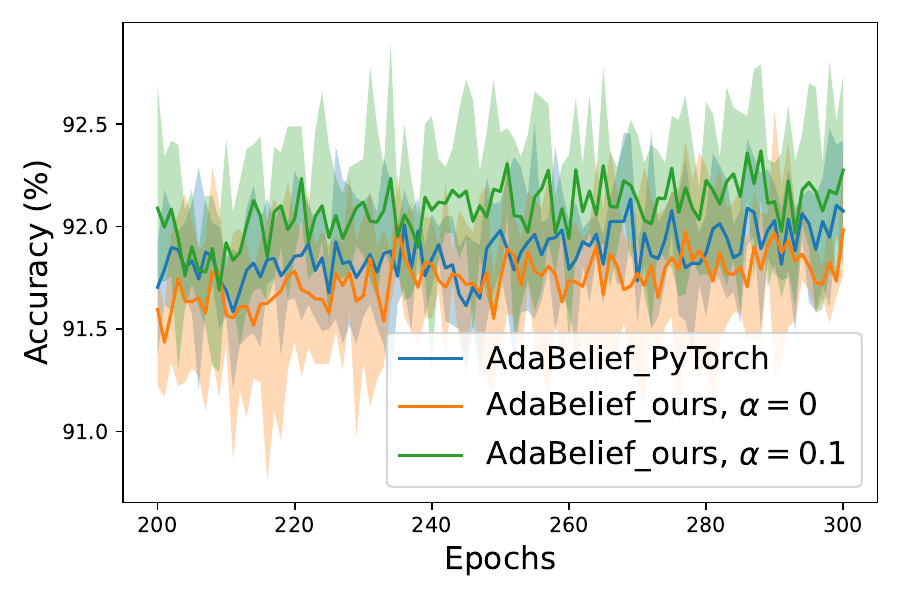}
			\label{Fig:Fig_Test2_Cifar10_AdaBelief_acclater}
		\end{minipage}%
	}%
	\subfigure[\scriptsize Train loss]{
		\begin{minipage}[t]{0.24\linewidth}
			\centering
			\includegraphics[width=\linewidth, height=0.12\textheight]{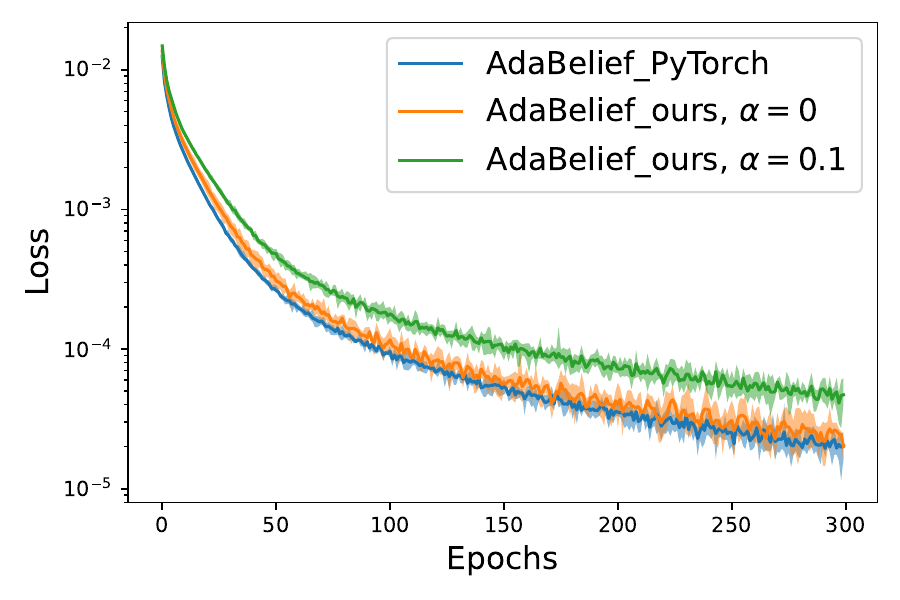}
			\label{Fig:Fig_Test2_Cifar10_AdaBelief_trainloss}
		\end{minipage}%
	}%
	\subfigure[\scriptsize Test loss]{
		\begin{minipage}[t]{0.24\linewidth}
			\centering
			\includegraphics[width=\linewidth, height=0.12\textheight]{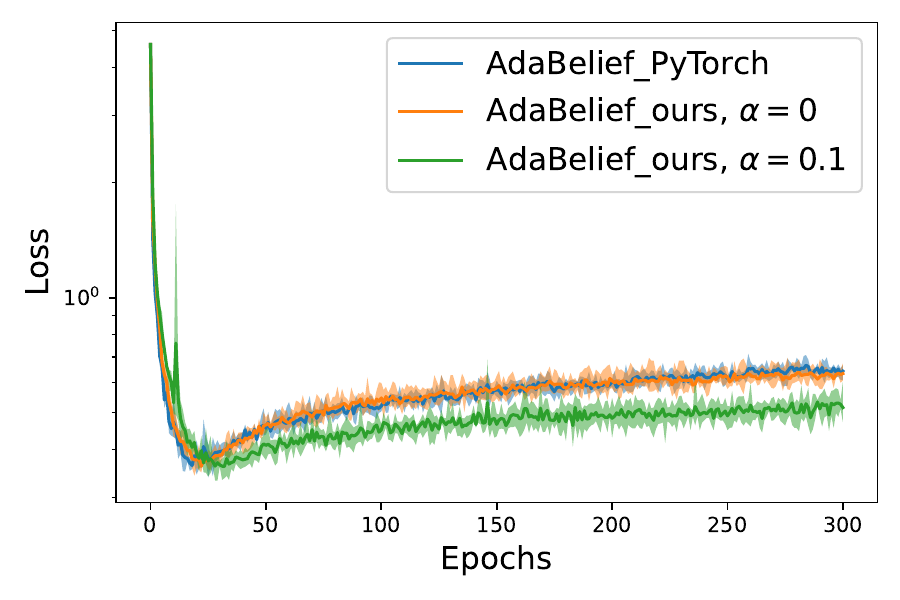}
			\label{Fig:Fig_Test2_Cifar10_AdaBelief_testloss}
		\end{minipage}%
	}%
	
	\subfigure[\scriptsize Test accuracy]{
		\begin{minipage}[t]{0.24\linewidth}
			\centering
			\includegraphics[width=\linewidth, height=0.12\textheight]{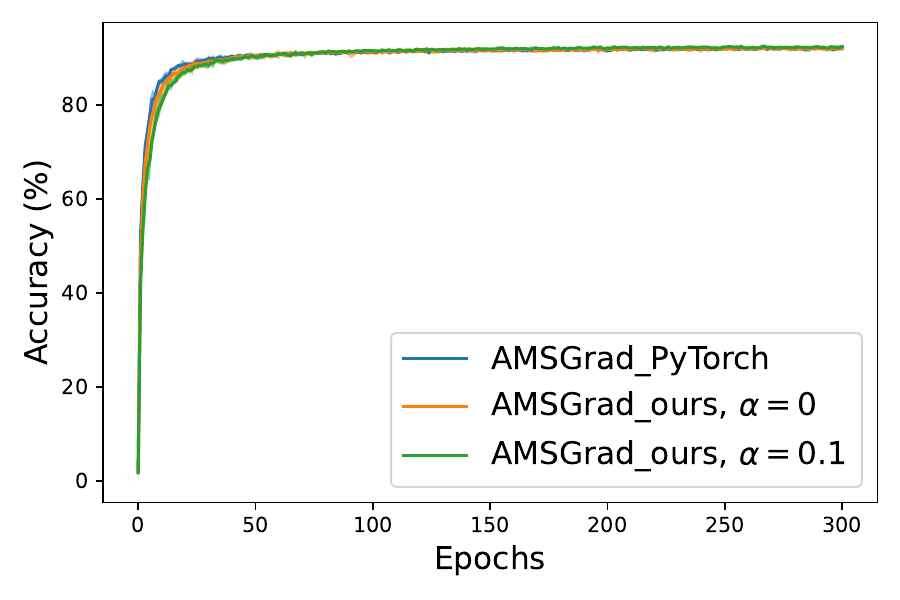}
			\label{Fig:Fig_Test2_Cifar10_AMSGrad_acc}
		\end{minipage}%
	}%
	\subfigure[\scriptsize Test acc. after $200$ epochs]{
		\begin{minipage}[t]{0.24\linewidth}
			\centering
			\includegraphics[width=\linewidth, height=0.12\textheight]{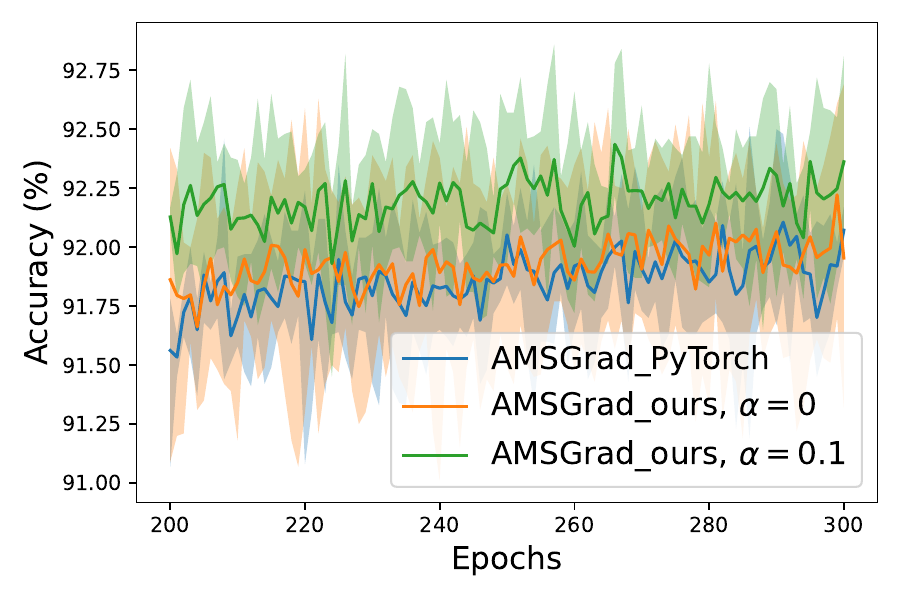}
			\label{Fig:Fig_Test2_Cifar10_AMSGrad_acclater}
		\end{minipage}%
	}%
	\subfigure[\scriptsize Train loss]{
		\begin{minipage}[t]{0.24\linewidth}
			\centering
			\includegraphics[width=\linewidth, height=0.12\textheight]{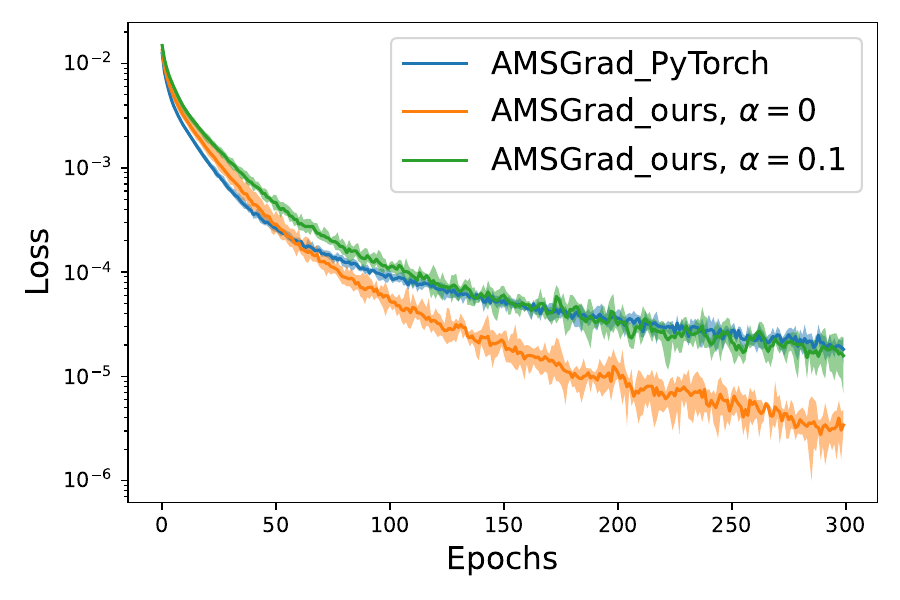}
			\label{Fig:Fig_Test2_Cifar10_AMSGrad_trainloss}
		\end{minipage}%
	}%
	\subfigure[\scriptsize Test loss]{
		\begin{minipage}[t]{0.24\linewidth}
			\centering
			\includegraphics[width=\linewidth, height=0.12\textheight]{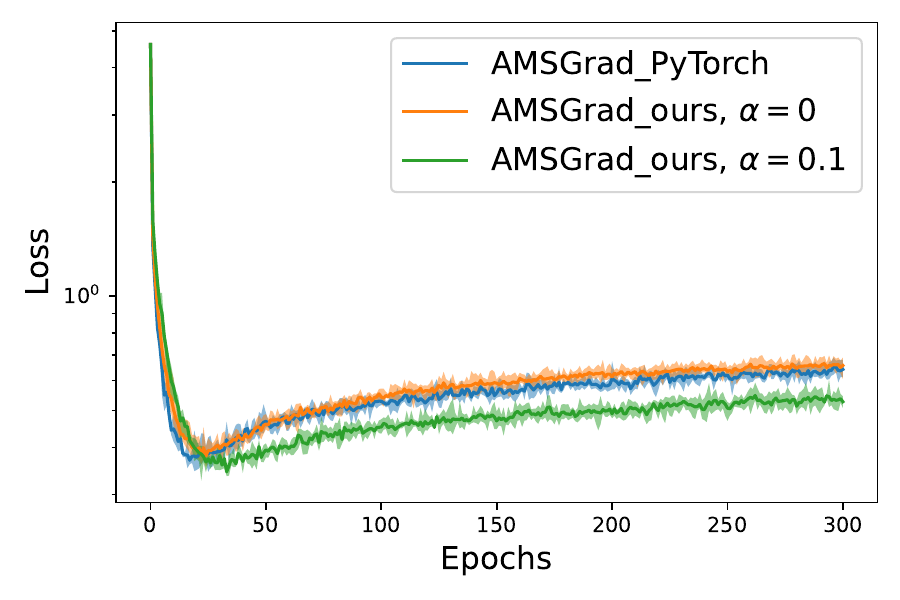}
			\label{Fig:Fig_Test2_Cifar10_AMSGrad_testloss}
		\end{minipage}%
	}%
	
	\subfigure[\scriptsize Test accuracy]{
		\begin{minipage}[t]{0.24\linewidth}
			\centering
			\includegraphics[width=\linewidth, height=0.12\textheight]{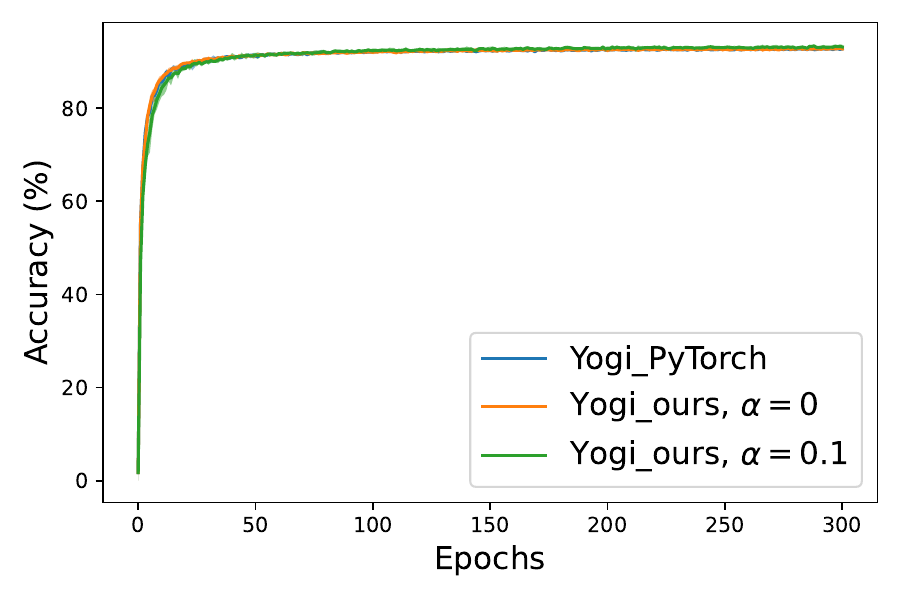}
			\label{Fig:Fig_Test2_Cifar10_Yogi_acc}
		\end{minipage}%
	}%
	\subfigure[\scriptsize Test acc. after $200$ epochs]{
		\begin{minipage}[t]{0.24\linewidth}
			\centering
			\includegraphics[width=\linewidth, height=0.12\textheight]{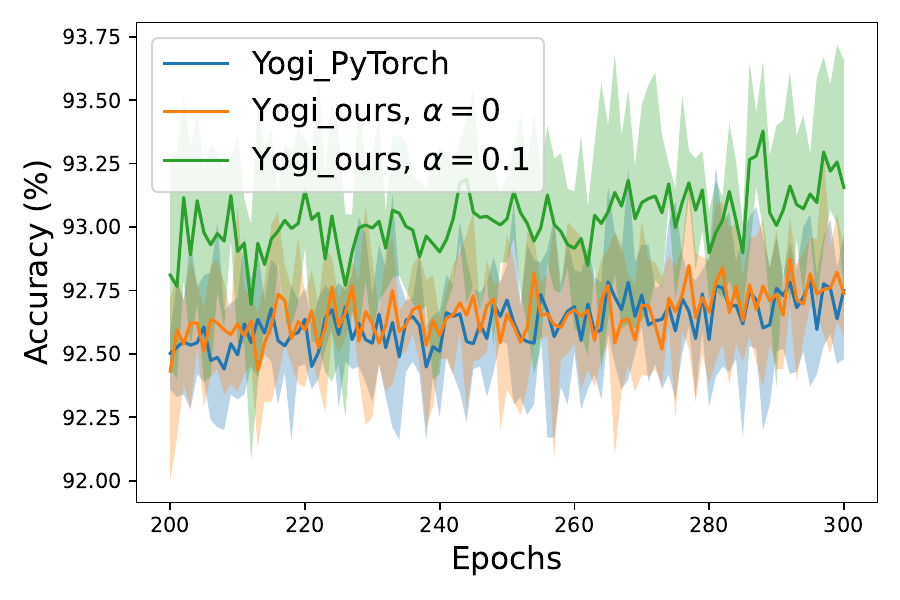}
			\label{Fig:Fig_Test2_Cifar10_Yogi_acclater}
		\end{minipage}%
	}%
	\subfigure[\scriptsize Train loss]{
		\begin{minipage}[t]{0.24\linewidth}
			\centering
			\includegraphics[width=\linewidth, height=0.12\textheight]{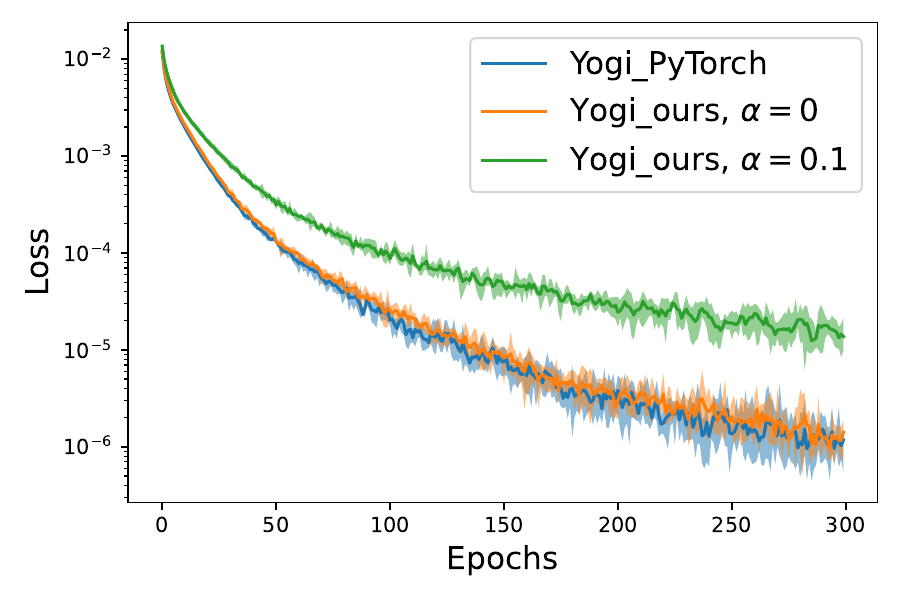}
			\label{Fig:Fig_Test2_Cifar10_Yogi_trainloss}
		\end{minipage}%
	}%
	\subfigure[\scriptsize Test loss]{
		\begin{minipage}[t]{0.24\linewidth}
			\centering
			\includegraphics[width=\linewidth, height=0.12\textheight]{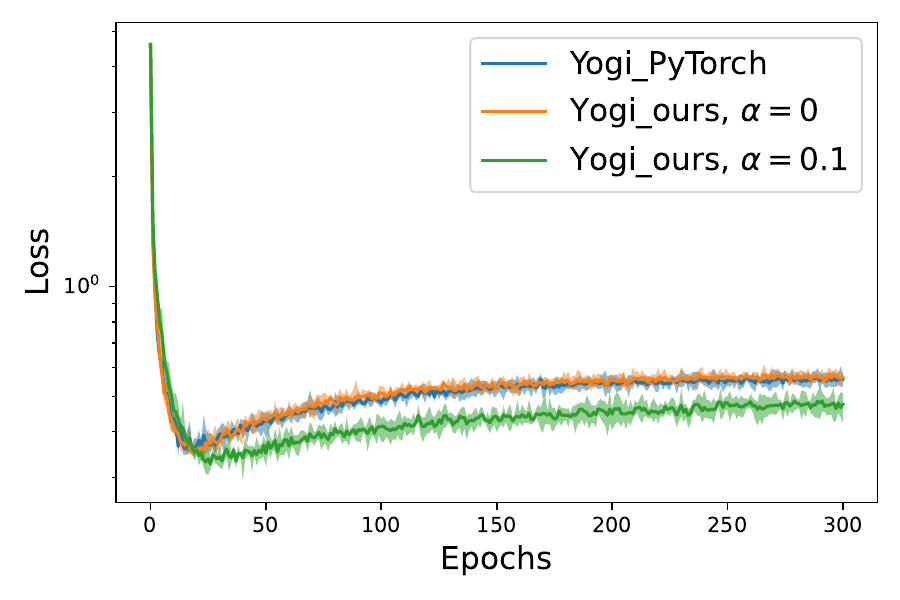}
			\label{Fig:Fig_Test2_Cifar10_Yogi_testloss}
		\end{minipage}%
	}%
	\caption{Test results on CIFAR-10 data set with ResNet50. Here ``acc.'' is the abbreviation of ``accuracy''.}
	\label{Fig_Test2_Cifar10}
\end{figure}

\begin{figure}[tb]
	\centering
	\subfigure[\scriptsize Test accuracy]{
		\begin{minipage}[t]{0.24\linewidth}
			\centering
			\includegraphics[width=\linewidth, height=0.12\textheight]{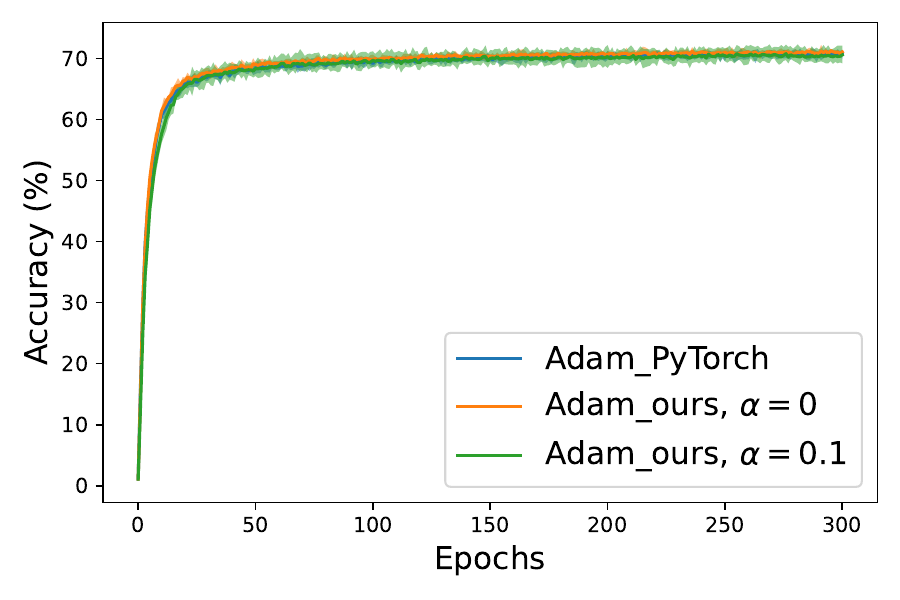}
			\label{Fig:Fig_Test2_Cifar100_Adam_acc}
		\end{minipage}%
	}%
	\subfigure[\scriptsize Test acc. after $200$ epochs]{
		\begin{minipage}[t]{0.24\linewidth}
			\centering
			\includegraphics[width=\linewidth, height=0.12\textheight]{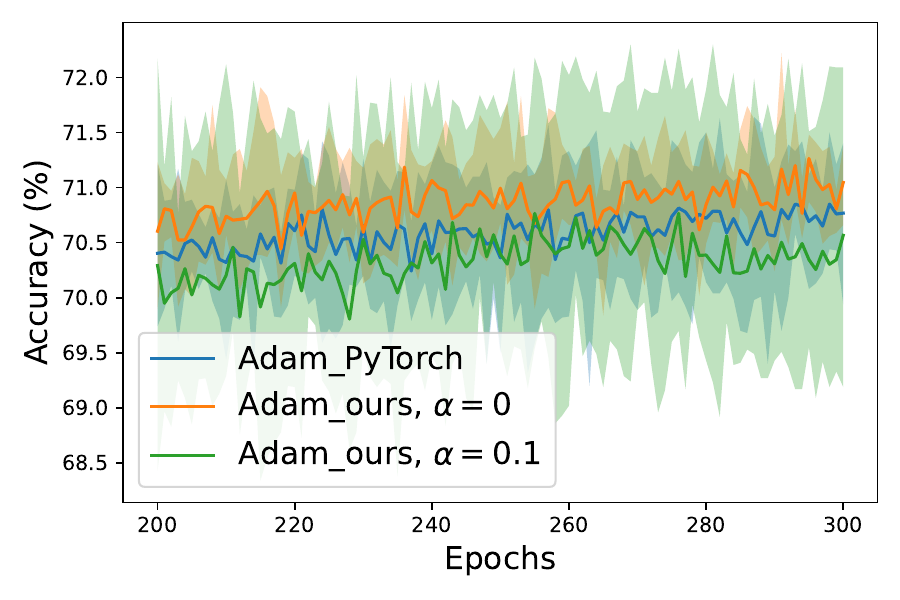}
			\label{Fig:Fig_Test2_Cifar100_Adam_acclater}
		\end{minipage}%
	}%
	\subfigure[\scriptsize Train loss]{
		\begin{minipage}[t]{0.24\linewidth}
			\centering
			\includegraphics[width=\linewidth, height=0.12\textheight]{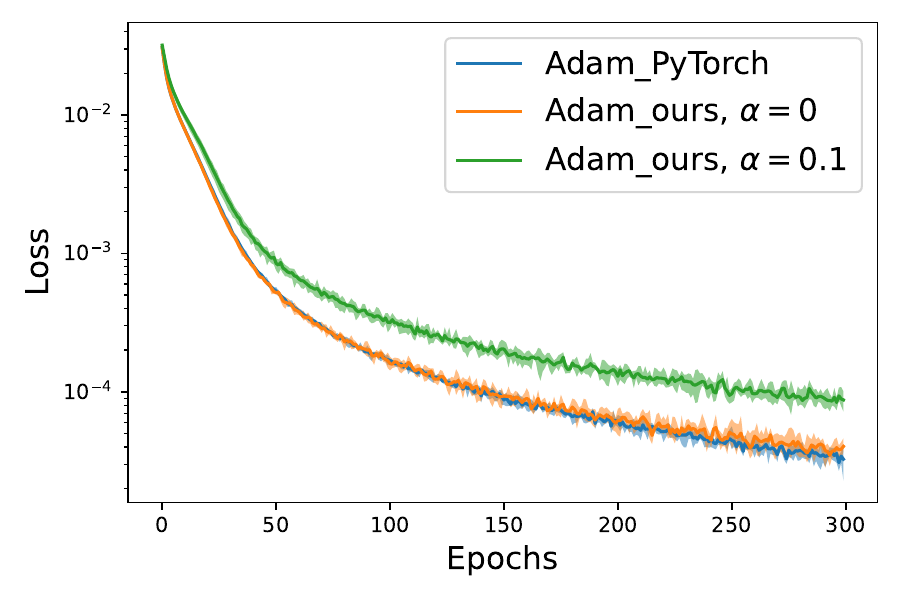}
			\label{Fig:Fig_Test2_Cifar100_Adam_trainloss}
		\end{minipage}%
	}%
	\subfigure[\scriptsize Test loss]{
		\begin{minipage}[t]{0.24\linewidth}
			\centering
			\includegraphics[width=\linewidth, height=0.12\textheight]{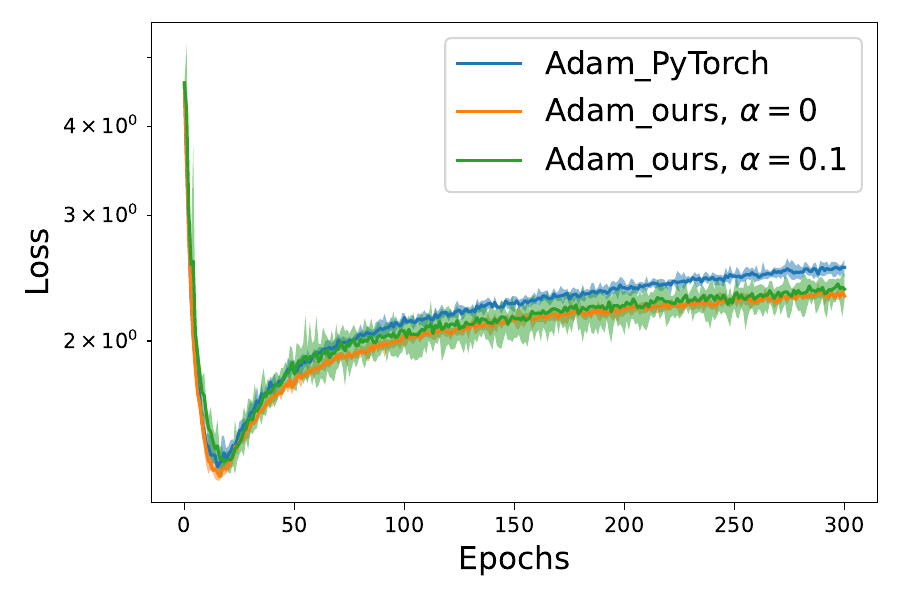}
			\label{Fig:Fig_Test2_Cifar100_Adam_testloss}
		\end{minipage}%
	}%

	\subfigure[\scriptsize Test accuracy]{
		\begin{minipage}[t]{0.24\linewidth}
			\centering
			\includegraphics[width=\linewidth, height=0.12\textheight]{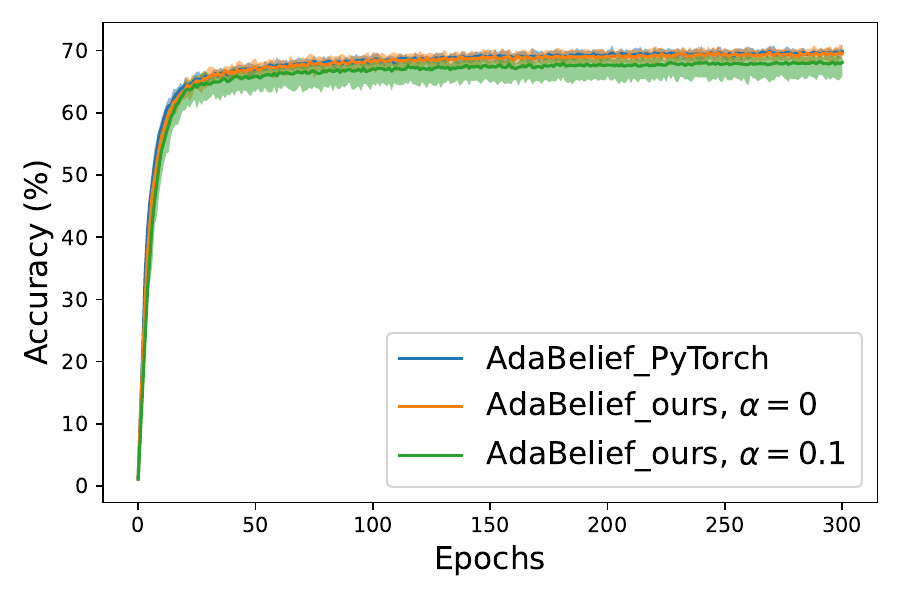}
			\label{Fig:Fig_Test2_Cifar100_AdaBelief_acc}
		\end{minipage}%
	}%
	\subfigure[\scriptsize Test acc. after $200$ epochs]{
		\begin{minipage}[t]{0.24\linewidth}
			\centering
			\includegraphics[width=\linewidth, height=0.12\textheight]{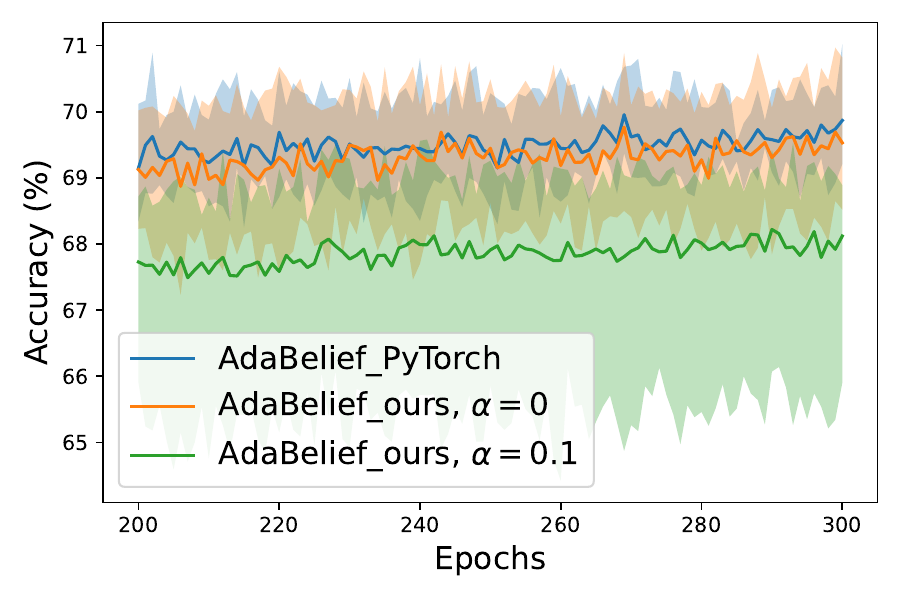}
			\label{Fig:Fig_Test2_Cifar100_AdaBelief_acclater}
		\end{minipage}%
	}%
	\subfigure[\scriptsize Train loss]{
		\begin{minipage}[t]{0.24\linewidth}
			\centering
			\includegraphics[width=\linewidth, height=0.12\textheight]{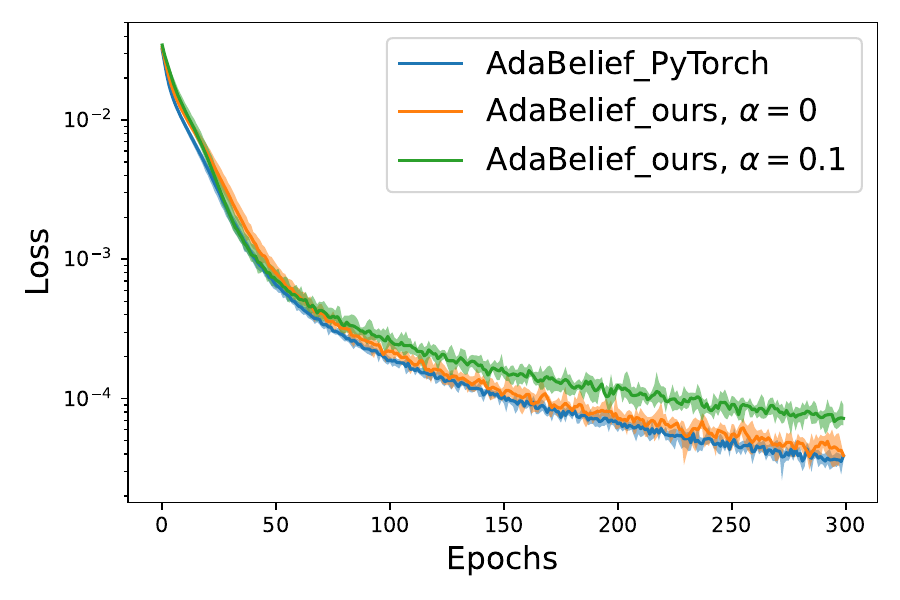}
			\label{Fig:Fig_Test2_Cifar100_AdaBelief_trainloss}
		\end{minipage}%
	}%
\subfigure[\scriptsize Test loss]{
		\begin{minipage}[t]{0.24\linewidth}
			\centering
			\includegraphics[width=\linewidth, height=0.12\textheight]{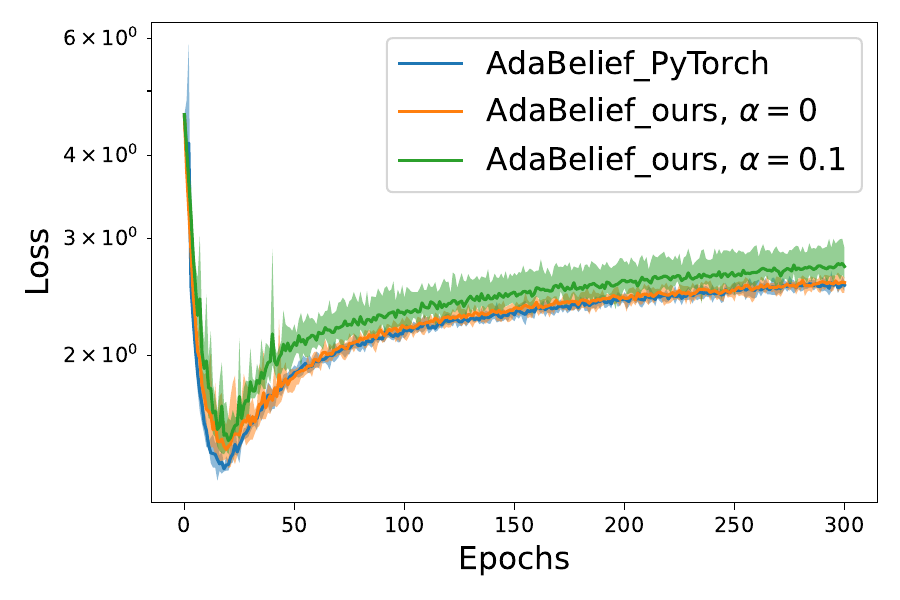}
			\label{Fig:Fig_Test2_Cifar100_AdaBelief_testloss}
		\end{minipage}%
	}%
	
	\subfigure[\scriptsize Test accuracy]{
		\begin{minipage}[t]{0.24\linewidth}
			\centering
			\includegraphics[width=\linewidth, height=0.12\textheight]{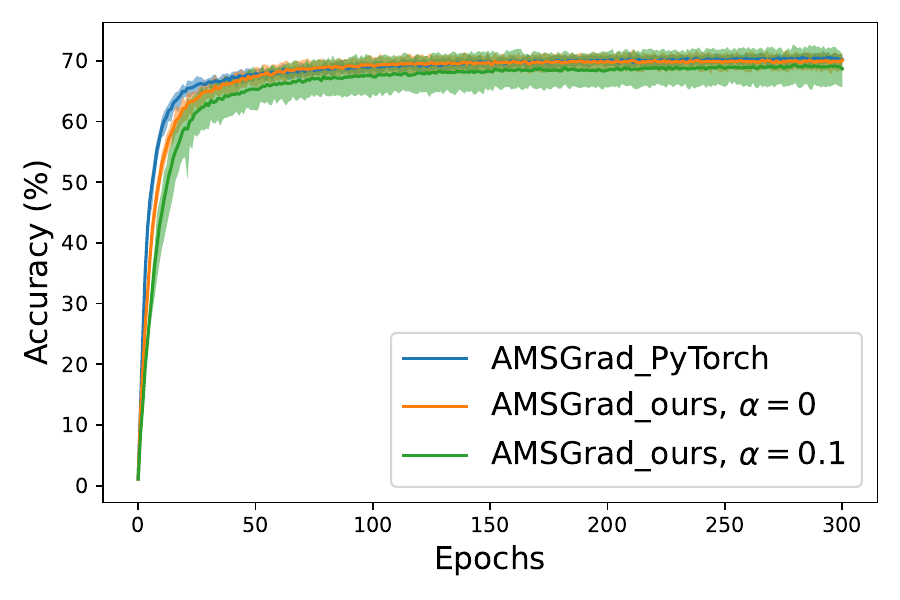}
			\label{Fig:Fig_Test2_Cifar100_AMSGrad_acc}
		\end{minipage}%
	}%
	\subfigure[\scriptsize Test acc. after $200$ epochs]{
		\begin{minipage}[t]{0.24\linewidth}
			\centering
			\includegraphics[width=\linewidth, height=0.12\textheight]{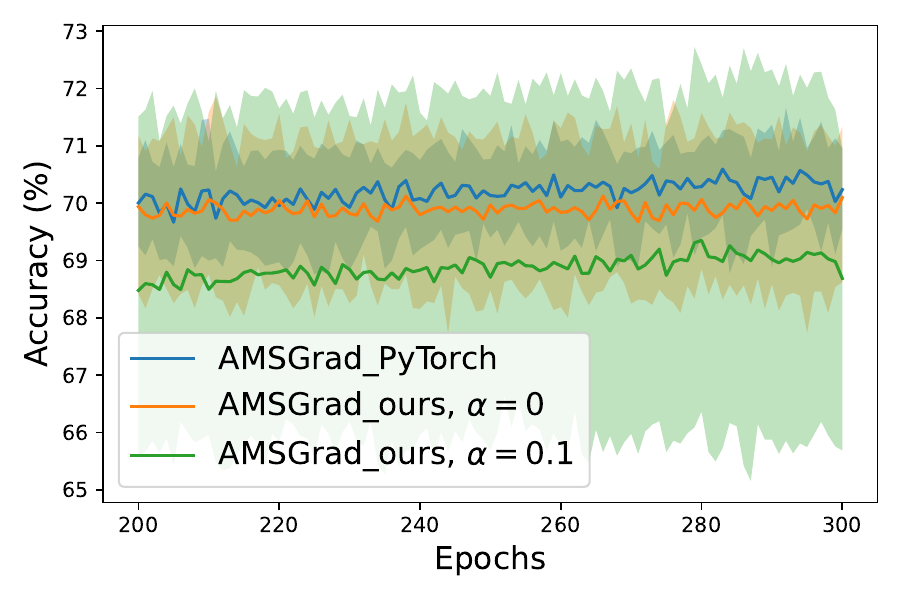}
			\label{Fig:Fig_Test2_Cifar100_AMSGrad_acclater}
		\end{minipage}%
	}%
	\subfigure[\scriptsize Train loss]{
		\begin{minipage}[t]{0.24\linewidth}
			\centering
			\includegraphics[width=\linewidth, height=0.12\textheight]{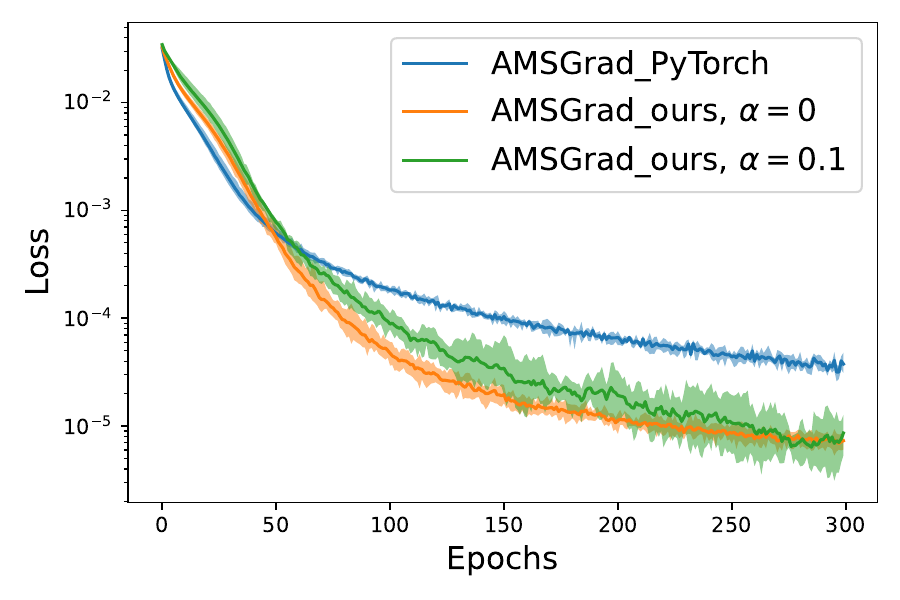}
			\label{Fig:Fig_Test2_Cifar100_AMSGrad_trainloss}
		\end{minipage}%
	}%
	\subfigure[\scriptsize Test loss]{
		\begin{minipage}[t]{0.24\linewidth}
			\centering
			\includegraphics[width=\linewidth, height=0.12\textheight]{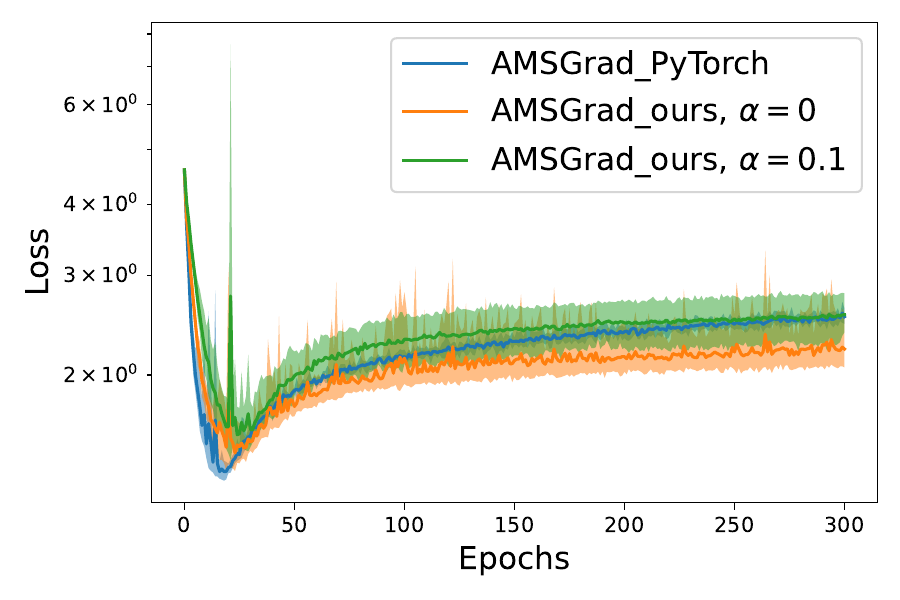}
			\label{Fig:Fig_Test2_Cifar100_AMSGrad_testloss}
		\end{minipage}%
	}%
	
	\subfigure[\scriptsize Test accuracy]{
		\begin{minipage}[t]{0.24\linewidth}
			\centering
			\includegraphics[width=\linewidth, height=0.12\textheight]{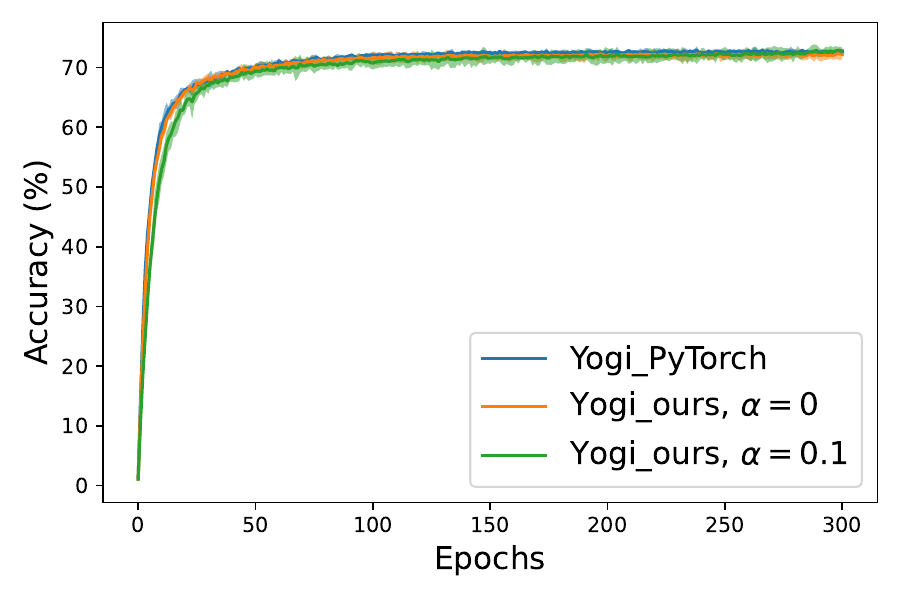}
			\label{Fig:Fig_Test2_Cifar100_Yogi_acc}
		\end{minipage}%
	}%
	\subfigure[\scriptsize Test acc. after $200$ epochs]{
		\begin{minipage}[t]{0.24\linewidth}
			\centering
			\includegraphics[width=\linewidth, height=0.12\textheight]{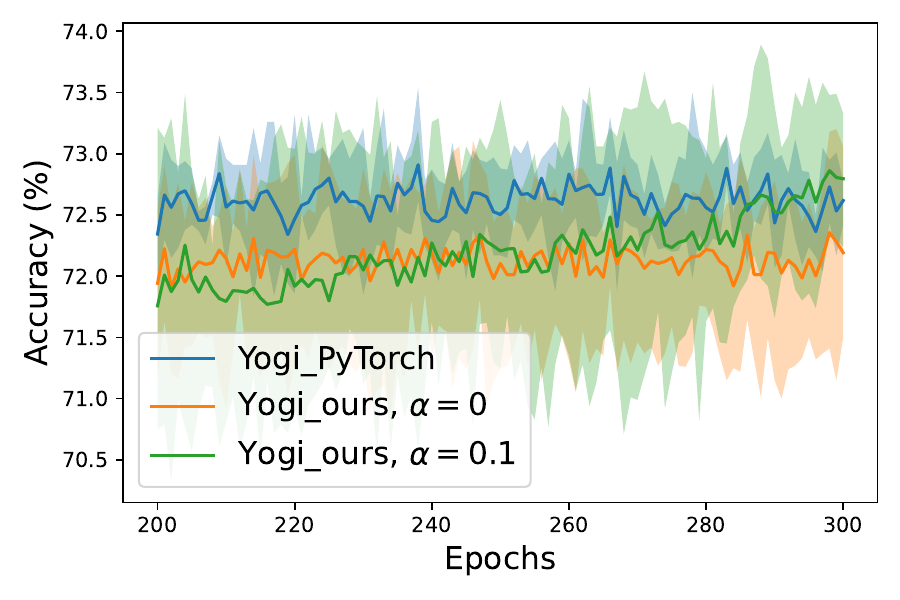}
			\label{Fig:Fig_Test2_Cifar100_Yogi_acclater}
		\end{minipage}%
	}%
	\subfigure[\scriptsize Train loss]{
		\begin{minipage}[t]{0.24\linewidth}
			\centering
			\includegraphics[width=\linewidth, height=0.12\textheight]{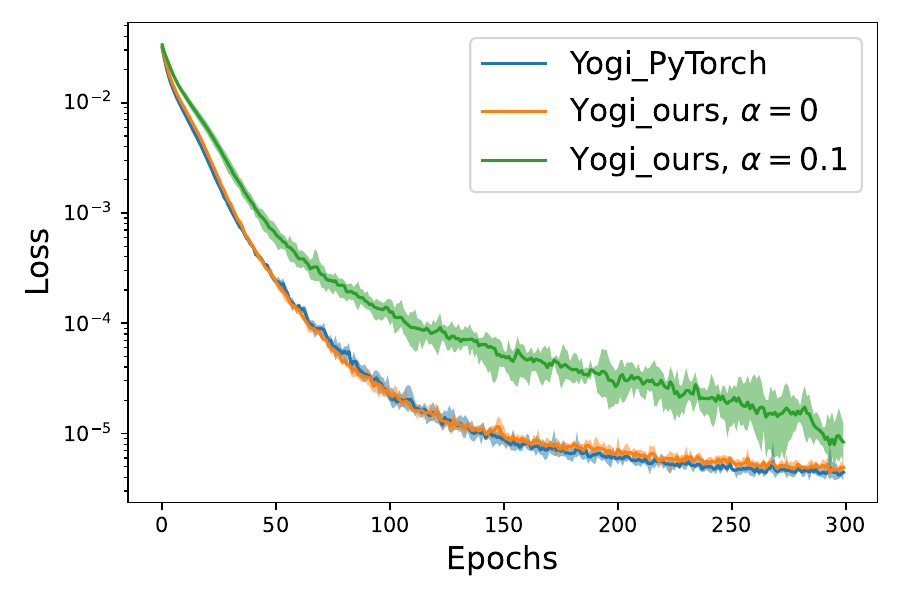}
			\label{Fig:Fig_Test2_Cifar100_Yogi_trainloss}
		\end{minipage}%
	}%
	\subfigure[\scriptsize Test loss]{
		\begin{minipage}[t]{0.24\linewidth}
			\centering
			\includegraphics[width=\linewidth, height=0.12\textheight]{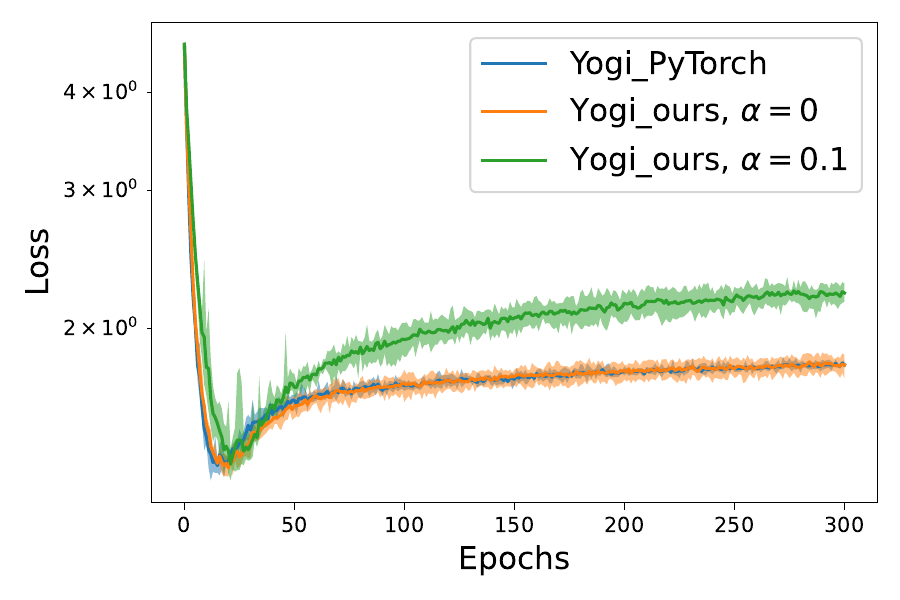}
			\label{Fig:Fig_Test2_Cifar100_Yogi_testloss}
		\end{minipage}%
	}%
	\caption{Test results on CIFAR-100 data set with ResNet50. Here ``acc.'' is the abbreviation of ``accuracy''. }
	\label{Fig_Test2_Cifar100}
\end{figure}

\subsection{Gradient Clipping}

In this subsection, we evaluate the numerical performance of our proposed stochastic subgradient methods with gradient clipping technique by comparing them to the optimizers provided by PyTorch. We conduct numerical experiments using the LeNet \citep{lecun1998gradient} for the classification task on the MNIST data set \citep{lecun1998mnist}.  Following the settings in \citep{grandvalet1997noise,maaten2013learning}, the training samples are randomly perturbed by noise following the Levy stable distribution, with the stability parameter of $1.1$, the skewness parameter of $1$, and the scale of $0.2$. Consequently, the imposed perturbation noise exhibits zero mean without finite second-order moment (hence it is heavy-tailed). In addition, we test the numerical performance of these compared stochastic subgradient methods in training language models \citep{vaswani2017attention}, and present the results in Appendix C.

In our numerical experiments, we set the batch size to $64$ for all test instances and select the regularization parameter $\varepsilon = 10^{-15}$ for all the Adam-family methods. Moreover, at the $k$-th epoch, we choose the stepsize as $\eta_k = \frac{\eta_0}{\sqrt{k+1}}$
 for all tested algorithms. For all compared optimizers, we choose the initial stepsize $\eta_0$ and the momentum parameters $\tau_1, \tau_2$ using the same grid search method as in Section \ref{Section_numerical_adaptive}, and retain all other parameters at their default values for the optimizers in PyTorch. We run each test instance $5$ times with different random seeds. In each test instance, all compared algorithms are tested using the same random seed and initialized with the same random weights by the default initialization function in PyTorch.

The numerical results are presented in Figure \ref{Fig_Test3_Gradient_Clipping}. These figures indicate that our proposed \ref{Eq_SGDM} and \ref{Eq_ADAMC} converge successfully and achieve high accuracy. In contrast, without gradient clipping, SGD fails to converge and Adam converges much slower than \ref{Eq_ADAMC}.  Moreover, compared with \ref{Eq_SGDM}, \ref{Eq_ADAMC} achieves improved accuracy and a faster decrease in the loss curve. Therefore, we can conclude that with the gradient clipping technique, our proposed Adam-family method \eqref{Eq_abstract_AS_framework} outperforms \ref{Eq_SGDM} and the Adam provided in PyTorch. These observations further demonstrate the great potential of our proposed stochastic subgradient methods with gradient clipping in solving \ref{Prob_Ori} in the presence of heavy-tailed noises.

\begin{figure}[tb]
	\centering
	\subfigure[Test accuracy]{
		\begin{minipage}[t]{0.33\linewidth}
			\centering
			\includegraphics[width=\linewidth, height=0.15\textheight]{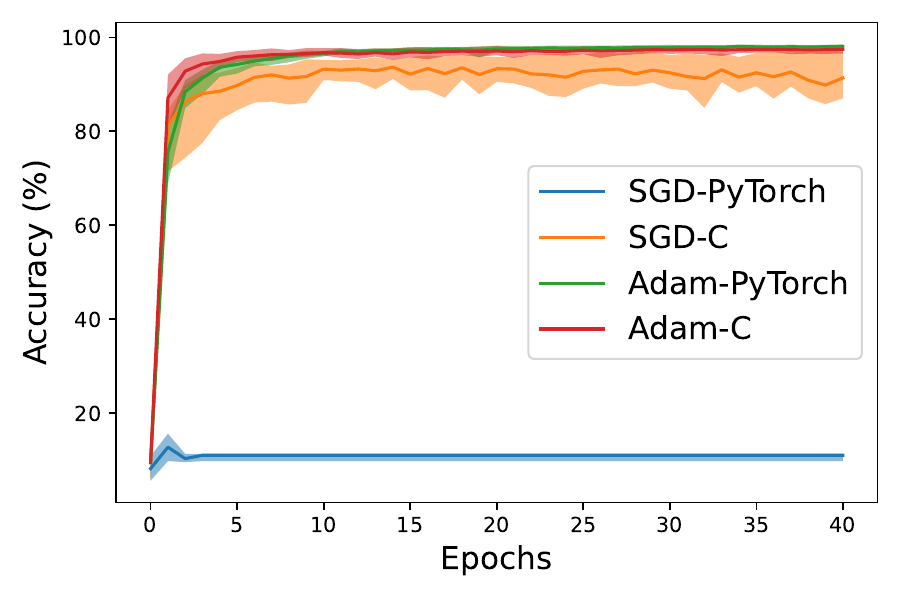}
			\label{Fig:Test3_GC_acc}
		\end{minipage}%
	}%
	\subfigure[Train loss]{
		\begin{minipage}[t]{0.33\linewidth}
			\centering
			\includegraphics[width=\linewidth, height=0.15\textheight]{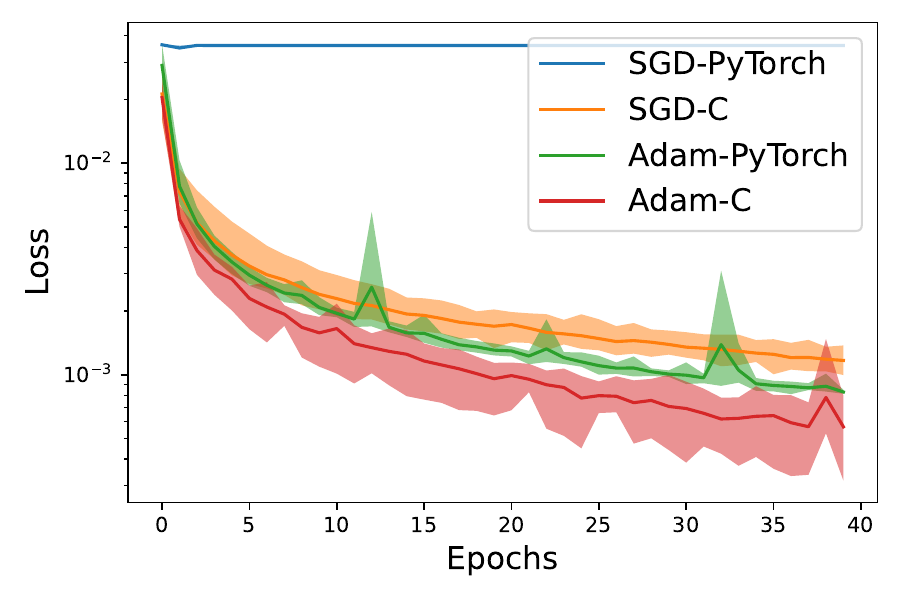}
			\label{Fig:Test3_GC_trainloss}
		\end{minipage}%
	}%
	\subfigure[Test loss]{
		\begin{minipage}[t]{0.33\linewidth}
			\centering
			\includegraphics[width=\linewidth, height=0.15\textheight]{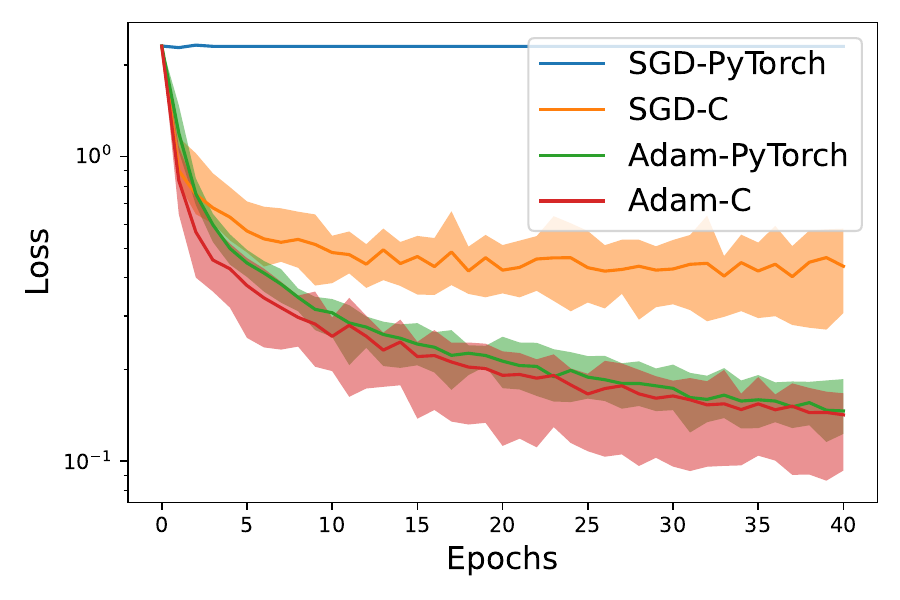}
			\label{Fig:Test3_GC_testloss}
		\end{minipage}%
	}%
	\caption{Test results on MNIST data set with LeNet.}
	\label{Fig_Test3_Gradient_Clipping}
\end{figure}

\section{Conclusion}

Adam-family methods are powerful tools for nonsmooth optimization, especially in training neural networks. However, as most of the neural networks are built from nonsmooth blocks, their loss functions are typically nonsmooth and not Clarke regular, thus leading to great difficulties in analyzing the convergence properties for these methods. Additionally, the presence of heavy-tailed evaluation noises in numerous applications of \ref{Prob_Ori}  poses significant challenges in designing efficient algorithms and establishing theoretical guarantees for \ref{Prob_Ori}. 

The primary contributions of this paper can be summarized as follows:
\begin{itemize}
	\item {\bf A novel framework for Adam-family methods} \\
	To establish convergence properties for Adam-family methods, we first introduce a two-timescale framework \eqref{Eq_framework} that assigns different stepsizes to the updating directions and evaluation noises, respectively. Then we establish convergence properties for \eqref{Eq_framework} in the sense of conservative field under mild assumptions.  Furthermore, we prove that under mild assumptions with almost every initialized stepsize and initial point, any cluster point of the sequence generated by our proposed framework is a Clarke stationary point of the objective function. These results provide theoretical guarantees for our proposed framework \eqref{Eq_framework}. 
	In particular, although AD algorithms may introduce spurious stationary points to \ref{Prob_Ori}, we prove that our proposed framework \eqref{Eq_framework} can avoid these spurious stationary points for almost every initial point and stepsize. 
	\item {\bf Convergence properties for Adam-family methods}\\
	We show that Adam, AdaBelief, AMSGrad, NAdam and Yogi, when equipped with diminishing stepsizes, follow our proposed framework \eqref{Eq_framework}. Consequently, through our established results for \eqref{Eq_framework}, we provide a convergence analysis for these Adam-family methods under mild assumptions in the sense of both conservative field and Clarke subdifferential. These results are applicable to a wide range of neural network training problems, hence providing convergence guarantees for the application of these Adam-family methods in training nonsmooth neural networks.
	\item {\bf Gradient clipping technique for heavy-tailed noises} \\
	We develop stochastic subgradient methods that incorporate the gradient clipping technique based on our proposed framework. Under mild assumptions and appropriately chosen clipping parameters, we show that these stochastic subgradient methods conform to our proposed framework \eqref{Eq_framework} even when the evaluation noises are only assumed to be integrable. Therefore, by employing the gradient clipping technique to tackle heavy-tailed evaluation noises, a wide range of stochastic subgradient methods can be developed with guaranteed convergence properties for solving \ref{Prob_Ori}. 
\end{itemize}

Furthermore, we conduct extensive numerical experiments to illustrate that our proposed Adam-family methods are as efficient as the widely employed Adam-family methods provided by PyTorch. Additionally, preliminary numerical experiments demonstrate the high efficiency and robustness of our proposed stochastic subgradient methods with gradient clipping in training neural networks with heavy-tailed evaluation noises. Therefore, we can conclude that our results have provided theoretical guarantees for Adam-family methods in practical settings, especially when the neural networks are nonsmooth or the evaluation noises are heavy-tailed. 

Future research questions of this work include establishing convergence rates and complexity results for Adam-family methods in minimizing nonsmooth and non-regular functions, which are extremely challenging to tackle.  Most existing works focus on SGD with the exact evaluation of the Clarke subdifferential, or only consider the convergence rate of the trajectories of the corresponding noiseless differential inclusions \citep{castera2021inertial}. To the best of our knowledge, there is no existing result for establishing the complexity for stochastic subgradient methods in the form of Algorithm \ref{Alg:ADAM} when $f$ is only assumed to be a potential function.  Furthermore, as the gradient clipping technique is widely employed in various natural language processing tasks, future works of this paper could investigate the performance of our proposed stochastic subgradient methods with gradient clipping in these real-world applications of \ref{Prob_Ori}. 

\section*{Acknowledgement}
We thank the reviewers for their valuable suggestions that have significantly contributed to the improvement of our paper.

The research of Nachuan Xiao and Kim-Chuan Toh is supported by Academic Research Fund Tier 3 grant call (MOE-2019-T3-1-010). 
The research of Xiaoyin Hu is supported by the National Natural Science Foundation of China (No. 12301408), Zhejiang Provincial Natural Science Foundation of China under Grant (No. LQ23A010002), Scientific research project of Zhejiang Provincial Education Department (Y202248716), and the advanced computing resources provided by the Supercomputing Center of HZCU. 
The research of Xin Liu is supported in part by the National Natural Science Foundation of China (12125108, 12226008, 12021001, 12288201, 11991021), Key Research Program of Frontier Sciences, Chinese Academy of Sciences (ZDBS-LY-7022), and CAS AMSS-PolyU Joint Laboratory of Applied Mathematics.

\appendix

\section{Proof for Proposition \ref{Prop_Clipping_mk_uniformly_bounded}}
In this section, we present the proof for Proposition \ref{Prop_Clipping_mk_uniformly_bounded}. We begin our proof with the following auxiliary proposition, which employs a similar proof technique as \citep[Proposition 4.10]{xiao2023convergence}. 

\begin{prop}
	\label{Prop_clipping_noise_estimation}
	Suppose $\{\xi_k\}$ is a sequence of uniformly bounded martingale difference sequence, $\{\eta_k\}$ and $\{C_k\}$ are positive sequences that satisfies 
    \begin{equation*}
        \lim_{k\to +\infty} \eta_k = 0,\quad \lim_{k\to +\infty} C_k = +\infty, \quad \text{and} \quad  \lim_{k\to +\infty} C_k^2 \eta_k \log(k) =0. 
    \end{equation*}
    Then almost surely, it holds that 
	\begin{equation*}
		\lim_{k \to +\infty}  \sum_{i = 
			0}^{k}\left(\eta_{i}\prod_{ j=i+1}^k(1- \eta_j) \right)C_i \xi_{i} = 0. 
	\end{equation*}
\end{prop}
\begin{proof}
	Let $z_k =   \sum_{i = 
		0}^{k}\left(\eta_{i}\prod_{ j=i+1}^k(1- \eta_j) \right)C_i \xi_{i}$ and $z_{0} = 0$, $\rho_{k, i} :=  C_i\eta_{i}\prod_{ j=i+1}^k(1- \eta_j)$ and $\rho_{k,k} := C_k \eta_k$, then there exists $K>0$ such that $|\rho_{k, i}| \leq C_i \eta_i$ holds for any $k\geq i\geq K$. Without loss of generality, we assume that $\rho_{k,i} \geq 0$ holds for any $k\geq i\geq K$. Moreover, from the expression of $z_k$, we can conclude that 
	\begin{equation*}
		z_k = \sum_{i = 0}^k \rho_{k,i} \xi_{i}. 
	\end{equation*}
	
	Since the martingale difference sequence $\{\xi_{m,k}\}$ is  uniformly bounded, it holds that $\xi_{m,k}$ is sub-Gaussian for any $k\geq 0$. Then there exists a constant $M>0$ such that for any $w \in \Rn$, it holds for any $k\geq 0$ that 
	\begin{equation*}
		\bb{E}\left[ \exp\left( \inner{w, \xi_{m, k+1}} \right) | \ca{F}_k \right] \leq \exp\left( \frac{M}{2}\norm{w}^2 \right).
	\end{equation*}
	Therefore, for any $s>K$, $T>0$,  $w \in \Rn$ and any $C > 0$, let 
	\begin{equation*}
		Z_i := \exp\left[  \inner{Cw, \sum_{k = s}^i \rho_{\Lambda(\lambda_s + T), k} \xi_{m, k}} - \frac{MC^2}{2}\sum_{k = s}^i \rho_{\Lambda(\lambda_s + T), k}^2 \norm{w}^2  \right],
	\end{equation*}
	where $\Lambda(0) := 0$, $\Lambda(i)  := \sum_{k = 0}^{i-1} \eta_k$, and $\Lambda(t) := \sup  \{k \geq 0: t\geq \Lambda(k)\} $. 
	Then  for any $i\geq s$, we have that $\bb{E}[Z_{i+1} | \ca{F}_i] \leq Z_{i}$. 
	Hence for any $\delta > 0$, and any $C > 0$, it holds that 
	\begin{equation*}
		\begin{aligned}
			&\bb{P}\left( \sup_{s\leq i \leq \Lambda(\lambda_s + T)} \inner{w, \sum_{k = s}^i \rho_{\Lambda(\lambda_s + T), k} \xi_{k}} > \delta \right)\\
			={}&\bb{P}\left( \sup_{s\leq i \leq \Lambda(\lambda_s + T)} \inner{Cw, \sum_{k = s}^i \rho_{\Lambda(\lambda_s + T), k} \xi_{k}} > C\delta \right)\\
			\leq{}& \bb{P}\left( \sup_{s\leq i \leq \Lambda(\lambda_s + T)} Z_i > \exp\left( C\delta - \frac{MC^2}{2} \sum_{k = s}^{\Lambda(\lambda_s + T)} \rho_{\Lambda(\lambda_s + T), k}^2 \norm{w}^2 \right) \right)\\
			\leq{}& \exp\left( \left(\frac{M}{2}\norm{w}^2 \sum_{k = s}^{\Lambda(\lambda_s + T)} \rho_{\Lambda(\lambda_s + T), k}^2\right)C^2   - \delta C \right).
		\end{aligned}
	\end{equation*}
	Here the second inequality holds since $\{Z_i\}$ is nonnegative, $\bb{E}[Z_{i+1}|\ca{F}_{i}] \leq Z_i$ holds for any $i\geq 0$ and $\bb{E}[Z_{s}] \leq 1$. Then from the arbitrariness of $C$, we can set $C = \frac{\delta}{M\norm{w}^2 \sum_{k = s}^{\Lambda(\lambda_s + T)}\rho_{\Lambda(\lambda_s + T), k}^2}$ to obtain that 
	\begin{equation*}
		\bb{P}\left( \sup_{s\leq i \leq \Lambda(\lambda_s + T)} \inner{w, \sum_{k = s}^i \rho_{\Lambda(\lambda_s + T), k} \xi_{k}} > \delta \right) \leq  \exp \left(\frac{-\delta^2}{2M\norm{w}^2 \sum_{k = s}^{\Lambda(\lambda_s + T)}\rho_{\Lambda(\lambda_s + T), k}^2 }\right). 
	\end{equation*}
	From the arbitrariness of $w$ and the fact that $\rho_{\Lambda(\lambda_s + T), k} \leq \eta_k$, we can deduce that there exists constants $C_1, C_2$ that only depend on $n$, such that 
	\begin{equation*}
		\begin{aligned}
			& \bb{P}\left( \sup_{s\leq i \leq \Lambda(\lambda_s + T)} \norm{\sum_{k = s}^i \rho_{\Lambda(\lambda_s + T), k} \xi_{m,k}} > \delta \right) \\
			\leq{}& C_1\exp\left(\frac{-\delta^2}{2C_2M\sum_{k = s}^{\Lambda(\lambda_s + T)} \rho_{\Lambda(\lambda_s + T), k}^2 }\right) 
			\leq C_1\exp\left(\frac{-\delta^2}{2C_2M\sum_{k = s}^{\Lambda(\lambda_s + T)} C_k^2\eta_k^2 }\right) \\
                \leq{}& C_1\exp\left(\frac{-\delta^2}{2C_2M\eta_{k'} C_{k'}^2\sum_{k = s}^{\Lambda(\lambda_s + T)} \eta_k }\right) 
                \leq  \exp\left(\frac{-\delta^2}{2MT\eta_{k'}C_{k'}^2 }\right),
		\end{aligned}
	\end{equation*}
	holds for some $k' \in [s, \Lambda(\lambda_s + T)]$. 
	
	Therefore, for any $j \geq 0$, there exists $k_j\in [\Lambda(jT), \Lambda((j+1)T) ]$, such that
	\begin{equation}
		\label{Eq_Prop_twotimescale_noise_estimation_0}
		\begin{aligned}
			&\sum_{j=0}^{+\infty}\bb{P}\left( \sup_{\Lambda(jT)\leq i \leq \Lambda(jT + T)}  \norm{\sum_{k = \Lambda(jT)}^i \rho_{\Lambda(jT + T), k }\xi_{k}} > \delta \right) \\
			\leq{}& \sum_{j=0}^{+\infty} \exp\left( \frac{-\delta^2}{2MT \eta_{k_j}C_{k_j}^2} \right) \leq \sum_{k=0}^{+\infty} \exp\left( \frac{-\delta^2}{2MT \eta_k C_k^2} \right) < +\infty. 
		\end{aligned}
	\end{equation}
	Here the last inequality holds from the fact that $\lim_{k \to +\infty} \eta_k  C_k^2\log(k) = 0 $. 
 
	Therefore, let $\ca{E}_j$ denote the event $\left\{\sup_{\Lambda(jT)\leq i \leq \Lambda( jT+T)}\norm{ \sum_{k = \Lambda(jT)}^{i}\rho_{\Lambda(jT + T), k} \xi_{k}} > \delta \right\}$. From  the Borel-Cantelli lemma and \eqref{Eq_Prop_twotimescale_noise_estimation_0}, we can conclude that $ \bb{P}\left( \lim_{j\to +\infty} \cap_{j = 1}^{+\infty} \cup_{l = j}^{+\infty} \ca{E}_l  \right) = 0$. Therefore, we can conclude that, almost surely, 
	\begin{equation*}
		\lim_{j\to +\infty} \sup_{\Lambda(jT)\leq i \leq \Lambda( jT+T)}\norm{ \sum_{k = \Lambda(jT)}^{i}\rho_{\Lambda(jT + T), k} \xi_{k}} \leq \delta. 
	\end{equation*}
	Then the arbitrariness of $\delta$ illustrates that, almost surely, we have
	\begin{equation}
		\label{Eq_Prop_twotimescale_noise_estimation_1}
		\lim_{j \to +\infty} \sup_{\Lambda(jT)\leq i \leq \Lambda( jT+T)}\norm{ \sum_{k = \Lambda(jT)}^{i}\rho_{\Lambda(jT + T), k} \xi_{k}} = 0. 
	\end{equation}
	Notice that for any $j\geq 0$ such that $\Lambda(jT) \geq K$, it holds that 
	\begin{equation*}
		z_{\Lambda( jT+T)} = \left(\prod_{k = \Lambda( jT)}^{\Lambda( jT+T)} (1-\eta_k)\right) z_{\Lambda( jT)} + \sum_{k = \Lambda( jT)}^{\Lambda( jT+T)}  \rho_{\Lambda(jT + T), k} \xi_{k},
	\end{equation*}
	which illustrates that almost surely, 
	\begin{equation}
		\label{Eq_Prop_twotimescale_noise_estimation_2}
		\norm{z_{\Lambda( jT+T)}} \leq \exp(-T) \norm{z_{\Lambda( jT)}} + \norm{\sum_{k = \Lambda( jT)}^{\Lambda( jT+T)}  \rho_{\Lambda(jT + T), k} \xi_{m,k}}, 
	\end{equation}
        hence $\lim_{j\to +\infty} \norm{z_{\Lambda( jT)}} = 0$. 

        Finally, for any $i$ such that $\Lambda( jT) \leq i \leq \Lambda( jT+T)$, it holds that 
        \begin{equation*}
            \begin{aligned}
                &\norm{z_{\Lambda( jT+T)}} = \norm{ \left(\prod_{k = i}^{\Lambda( jT+T)} (1-\eta_k)\right) z_{i} + \sum_{k = i}^{\Lambda( jT+T)}  \rho_{\Lambda(jT + T), k} \xi_{k} } \\
                \geq{}& \exp(-T) \norm{z_i} - \norm{\sum_{k = i}^{\Lambda( jT+T)}  \rho_{\Lambda(jT + T), k} \xi_{k} } \\
                \geq{}& \exp(-T) \norm{z_i} - 2 \sup_{\Lambda( jT)\leq i \leq \Lambda(\lambda_s + T)} \norm{\sum_{k = \Lambda( jT)}^i \rho_{\Lambda(\lambda_s + T), k} \xi_{k}}.
            \end{aligned}
        \end{equation*}
        As a result, we have 
        \begin{equation}
            \label{Eq_Prop_twotimescale_noise_estimation_3}
            \begin{aligned}
                &\sup_{\Lambda( jT) \leq i \leq \Lambda( jT+T)} \norm{z_{i}} \\
                \leq{}& \exp(T) \left( \norm{z_{\Lambda( jT)}} + 2  \sup_{\Lambda( jT)\leq i \leq \Lambda(\lambda_s + T)} \norm{\sum_{k = \Lambda( jT)}^i \rho_{\Lambda(\lambda_s + T), k} \xi_{k}} \right). 
            \end{aligned}
        \end{equation}
	Combining \eqref{Eq_Prop_twotimescale_noise_estimation_1}, \eqref{Eq_Prop_twotimescale_noise_estimation_2}, and \eqref{Eq_Prop_twotimescale_noise_estimation_3} together, we achieve that, 
	\begin{equation*}
		\begin{aligned}
		    &\mathop{\lim\sup}_{k \to +\infty} \norm{z_k} = \lim_{j\to +\infty} \sup_{\Lambda( jT) \leq i \leq \Lambda( jT+T)} \norm{z_{i}} \\
                \leq {}& \exp(T)\lim_{j \to +\infty}\norm{z_{\Lambda( jT)}} + 2\exp(T)\lim_{j\to +\infty}  \sup_{\Lambda( jT)\leq i \leq \Lambda(\lambda_s + T)} \norm{\sum_{k = \Lambda( jT)}^i \rho_{\Lambda(\lambda_s + T), k} \xi_{k}}\\
                ={}& 0. 
		\end{aligned}
	\end{equation*}
	holds almost surely.  This completes the proof.

\end{proof}

With Proposition \ref{Prop_clipping_noise_estimation}, we now present the proof for Proposition \ref{Prop_Clipping_mk_uniformly_bounded}. \newline
\textbf{Proof of Proposition \ref{Prop_Clipping_mk_uniformly_bounded}:}

    For any $k \geq 0$, the $\mkp$ in \ref{Eq_SGDM} can be expressed as 
    \begin{equation*}
        \begin{aligned}
            &\mkp = \sum_{i = 0}^k  \left(\prod_{j = i+1}^k (1-\tau_1 \eta_j) \right) \tau_1 \eta_i \hat{g}_i\\
            ={}& \sum_{i = 0}^k  \left(\prod_{j = i+1}^k (1-\tau_1 \eta_j) \right) \tau_1 \eta_i (d_i + C_i \xi_i)\\
            ={}& \sum_{i = 0}^k  \left(\prod_{j = i+1}^k (1-\tau_1 \eta_j) \right) \tau_1 \eta_i d_i + \sum_{i = 0}^k  \left(\prod_{j = i+1}^k (1-\tau_1 \eta_j) \right) \tau_1 \eta_i C_i\xi_i. 
        \end{aligned}
    \end{equation*}
    Here we set $\prod_{i+1}^i (1-\tau_1 \eta_j) = 1$ for simplicity. 
    
    As illustrated in Proposition \ref{Prop_clipping_noise_estimation}, almost surely, it holds that  
    \begin{equation*}
        \lim_{k\to +\infty} \sum_{i = 0}^k  \left(\prod_{j = i+1}^k (1-\tau_1 \eta_j) \right) \tau_1 \eta_i C_i\xi_i = 0,
    \end{equation*}
    which implies that  $\sup_{k\geq  0}\norm{\sum_{i = 0}^k  \left(\prod_{j = i+1}^k (1-\tau_1 \eta_j) \right) \tau_1 \eta_i C_i\xi_i} < +\infty$. 
    
    On the other hand, Lemma \ref{Le_clipping_convergence_dk} illustrates that there exists a nonnegative diminishing sequence $\{\delta_k\}$ such that  $d_k \in \D_f^{\delta_k}(\xk)$ holds for any $k\geq 0$. Then from the local boundedness of $\D_f$ and the uniform boundedness of the sequence $\{\xk\}$, we have that $\sup_{k\geq 0} \norm{d_k} < +\infty$ holds almost surely. 
    Then for any $k > 0$, it holds that 
    \begin{equation*}
        \norm{\sum_{i = 0}^k  \left(\prod_{j = i+1}^k (1-\tau_1 \eta_j) \right) \tau_1 \eta_i d_i} \leq \sum_{i = 0}^k  \left(\prod_{j = i+1}^k (1-\tau_1 \eta_j) \right) \tau_1 \eta_i\norm{ d_i} \leq \sup_{ 0\leq i \leq k} \norm{d_i}. 
    \end{equation*}
    Then we can conclude that 
    \begin{equation*}
        \sup_{k\geq 0} \norm{\mk} \leq \sup_{k\geq 0} \norm{d_k} + \sup_{k\geq  0}\norm{\sum_{i = 0}^k  \left(\prod_{j = i+1}^k (1-\tau_1 \eta_j) \right) \tau_1 \eta_i C_i\xi_i} <+\infty. 
    \end{equation*}
    This completes the proof of Proposition \ref{Prop_Clipping_mk_uniformly_bounded}.

\section{Supplementary Numerical Experiments for Section 6.1}

\begin{figure}[tb]
	\centering
	\subfigure[Test accuracy, CIFAR-10]{
		\begin{minipage}[t]{0.48\linewidth}
			\centering
			\includegraphics[width=\linewidth, height=0.20\textheight]{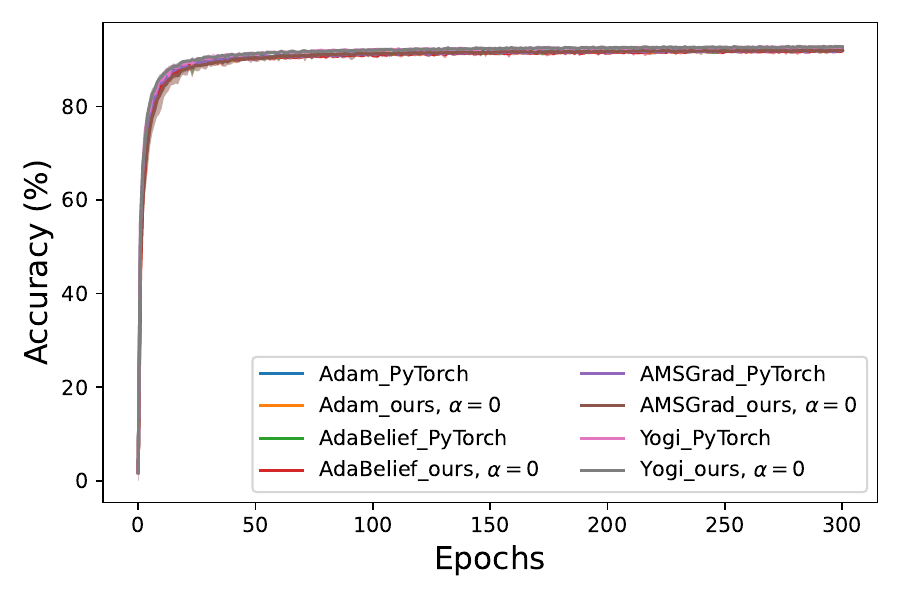}
			\label{Fig:Test4_cifar10_acc}
		\end{minipage}%
	}%
        \subfigure[Test accuracy, CIFAR-100]{
		\begin{minipage}[t]{0.48\linewidth}
			\centering
			\includegraphics[width=\linewidth, height=0.20\textheight]{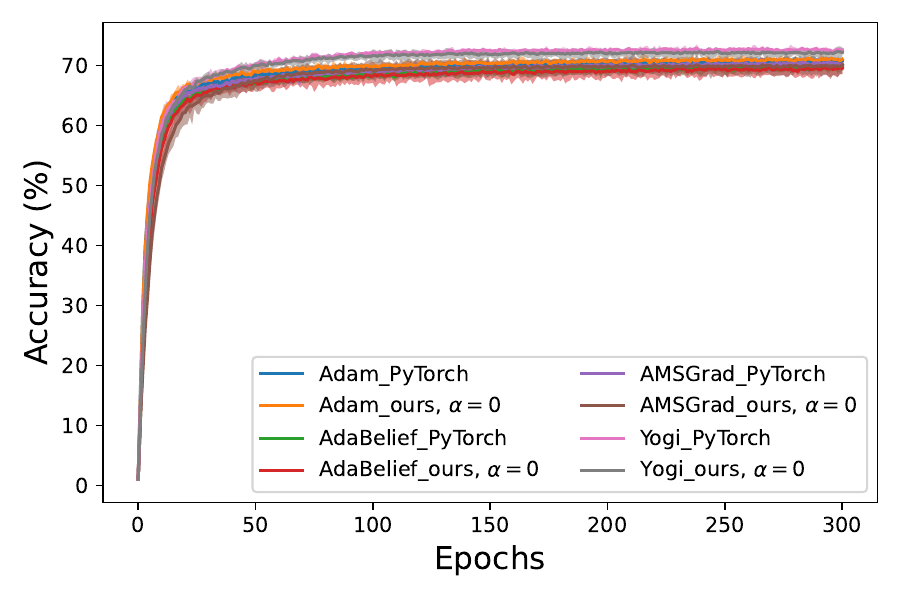}
			\label{Fig:Test4_cifar100_acc}
		\end{minipage}%
	}%
 
	\subfigure[Train loss, CIFAR-10]{
		\begin{minipage}[t]{0.48\linewidth}
			\centering
			\includegraphics[width=\linewidth, height=0.20\textheight]{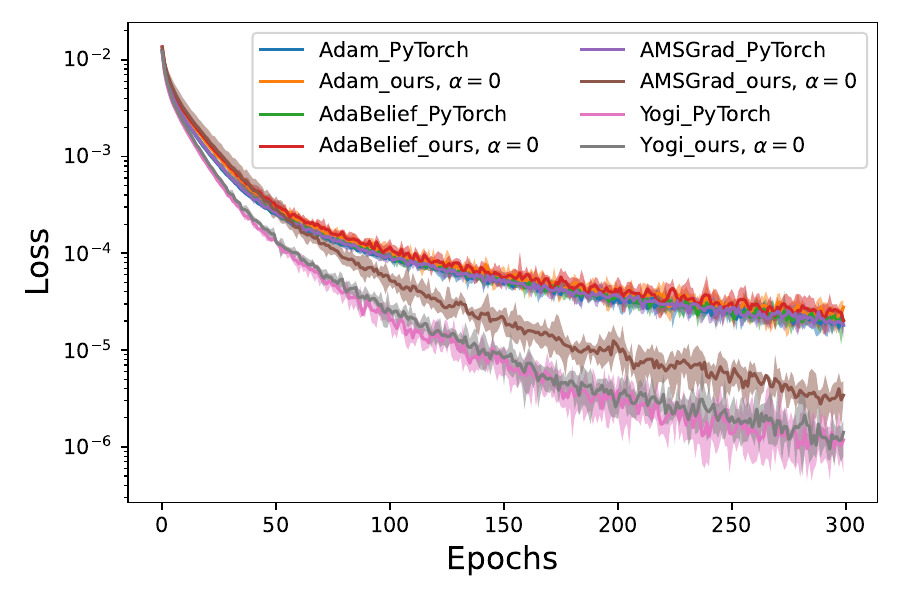}
			\label{Fig:Test4_cifar10_trainloss}
		\end{minipage}%
	}%
        \subfigure[Train loss, CIFAR-100]{
		\begin{minipage}[t]{0.48\linewidth}
			\centering
			\includegraphics[width=\linewidth, height=0.20\textheight]{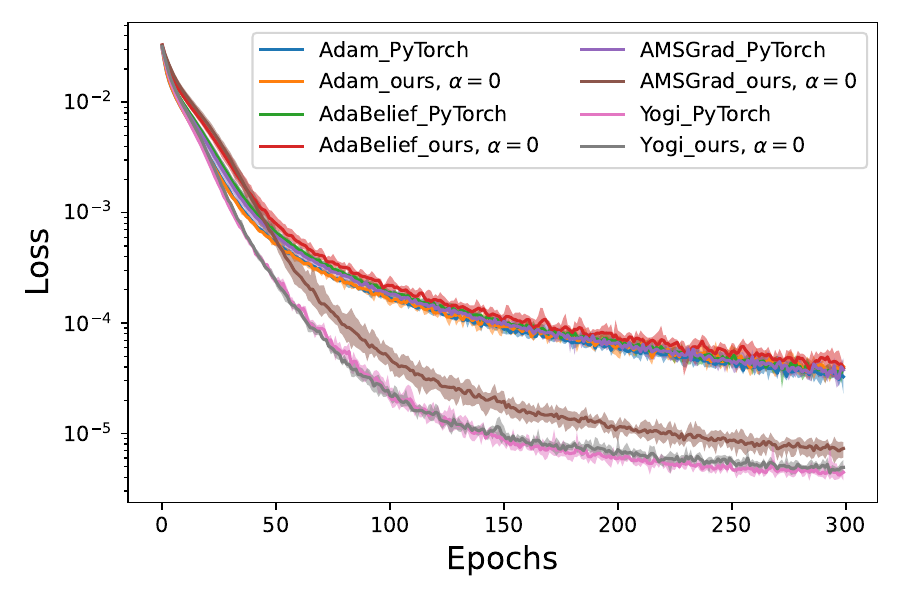}
			\label{Fig:Test4_cifar100_trainloss}
		\end{minipage}%
	}%
 
	\subfigure[Test loss, CIFAR-10]{
		\begin{minipage}[t]{0.48\linewidth}
			\centering
			\includegraphics[width=\linewidth, height=0.20\textheight]{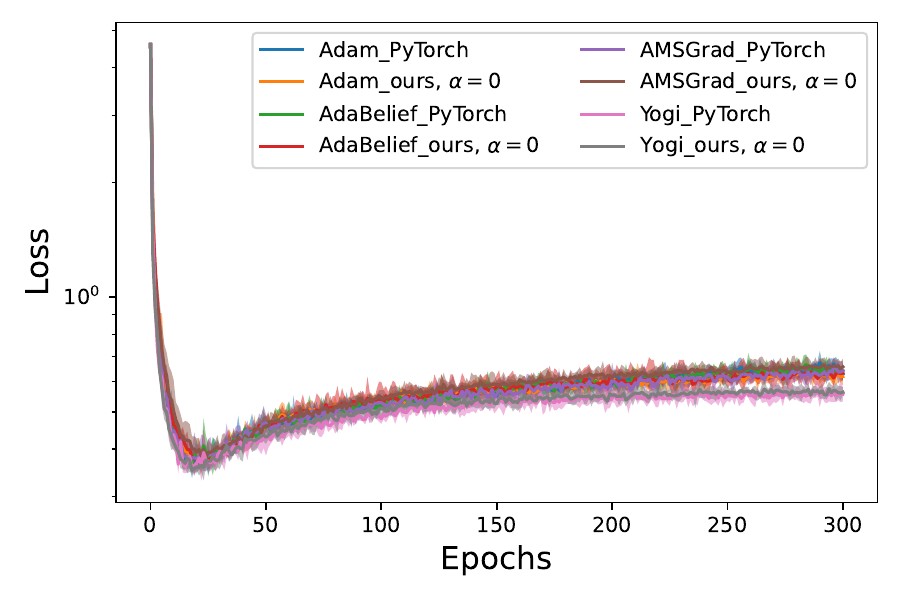}
			\label{Fig:Test4_cifar10_testloss}
		\end{minipage}%
	}%
        \subfigure[Test loss, CIFAR-100]{
		\begin{minipage}[t]{0.48\linewidth}
			\centering
			\includegraphics[width=\linewidth, height=0.20\textheight]{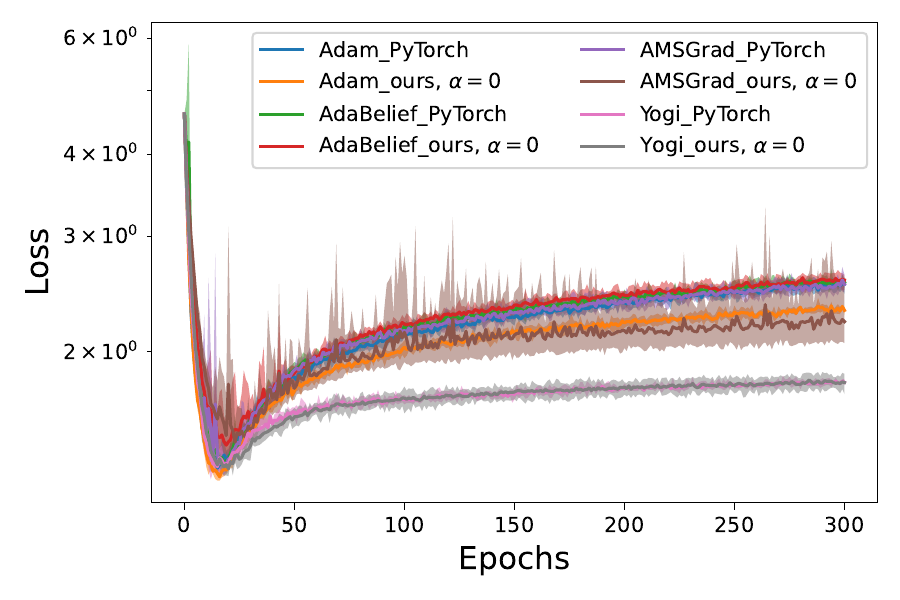}
			\label{Fig:Test4_cifar100_testloss}
		\end{minipage}%
	}%
	\caption{Test results on CIFAR data sets with ResNet50.}
	\label{Fig_Test4_Cifar}
\end{figure}
In this section, we present the supplementary results for the numerical experiments in Section 6.1. In Figure \ref{Fig_Test4_Cifar}, we present the performance of all the compared Adam-family methods, in the aspects of test accuracy, test error, and train error. In particular, the curves of all the compared Adam-family methods are plotted in a single subfigure in Figure \ref{Fig_Test4_Cifar}, for a better illustration on the performances of different Adam-family methods.

\section{Supplementary Numerical Experiments for Section 6.2}
In this section, we present numerical experiments on testing the efficiency of our proposed stochastic subgradient methods with clipping on the natural language processing (NLP) tasks. In these numerical experiments, different from the settings in Section 6.2,  we do not introduce any corruption to the training samples.

We first evaluate the performance of \eqref{Eq_SGDM} and \eqref{Eq_ADAMC} by training language-translation model \citep{vaswani2017attention} on the Multi30k data set \citep{elliott2016Multi30K}. In our numerical experiments, we choose the language translation model as the Seq2Seq network with transformer proposed by \citep{vaswani2017attention}.
Similar to the settings in Section 6.2, we set the batch size to $128$ for all test instances and select the regularization parameter $\varepsilon = 10^{-15}$ for all the Adam-family methods. Moreover, at the $k$-th epoch, we choose the stepsize as $\eta_k = \frac{\eta_0}{\sqrt{k+1}}$ for all tested algorithms. For all compared optimizers, we choose the initial stepsize $\eta_0$ and the momentum parameters $\tau_1, \tau_2$ using the same grid search method as in Section \ref{Section_numerical_adaptive}, and retain all other parameters at their default values for the optimizers in PyTorch. We run each test instance $5$ times with different random seeds. In each test instance, all compared algorithms are tested using the same random seed and initialized with the same random weights by the default initialization function in PyTorch. 

Then we evaluate the efficiency of all the compared methods in training $3$-layer long short-term memory (LSTM) models. In all the numerical experiments, we consistently train our models for $200$ epochs while employing a batch size of 128. These settings adhere to the commonly used experimental setup for training LSTM models, as demonstrated in previous works \citep{zhuang2020adabelief}.

Figure \ref{Fig_Test5_NLP} exhibits the results of our numerical experiments with error bars. Notably, although we run each compared method for $5$ times with different random seeds, the loss curves seem to be very close. As depicted in Figure \ref{Fig_Test5_NLP}, the method outlined in \eqref{Eq_SGDM} slightly outperforms the default SGD method in PyTorch. Moreover, the \ref{Eq_ADAMC} method achieves slightly lower training loss than the build-in Adam method in PyTorch, although the test loss for  \ref{Eq_ADAMC} is slightly worse than that of Adam-PyTorch. These observations illustrate that even when the training samples are finite, the stochastic subgradient methods with gradient clipping exhibit slightly better performance in the training of language models. Combined with the results in Section 6.2, our numerical experiments results illustrate the potential of our proposed stochastic subgradient methods with gradient clipping technique for solving \ref{Prob_Ori}.

\begin{figure}[tb]
	\centering
	\subfigure[Train loss on Multi30K with Seq2Seq network]{
		\begin{minipage}[t]{0.48\linewidth}
			\centering
			\includegraphics[width=\linewidth, height=0.20\textheight]{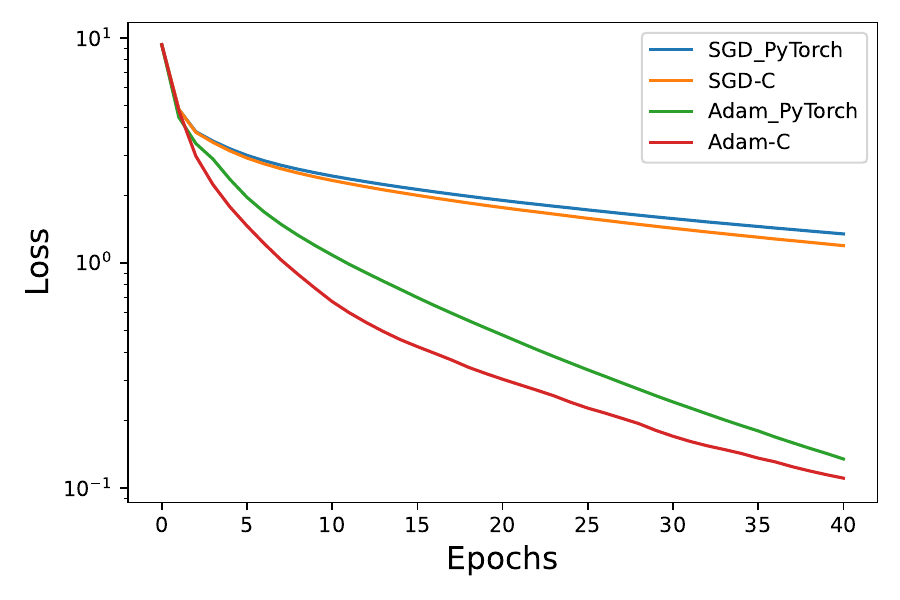}
			\label{Fig:Test5_trainloss}
		\end{minipage}%
	}%
        \subfigure[Test loss on Multi30K with Seq2Seq network]{
		\begin{minipage}[t]{0.48\linewidth}
			\centering
			\includegraphics[width=\linewidth, height=0.20\textheight]{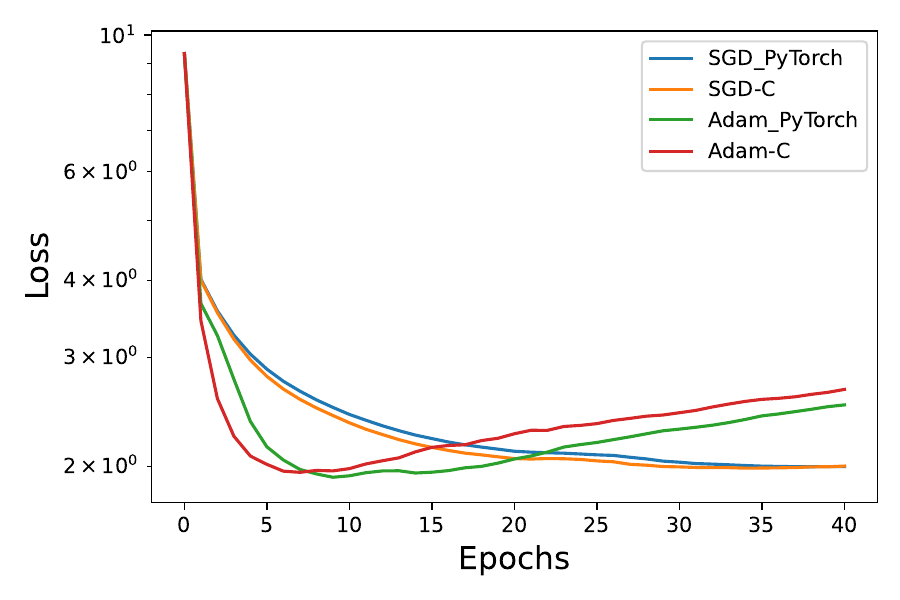}
			\label{Fig:Test5_testloss}
		\end{minipage}%
	}%

        \subfigure[Train perplexity on Penn Treebank with LSTM]{
		\begin{minipage}[t]{0.48\linewidth}
			\centering
			\includegraphics[width=\linewidth, height=0.20\textheight]{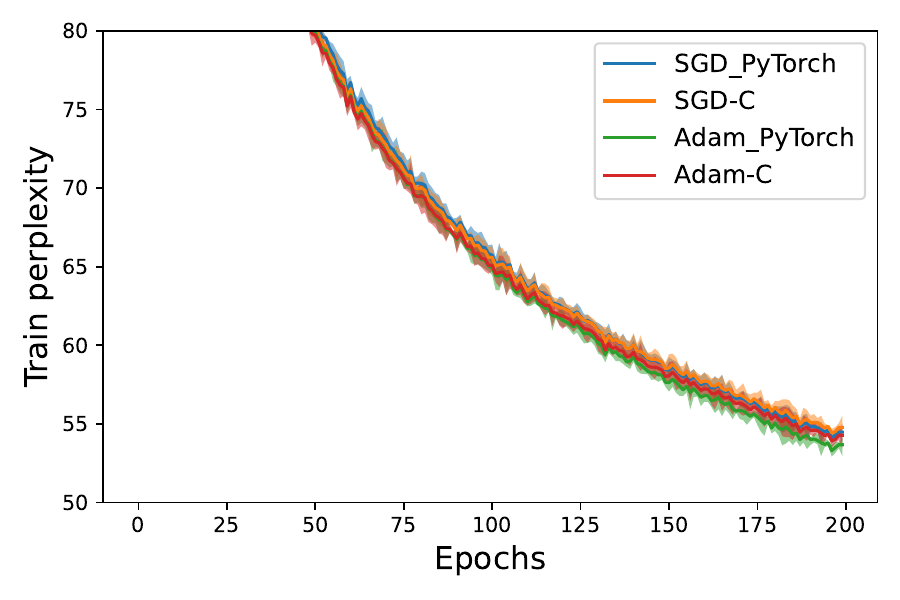}
			\label{Fig:Test5_LSTM_trainloss}
		\end{minipage}%
	}%
        \subfigure[Test perplexity on Penn Treebank with LSTM]{
		\begin{minipage}[t]{0.48\linewidth}
			\centering
			\includegraphics[width=\linewidth, height=0.20\textheight]{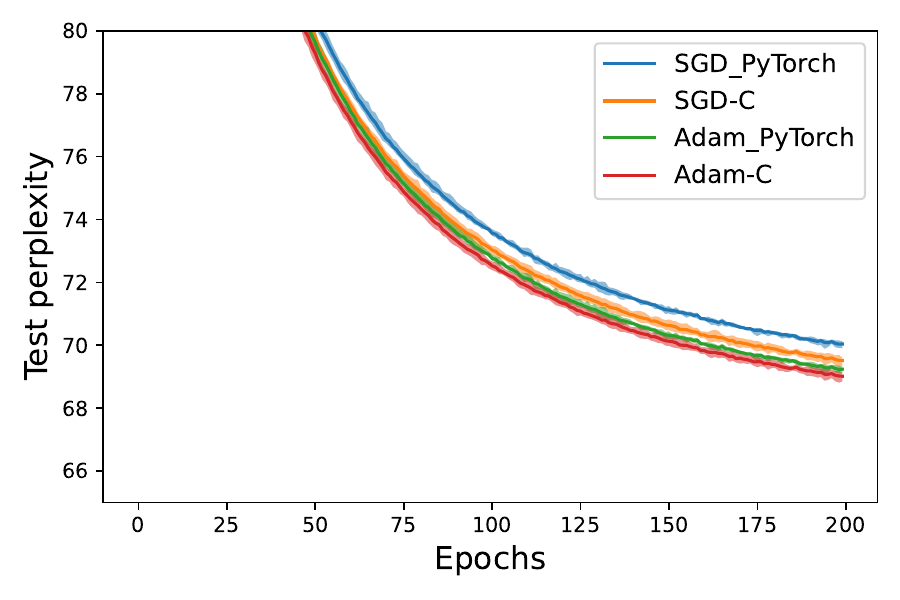}
			\label{Fig:Test5_LSTM_testloss}
		\end{minipage}%
	}%
	\caption{Numerical results on NLP tasks.}
	\label{Fig_Test5_NLP}
\end{figure}

\vskip 0.2in
\bibliography{ref}

\begin{thebibliography}{69}
\providecommand{\natexlab}[1]{#1}
\providecommand{\url}[1]{\texttt{#1}}
\expandafter\ifx\csname urlstyle\endcsname\relax
  \providecommand{\doi}[1]{doi: #1}\else
  \providecommand{\doi}{doi: \begingroup \urlstyle{rm}\Url}\fi

\bibitem[Barakat and Bianchi(2021)]{barakat2021convergence}
Anas Barakat and Pascal Bianchi.
\newblock Convergence and dynamical behavior of the {ADAM} algorithm for
  nonconvex stochastic optimization.
\newblock \emph{SIAM Journal on Optimization}, 31\penalty0 (1):\penalty0
  244--274, 2021.

\bibitem[Barakat et~al.(2021)Barakat, Bianchi, Hachem, and
  Schechtman]{barakat2021stochastic}
Anas Barakat, Pascal Bianchi, Walid Hachem, and Sholom Schechtman.
\newblock Stochastic optimization with momentum: convergence, fluctuations, and
  traps avoidance.
\newblock \emph{Electronic Journal of Statistics}, 15\penalty0 (2):\penalty0
  3892--3947, 2021.

\bibitem[Barton et~al.(2018)Barton, Khan, Stechlinski, and
  Watson]{barton2018computationally}
Paul~I Barton, Kamil~A Khan, Peter Stechlinski, and Harry~AJ Watson.
\newblock Computationally relevant generalized derivatives: theory, evaluation
  and applications.
\newblock \emph{Optimization Methods and Software}, 33\penalty0 (4-6):\penalty0
  1030--1072, 2018.

\bibitem[Bena{\"\i}m(2006)]{benaim2006dynamics}
Michel Bena{\"\i}m.
\newblock Dynamics of stochastic approximation algorithms.
\newblock In \emph{Seminaire de Probabilites XXXIII}, pages 1--68. Springer,
  2006.

\bibitem[Bena{\"\i}m et~al.(2005)Bena{\"\i}m, Hofbauer, and
  Sorin]{benaim2005stochastic}
Michel Bena{\"\i}m, Josef Hofbauer, and Sylvain Sorin.
\newblock Stochastic approximations and differential inclusions.
\newblock \emph{SIAM Journal on Control and Optimization}, 44\penalty0
  (1):\penalty0 328--348, 2005.

\bibitem[Bianchi et~al.(2022)Bianchi, Hachem, and
  Schechtman]{bianchi2022convergence}
Pascal Bianchi, Walid Hachem, and Sholom Schechtman.
\newblock Convergence of constant step stochastic gradient descent for
  non-smooth non-convex functions.
\newblock \emph{Set-Valued and Variational Analysis}, pages 1--31, 2022.

\bibitem[Bierstone and Milman(1988)]{bierstone1988semianalytic}
Edward Bierstone and Pierre~D Milman.
\newblock Semianalytic and subanalytic sets.
\newblock \emph{Publications Math{\'e}matiques de l'IH{\'E}S}, 67:\penalty0
  5--42, 1988.

\bibitem[Bolte and Pauwels(2020)]{bolte2020mathematical}
J{\'e}r{\^o}me Bolte and Edouard Pauwels.
\newblock A mathematical model for automatic differentiation in machine
  learning.
\newblock \emph{Advances in Neural Information Processing Systems},
  33:\penalty0 10809--10819, 2020.

\bibitem[Bolte and Pauwels(2021)]{bolte2021conservative}
J{\'e}r{\^o}me Bolte and Edouard Pauwels.
\newblock Conservative set valued fields, automatic differentiation, stochastic
  gradient methods and deep learning.
\newblock \emph{Mathematical Programming}, 188\penalty0 (1):\penalty0 19--51,
  2021.

\bibitem[Bolte et~al.(2007)Bolte, Daniilidis, Lewis, and
  Shiota]{bolte2007clarke}
J{\'e}r{\^o}me Bolte, Aris Daniilidis, Adrian Lewis, and Masahiro Shiota.
\newblock Clarke subgradients of stratifiable functions.
\newblock \emph{SIAM Journal on Optimization}, 18\penalty0 (2):\penalty0
  556--572, 2007.

\bibitem[Bolte et~al.(2021)Bolte, Le, Pauwels, and
  Silveti-Falls]{bolte2021nonsmooth}
J{\'e}r{\^o}me Bolte, Tam Le, Edouard Pauwels, and Tony Silveti-Falls.
\newblock Nonsmooth implicit differentiation for machine-learning and
  optimization.
\newblock \emph{Advances in Neural Information Processing Systems}, 34, 2021.

\bibitem[Bolte et~al.(2023)Bolte, Le, and Pauwels]{bolte2022subgradient}
J{\'e}r{\^o}me Bolte, Tam Le, and Edouard Pauwels.
\newblock Subgradient sampling for nonsmooth nonconvex minimization.
\newblock \emph{SIAM Journal on Optimization}, 33\penalty0 (4):\penalty0
  2542--2569, 2023.

\bibitem[Camuto et~al.(2021)Camuto, Wang, Zhu, Holmes, Gurbuzbalaban, and
  Simsekli]{camuto2021asymmetric}
Alexander Camuto, Xiaoyu Wang, Lingjiong Zhu, Chris Holmes, Mert Gurbuzbalaban,
  and Umut Simsekli.
\newblock Asymmetric heavy tails and implicit bias in gaussian noise
  injections.
\newblock In \emph{International Conference on Machine Learning}, pages
  1249--1260. PMLR, 2021.

\bibitem[Castera et~al.(2021)Castera, Bolte, F{\'e}votte, and
  Pauwels]{castera2021inertial}
Camille Castera, J{\'e}r{\^o}me Bolte, C{\'e}dric F{\'e}votte, and Edouard
  Pauwels.
\newblock An inertial {Newton} algorithm for deep learning.
\newblock \emph{Journal of Machine Learning Research}, 22\penalty0
  (134):\penalty0 1--31, 2021.

\bibitem[Chen et~al.(2022)Chen, Shen, Zou, and Liu]{chen2022towards}
Congliang Chen, Li~Shen, Fangyu Zou, and Wei Liu.
\newblock Towards practical {Adam}: Non-convexity, convergence theory, and
  mini-batch acceleration.
\newblock \emph{Journal of Machine Learning Research}, 23:\penalty0 1--47,
  2022.

\bibitem[Clarke(1990)]{clarke1990optimization}
Frank~H Clarke.
\newblock \emph{Optimization and Nonsmooth Analysis}, volume~5.
\newblock SIAM, 1990.

\bibitem[Da~Silva and Gazeau(2020)]{da2020general}
Andr{\'e}~Belotto Da~Silva and Maxime Gazeau.
\newblock A general system of differential equations to model first-order
  adaptive algorithms.
\newblock \emph{The Journal of Machine Learning Research}, 21\penalty0
  (1):\penalty0 5072--5113, 2020.

\bibitem[Davis et~al.(2020)Davis, Drusvyatskiy, Kakade, and
  Lee]{davis2020stochastic}
Damek Davis, Dmitriy Drusvyatskiy, Sham Kakade, and Jason~D Lee.
\newblock Stochastic subgradient method converges on tame functions.
\newblock \emph{Foundations of Computational Mathematics}, 20\penalty0
  (1):\penalty0 119--154, 2020.

\bibitem[De et~al.(2018)De, Mukherjee, and Ullah]{de2018convergence}
Soham De, Anirbit Mukherjee, and Enayat Ullah.
\newblock Convergence guarantees for {RMSProp} and {ADAM} in non-convex
  optimization and an empirical comparison to nesterov acceleration.
\newblock \emph{arXiv preprint arXiv:1807.06766}, 2018.

\bibitem[Dozat(2016)]{dozat2016incorporating}
Timothy Dozat.
\newblock Incorporating {Nesterov Momentum into {Adam}}.
\newblock In \emph{Proceedings of the 4th International Conference on Learning
  Representations}, pages 1--4, 2016.

\bibitem[Elesedy and Hutter(2023)]{elesedy2023u}
Bryn Elesedy and Marcus Hutter.
\newblock U-clip: On-average unbiased stochastic gradient clipping.
\newblock \emph{arXiv preprint arXiv:2302.02971}, 2023.

\bibitem[{Elliott} et~al.(2016){Elliott}, {Frank}, {Sima'an}, and
  {Specia}]{elliott2016Multi30K}
D.~{Elliott}, S.~{Frank}, K.~{Sima'an}, and L.~{Specia}.
\newblock Multi30k: Multilingual english-german image descriptions.
\newblock In \emph{Proceedings of the 5th Workshop on Vision and Language},
  pages 70--74, 2016.

\bibitem[Fu et~al.(2020)Fu, Zhang, Liu, and Zhang]{fu2020drts}
Qiankun Fu, Yue Zhang, Jiangming Liu, and Meishan Zhang.
\newblock {DRTS} parsing with structure-aware encoding and decoding.
\newblock In \emph{Proceedings of the 58th Annual Meeting of the Association
  for Computational Linguistics}, pages 6818--6828, Online, July 2020.
  Association for Computational Linguistics.

\bibitem[Fu et~al.(2021)Fu, Song, Du, and Zhang]{fu2021end}
Qiankun Fu, Linfeng Song, Wenyu Du, and Yue Zhang.
\newblock End-to-end {AMR} coreference resolution.
\newblock In \emph{Proceedings of the 59th Annual Meeting of the Association
  for Computational Linguistics and the 11th International Joint Conference on
  Natural Language Processing (Volume 1: Long Papers)}, pages 4204--4214,
  Online, August 2021. Association for Computational Linguistics.

\bibitem[Gadat and Gavra(2022)]{gadat2022asymptotic}
S{\'e}bastien Gadat and Ioana Gavra.
\newblock Asymptotic study of stochastic adaptive algorithms in non-convex
  landscape.
\newblock \emph{The Journal of Machine Learning Research}, 23\penalty0
  (1):\penalty0 10357--10410, 2022.

\bibitem[Gorbunov et~al.(2020)Gorbunov, Danilova, and
  Gasnikov]{gorbunov2020stochastic}
Eduard Gorbunov, Marina Danilova, and Alexander Gasnikov.
\newblock Stochastic optimization with heavy-tailed noise via accelerated
  gradient clipping.
\newblock \emph{Advances in Neural Information Processing Systems},
  33:\penalty0 15042--15053, 2020.

\bibitem[Grandvalet et~al.(1997)Grandvalet, Canu, and
  Boucheron]{grandvalet1997noise}
Yves Grandvalet, St{\'e}phane Canu, and St{\'e}phane Boucheron.
\newblock Noise injection: Theoretical prospects.
\newblock \emph{Neural Computation}, 9\penalty0 (5):\penalty0 1093--1108, 1997.

\bibitem[Guo et~al.(2021)Guo, Xu, Yin, Jin, and Yang]{guo2021novel}
Zhishuai Guo, Yi~Xu, Wotao Yin, Rong Jin, and Tianbao Yang.
\newblock A novel convergence analysis for algorithms of the {Adam} family.
\newblock \emph{NeurIPS OPT Workshop}, 2021.

\bibitem[He et~al.(2016)He, Zhang, Ren, and Sun]{he2016deep}
Kaiming He, Xiangyu Zhang, Shaoqing Ren, and Jian Sun.
\newblock Deep residual learning for image recognition.
\newblock In \emph{Proceedings of the IEEE Conference on Computer Vision and
  Pattern Recognition (CVPR)}, pages 770--778, 2016.

\bibitem[Hu et~al.(2023)Hu, Xiao, Liu, and Toh]{hu2023improved}
Xiaoyin Hu, Nachuan Xiao, Xin Liu, and Kim-Chuan Toh.
\newblock An improved unconstrained approach for bilevel optimization.
\newblock \emph{SIAM Journal on Optimization}, 33\penalty0 (4):\penalty0
  2801--2829, 2023.

\bibitem[Josz and Lai(2023)]{josz2023global}
C{\'e}dric Josz and Lexiao Lai.
\newblock Global stability of first-order methods for coercive tame functions.
\newblock \emph{Mathematical Programming}, pages 1--26, 2023.

\bibitem[Kelley(2017)]{kelley2017general}
John~L Kelley.
\newblock \emph{General Topology}.
\newblock Courier Dover Publications, 2017.

\bibitem[Kingma and Ba(2015)]{kingma2014adam}
Diederik~P Kingma and Jimmy Ba.
\newblock Adam: A method for stochastic optimization.
\newblock \emph{In Proceedings of the 3rd International Conference for Learning
  Representations}, 2015.

\bibitem[Krizhevsky et~al.(Toronto, ON, Canada, 2009)Krizhevsky, Hinton,
  et~al.]{krizhevsky2009learning}
Alex Krizhevsky, Geoffrey Hinton, et~al.
\newblock Learning multiple layers of features from tiny images.
\newblock \emph{Technical report}, Toronto, ON, Canada, 2009.

\bibitem[Le(2023)]{le2023nonsmooth}
Tam Le.
\newblock Nonsmooth nonconvex stochastic heavy ball.
\newblock \emph{arXiv preprint arXiv:2304.13328}, 2023.

\bibitem[LeCun(1998)]{lecun1998mnist}
Yann LeCun.
\newblock The mnist database of handwritten digits, 1998.
\newblock URL \url{http://yann.lecun.com/exdb/mnist/}.

\bibitem[LeCun et~al.(1998)LeCun, Bottou, Bengio, and
  Haffner]{lecun1998gradient}
Yann LeCun, L{\'e}on Bottou, Yoshua Bengio, and Patrick Haffner.
\newblock Gradient-based learning applied to document recognition.
\newblock \emph{Proceedings of the IEEE}, 86\penalty0 (11):\penalty0
  2278--2324, 1998.

\bibitem[Li and Milzarek(2022)]{li2022unified}
Xiao Li and Andre Milzarek.
\newblock A unified convergence theorem for stochastic optimization methods.
\newblock \emph{Advances in Neural Information Processing Systems},
  35:\penalty0 33107--33119, 2022.

\bibitem[Loizou and Richt{\'a}rik(2017)]{loizou2017linearly}
Nicolas Loizou and Peter Richt{\'a}rik.
\newblock Linearly convergent stochastic heavy ball method for minimizing
  generalization error.
\newblock \emph{NIPS-Workshop on Optimization for Machine Learning}, 2017.

\bibitem[Maaten et~al.(2013)Maaten, Chen, Tyree, and
  Weinberger]{maaten2013learning}
Laurens Maaten, Minmin Chen, Stephen Tyree, and Kilian Weinberger.
\newblock Learning with marginalized corrupted features.
\newblock In \emph{International Conference on Machine Learning}, pages
  410--418. PMLR, 2013.

\bibitem[Mahoney and Martin(2019)]{mahoney2019traditional}
Michael Mahoney and Charles Martin.
\newblock Traditional and heavy tailed self regularization in neural network
  models.
\newblock In \emph{International Conference on Machine Learning}, pages
  4284--4293. PMLR, 2019.

\bibitem[Mai and Johansson(2021)]{mai2021stability}
Vien~V Mai and Mikael Johansson.
\newblock Stability and convergence of stochastic gradient clipping: Beyond
  lipschitz continuity and smoothness.
\newblock In \emph{International Conference on Machine Learning}, pages
  7325--7335. PMLR, 2021.

\bibitem[Ming et~al.(2018)Ming, Chen, Cao, Forde, Ngo, and Chua]{ming2018food}
Zhao-Yan Ming, Jingjing Chen, Yu~Cao, Ciar{\'a}n Forde, Chong-Wah Ngo, and
  Tat~Seng Chua.
\newblock Food photo recognition for dietary tracking: System and experiment.
\newblock In \emph{MultiMedia Modeling: 24th International Conference, MMM
  2018, Bangkok, Thailand, February 5-7, 2018, Proceedings, Part II 24}, pages
  129--141. Springer, 2018.

\bibitem[Nesterov(2005)]{nesterov2005lexicographic}
Yu~Nesterov.
\newblock Lexicographic differentiation of nonsmooth functions.
\newblock \emph{Mathematical Programming}, 104\penalty0 (2):\penalty0 669--700,
  2005.

\bibitem[Nesterov(2003)]{nesterov2003introductory}
Yurii Nesterov.
\newblock \emph{Introductory Lectures on Convex Optimization: A Basic Course},
  volume~87.
\newblock Springer Science \& Business Media, 2003.

\bibitem[Pan and Li(2023)]{pan2023toward}
Yan Pan and Yuanzhi Li.
\newblock Toward understanding why {Adam} converges faster than sgd for
  transformers.
\newblock \emph{arXiv preprint arXiv:2306.00204}, 2023.

\bibitem[Pascanu et~al.(2012)Pascanu, Mikolov, and
  Bengio]{pascanu2012understanding}
Razvan Pascanu, Tomas Mikolov, and Yoshua Bengio.
\newblock Understanding the exploding gradient problem.
\newblock \emph{CoRR, abs/1211.5063}, 2\penalty0 (417):\penalty0 1, 2012.

\bibitem[Qian et~al.(2021)Qian, Wu, Zhuang, Wang, and
  Xiao]{qian2021understanding}
Jiang Qian, Yuren Wu, Bojin Zhuang, Shaojun Wang, and Jing Xiao.
\newblock Understanding gradient clipping in incremental gradient methods.
\newblock In \emph{International Conference on Artificial Intelligence and
  Statistics}, pages 1504--1512. PMLR, 2021.

\bibitem[Reddi et~al.(2018)Reddi, Kale, and Kumar]{reddi2019convergence}
Sashank~J Reddi, Satyen Kale, and Sanjiv Kumar.
\newblock On the convergence of {Adam} and beyond.
\newblock \emph{In 6th International Conference on Learning Representations
  (ICLR)}, 2018.

\bibitem[Reisizadeh et~al.(2023)Reisizadeh, Li, Das, and
  Jadbabaie]{reisizadeh2023variance}
Amirhossein Reisizadeh, Haochuan Li, Subhro Das, and Ali Jadbabaie.
\newblock Variance-reduced clipping for non-convex optimization.
\newblock \emph{arXiv preprint arXiv:2303.00883}, 2023.

\bibitem[Robbins and Siegmund(1971)]{robbins1971convergence}
Herbert Robbins and David Siegmund.
\newblock A convergence theorem for nonnegative almost supermartingales and
  some applications.
\newblock In \emph{Optimizing Methods in Statistics}, pages 233--257. Elsevier,
  1971.

\bibitem[Ruszczy{\'n}ski(2020)]{ruszczynski2020convergence}
Andrzej Ruszczy{\'n}ski.
\newblock Convergence of a stochastic subgradient method with averaging for
  nonsmooth nonconvex constrained optimization.
\newblock \emph{Optimization Letters}, 14\penalty0 (7):\penalty0 1615--1625,
  2020.

\bibitem[Shapiro and Xu(2007)]{shapiro2007uniform}
Alexander Shapiro and Huifu Xu.
\newblock Uniform laws of large numbers for set-valued mappings and
  subdifferentials of random functions.
\newblock \emph{Journal of Mathematical Analysis and Applications},
  325\penalty0 (2):\penalty0 1390--1399, 2007.

\bibitem[Shi et~al.(2021)Shi, Li, Hong, and Sun]{shi2021rmsprop}
Naichen Shi, Dawei Li, Mingyi Hong, and Ruoyu Sun.
\newblock Rmsprop converges with proper hyperparameter.
\newblock In \emph{International Conference on Learning Representation}, 2021.

\bibitem[Simsekli et~al.(2019)Simsekli, Sagun, and
  Gurbuzbalaban]{simsekli2019tail}
Umut Simsekli, Levent Sagun, and Mert Gurbuzbalaban.
\newblock A tail-index analysis of stochastic gradient noise in deep neural
  networks.
\newblock In \emph{International Conference on Machine Learning}, pages
  5827--5837. PMLR, 2019.

\bibitem[Simsekli et~al.(2020)Simsekli, Zhu, Teh, and
  Gurbuzbalaban]{simsekli2020fractional}
Umut Simsekli, Lingjiong Zhu, Yee~Whye Teh, and Mert Gurbuzbalaban.
\newblock Fractional underdamped langevin dynamics: Retargeting {SGD} with
  momentum under heavy-tailed gradient noise.
\newblock In \emph{International Conference on Machine Learning}, pages
  8970--8980. PMLR, 2020.

\bibitem[Van~den Dries and Miller(1996)]{van1996geometric}
Lou Van~den Dries and Chris Miller.
\newblock Geometric categories and o-minimal structures.
\newblock \emph{Duke Mathematical Journal}, 84\penalty0 (2):\penalty0 497--540,
  1996.

\bibitem[Vaswani et~al.(2017)Vaswani, Shazeer, Parmar, Uszkoreit, Jones, Gomez,
  Kaiser, and Polosukhin]{vaswani2017attention}
Ashish Vaswani, Noam Shazeer, Niki Parmar, Jakob Uszkoreit, Llion Jones,
  Aidan~N Gomez, {\L}ukasz Kaiser, and Illia Polosukhin.
\newblock Attention is all you need.
\newblock \emph{Advances in Neural Information Processing Systems}, 30, 2017.

\bibitem[Wan et~al.(2023)Wan, Zaidi, and Simsekli]{wan2023implicit}
Yijun Wan, Abdellatif Zaidi, and Umut Simsekli.
\newblock Implicit compressibility of overparametrized neural networks trained
  with heavy-tailed {SGD}.
\newblock \emph{arXiv preprint arXiv:2306.08125}, 2023.

\bibitem[Wang et~al.(2022)Wang, Zhang, Zhang, Meng, Ma, Liu, and
  Chen]{wang2022provable}
Bohan Wang, Yushun Zhang, Huishuai Zhang, Qi~Meng, Zhi-Ming Ma, Tie-Yan Liu,
  and Wei Chen.
\newblock Provable adaptivity in {Adam}.
\newblock \emph{arXiv preprint arXiv:2208.09900}, 2022.

\bibitem[Wang et~al.(2023)Wang, Zhou, Chen, Hu, Wu, Jiang, Xu, and
  Qian]{wang2023chromosome}
Jun Wang, Chengfeng Zhou, Songchang Chen, Jianwu Hu, Minghui Wu, Xudong Jiang,
  Chenming Xu, and Dahong Qian.
\newblock Chromosome detection in metaphase cell images using morphological
  priors.
\newblock \emph{IEEE Journal of Biomedical and Health Informatics}, 2023.

\bibitem[Wilkie(1996)]{wilkie1996model}
Alex~J Wilkie.
\newblock Model completeness results for expansions of the ordered field of
  real numbers by restricted pfaffian functions and the exponential function.
\newblock \emph{Journal of the American Mathematical Society}, 9\penalty0
  (4):\penalty0 1051--1094, 1996.

\bibitem[Xiao et~al.(2023)Xiao, Hu, and Toh]{xiao2023convergence}
Nachuan Xiao, Xiaoyin Hu, and Kim-Chuan Toh.
\newblock Convergence guarantees for stochastic subgradient methods in
  nonsmooth nonconvex optimization.
\newblock \emph{arXiv preprint arXiv:2307.10053}, 2023.

\bibitem[Zaheer et~al.(2018)Zaheer, Reddi, Sachan, Kale, and
  Kumar]{zaheer2018adaptive}
Manzil Zaheer, Sashank Reddi, Devendra Sachan, Satyen Kale, and Sanjiv Kumar.
\newblock Adaptive methods for nonconvex optimization.
\newblock \emph{{A}dvances in {N}eural {I}nformation {P}rocessing {S}ystems},
  31, 2018.

\bibitem[Zhang et~al.(2020{\natexlab{a}})Zhang, He, Sra, and
  Jadbabaie]{zhang2019gradient}
Jingzhao Zhang, Tianxing He, Suvrit Sra, and Ali Jadbabaie.
\newblock Why gradient clipping accelerates training: A theoretical
  justification for adaptivity.
\newblock \emph{In 8th International Conference on Learning Representations
  (ICLR)}, 2020{\natexlab{a}}.

\bibitem[Zhang et~al.(2020{\natexlab{b}})Zhang, Karimireddy, Veit, Kim, Reddi,
  Kumar, and Sra]{zhang2020adaptive}
Jingzhao Zhang, Sai~Praneeth Karimireddy, Andreas Veit, Seungyeon Kim, Sashank
  Reddi, Sanjiv Kumar, and Suvrit Sra.
\newblock Why are adaptive methods good for attention models?
\newblock \emph{Advances in Neural Information Processing Systems},
  33:\penalty0 15383--15393, 2020{\natexlab{b}}.

\bibitem[Zhang et~al.(2022)Zhang, Chen, Shi, Sun, and Luo]{zhang2022adam}
Yushun Zhang, Congliang Chen, Naichen Shi, Ruoyu Sun, and Zhi-Quan Luo.
\newblock Adam can converge without any modification on update rules.
\newblock In \emph{Advances in Neural Information Processing Systems}, 2022.

\bibitem[Zhuang et~al.(2020)Zhuang, Tang, Ding, Tatikonda, Dvornek,
  Papademetris, and Duncan]{zhuang2020adabelief}
Juntang Zhuang, Tommy Tang, Yifan Ding, Sekhar~C Tatikonda, Nicha Dvornek,
  Xenophon Papademetris, and James Duncan.
\newblock Adabelief optimizer: Adapting stepsizes by the belief in observed
  gradients.
\newblock \emph{Advances in Neural Information Processing Systems},
  33:\penalty0 18795--18806, 2020.

\bibitem[Zou et~al.(2019)Zou, Shen, Jie, Zhang, and Liu]{zou2019sufficient}
Fangyu Zou, Li~Shen, Zequn Jie, Weizhong Zhang, and Wei Liu.
\newblock A sufficient condition for convergences of {Adam} and {RMSProp}.
\newblock In \emph{Proceedings of the IEEE/CVF Conference on Computer Vision
  and Pattern Recognition (CVPR)}, pages 11127--11135, 2019.

\end{thebibliography}

\end{document}